\documentclass[11pt]{article}

\usepackage[margin=1.1in]{geometry}

\newcommand{\red}{\color{red}}

\usepackage[english]{babel}
\usepackage{xspace}
\usepackage{algorithm,algorithmicx}
\usepackage{tcolorbox}
\usepackage{tikz}

\usepackage{framed}
\usepackage{algpseudocode}
\usepackage{amsthm,nccmath}
\usepackage{amsmath,bm}
\usepackage[normalem]{ulem}
\usepackage{bigints}
\usepackage{amssymb}
\usepackage{caption}
\usepackage{graphicx}
\usepackage{cite}
\usepackage{enumerate}
\usepackage{multirow}
\usepackage{hyperref}
\usepackage{booktabs}
\usepackage{hyperref}
\usepackage{url}
\usepackage{color, colortbl}
\usepackage{graphicx}
\usepackage{amsmath,amssymb, bm}
\usepackage{algorithm}
\usepackage{algpseudocode}
\usepackage{xcolor}
\usepackage{flushend}
\usepackage{subfigure}
\usepackage[none]{hyphenat}
\usepackage{enumitem}
\usepackage[bottom]{footmisc}
\usepackage{hyperref}
\hypersetup{
	colorlinks   = true, %Colours links instead of ugly boxes
	urlcolor     = blue, %Colour for external hyperlinks
	linkcolor    = blue, %Colour of internal links
	citecolor    = blue   %Colour of citations
}

\usepackage{float}
\usepackage{epstopdf}
\usepackage{footnote}
\makesavenoteenv{tabular}
\makesavenoteenv{table}
\usepackage{bbm}
\usepackage{amsthm}
\usepackage[makeroom]{cancel}
\usepackage{stmaryrd}
\usepackage{amsmath,bm}
\usepackage{amssymb}
\usepackage{mathtools, cuted}
\usepackage{kantlipsum,setspace}

\theoremstyle{plain}
\newtheorem{theorem}{Theorem}[section]
\newtheorem{lem}[theorem]{Lemma}
\newtheorem*{lem*}{Lemma}

\newtheorem*{cor*}{Corollary}

\theoremstyle{definition}
\newtheorem{defn}{Definition}[section]
\newtheorem*{defn*}{Definition}
\newtheorem{assump}{Assumption}

\theoremstyle{remark}

\newtheorem*{rem*}{Remark}

\DeclareMathOperator*{\argmin}{arg\,min}

\newcommand{\aname}{{\textsf{SUSTAIN}}}

\def\du{d_{\sf up}}
\def\dl{d_{\sf lo}}

\title{{\bf A Near-Optimal  Algorithm for Stochastic Bilevel Optimization via Double-Momentum}}
\date{} 
\author{\large Prashant Khanduri$^\dagger$, Siliang Zeng$^\dagger$, Mingyi Hong$^\dagger$, Hoi-To Wai$^\ast$, \\
	\large Zhaoran Wang$^\ddagger$, and Zhuoran Yang$^\diamond$  \\[.5cm]
	\small $^{\dagger}$Department  of Electrical and Computer Engineering, \\
	\small University of Minnesota, MN, USA\\
	\small $^{\ast}$Department of Systems Engineering and Engineering Management,\\
	\small The Chinese University of Hong Kong, Hong Kong\\
	\small $^{\ddagger}$Department of Industrial Engineering and Management Sciences,\\
	\small Northwestern University, IL, USA\\
	\small $^{\diamond}$ Department of Operations Research and Financial Engineering,\\
	\small Princeton University, NJ, USA\\
	\small Email: \texttt{\{khand095, zeng0176, mhong\}@umn.edu}, \texttt{htwai@se.cuhk.edu.hk},\\ 
	\small \texttt{zhaoranwang@gmail.com}, \texttt{zy6@princeton.edu}} 

\begin{document}
	
	\maketitle
	
	\begin{abstract}
This paper proposes a new algorithm -- the  \underline{S}ingle-timescale Do\underline{u}ble-momentum \underline{St}ochastic \underline{A}pprox\underline{i}matio\underline{n} (\aname) -- for tackling stochastic unconstrained bilevel optimization problems. We focus on bilevel problems where the lower level subproblem is  strongly-convex and the upper level objective function is smooth. Unlike prior works which rely on \emph{two-timescale} or \emph{double loop} techniques, 
% that track the optimal solution to the lower level subproblem, 
we design a stochastic momentum-assisted gradient estimator for both the upper and lower level updates. The latter allows us to control the error in the stochastic gradient updates due to inaccurate solution to both subproblems. If the upper objective function is smooth but possibly non-convex,
% (resp. strongly-convex),
we show that {\aname}~requires $\mathcal{O}(\epsilon^{-3/2})$ 
% (resp. $\mathcal{O}(\epsilon^{-1})$) 
iterations (each using ${\cal O}(1)$ samples) to find an $\epsilon$-stationary 
% (resp. $\epsilon$-optimal) 
solution. 
% Further, each iteration of the algorithm depends quadratically on the dimension of the inner problem. 
The $\epsilon$-stationary 
% (resp. $\epsilon$-optimal) 
solution is defined as the point whose squared norm  of the gradient of the outer function 
% (resp. difference of outer function from optimal objective value) 
is less than or equal to $\epsilon$.  The total  number of stochastic gradient samples required for  the upper and lower level objective functions matches the best-known complexity for single-level stochastic gradient algorithms.
We also analyze the case when the upper level objective function is strongly-convex.
%{\blue while the dimension dependency }.
% 	{\red[do we mention log?]}
		%We validate our theoretical results by showing the efficiency of the {\aname} algorithm on hyperparameter optimization and data hyper-cleaning problems.  
	\end{abstract}
	
%	\keywords{Bilevel optimization, non-convex optimization, stochastic gradient descent, momentum methods, variance reduction, meta-learning, hyperparameter optimization}
	
\section{Introduction}
\label{sec: Intro}
Many learning and inference problems take a ``hierarchical" form, wherein the optimal solution of one problem affects the objective function of others \cite{Migdalas_Book_MultilevelOpt_2013}. Bilevel optimization  is often used to model problems of this kind with two levels of hierarchy \cite{Migdalas_Book_MultilevelOpt_2013, Dempe_Book_2002}, where the variables of an {\it upper level} problem depend on the optimizer of certain {\it lower level} problem. In this work, we consider unconstrained bilevel optimization problems of the form:
\begin{align}
	\begin{split}
	\label{Eq: BilevelProb_Deterministic}
	\textstyle \min_{x \in \mathbb{R}^{\du}} \hspace{-.2cm} & \quad  \ell(x) = f(x , y^\ast(x)) \coloneqq \mathbb{E}_{\xi}[f(x,y^\ast(x) ; \xi) ] \\
	\text{s.t.}& \textstyle \quad y^\ast(x) =  \argmin_{y \in \mathbb{R}^{\dl}} \big\{ g(x , y) \coloneqq \mathbb{E}_{\zeta}[g(x,y; \zeta) ] \big\},
	\end{split}
\end{align}
where $f,g : \mathbb{R}^{\du} \times \mathbb{R}^{\dl} \to \mathbb{R}$ with $x \in \mathbb{R}^{\du}$ and $y \in \mathbb{R}^{\dl}$;  $f(x , y; \xi)$ with $\xi \sim \pi_f$  (resp. $g(x,y; \zeta)$ with $\zeta \sim \pi_g$) represents a stochastic sample of the upper level objective (resp. lower level objective). Note here that the {\it upper level objective} $f$ depends on the minimizer of the {\it lower level objective} $g$, and we refer to $\ell(x)$ as the {\it outer function}. Throughout this paper, $g(x,y)$ is assumed to be  strongly-convex in $y$, which implies that  $\ell(x)$ is smooth but possibly non-convex.
	
The applications of \eqref{Eq: BilevelProb_Deterministic} include many machine learning problems that have a hierarchical structure. Examples are meta learning \cite{Franceschi_ICML_2018, Rajeswaran_MetaLearning_NIPS_2019}, data hyper-cleaning \cite{Shaban_TruncatedBackProp_2019}, hyper-parameter optimization \cite{Franceschi_ICML_2017, Franceschi_ICML_2018, Pedregosa_ICML_2016}, and reinforcement learning \cite{Konda_AC_NIPS_2000}, etc.. To better contextualize our study, below we describe examples on meta-learning problem and data hyper-cleaning problem:  

\emph{Example 1: Meta learning.}~~The meta learning problem aims to learn task specific parameters that generalize to a diverse set of tasks \cite{raghu2019rapid}. Suppose we have $M$ tasks $\{\mathcal{T}_i, i = 1, \ldots, M\}$ and each task has a corresponding loss function $L(x, y_i ; \xi_i)$ with $\xi_i$ representing a data sample for task $\mathcal{T}_i$, $x \in \mathbb{R}^{\du}$ the model parameters shared among tasks, and $y_i \in \mathbb{R}^{d_{\sf lo}^i}$ the task specific parameters. The goal of meta learning is then to solve the following problem:
\begin{align}
    \label{Eq: Meta_Leaning}
   & \textstyle \min_{x \in \mathbb{R}^{\du} }~ \big\{ L_{\sf ts}(x , \bar{y}^\ast(x))  \coloneqq \frac{1}{M} \sum_{i = 1}^M \mathbb{E}_{\xi_i \sim \mathcal{D}_i}[L (x , \bar{y}_i^\ast(x) ; \xi_i)] \big\} \nonumber\\
    & \textstyle ~\text{s.t.}~~ \bar{y}^\ast(x) \in \argmin_{\bar{y} \in \mathbb{R}^{\sum_{i=1}^M d_{\sf lo}^i} } L_{\sf tr}(x, \bar{y}) \coloneqq
    % = \argmin_{(y_1, \ldots, y_M)} 
    \frac{1}{M} \sum_{i = 1}^M \big( \mathbb{E}_{\zeta_i \sim \mathcal{S}_i} [ L(x , \bar{y}_i ; \zeta_i) ] + \mathcal{R}(y_i) \big),
\end{align}
where $\mathcal{R}(\cdot)$ is a strongly convex regularizer while $\mathcal{S}_i$ and $\mathcal{D}_i$ are the training and testing datasets for task $\mathcal{T}_i$. Compared to the number of tasks, the dataset sizes are usually small for meta-learning problems, so the stochasticity in tackling \eqref{Eq: Meta_Leaning} results from the fact that at each iteration we can only sample a subset  $m$ out of $M$ tasks. Note that this problem is a special case of \eqref{Eq: BilevelProb_Deterministic}. \hfill $\square$

\emph{Example 2: Data hyper-cleaning.}~~
The data hyper-cleaning is a hyperparameter optimization problem that aims to train a classifier model with a dataset of randomly corrupted labels \cite{Shaban_TruncatedBackProp_2019}. The optimization problem is formulated below: 
\begin{align}\label{eq:clean}
\textstyle \min_{x \in \mathbb{R}^{\du}}  &~~ \textstyle \ell(x) := \sum_{i \in \mathcal{D}_{\text{val}}} L(a_i^\top y^*(x), b_i) \\
\text{s.t.} &~~ \textstyle y^*(x) = \argmin_{y \in \mathbb{R}^{\dl}}  \big\{ c\| y \|^2 + \sum_{i\in \mathcal{D}_{\text{tr}}} \sigma(x_i) L(a_i^\top y , b_i) \big\} .\nonumber
\end{align}
In this problem, we have $\du = | {\cal D}_{\rm tr} |$ and $\dl$ is the dimension of the classifier.
Moreover, $(a_i,b_i)$ is the $i$th data point; $L(\cdot)$ is the loss function, with $y$ being the model parameter; $x_i$ is the parameter that determines the weight for the $i$th data sample, and $\sigma: \mathbb{R} \rightarrow \mathbb{R}_+$ is the weight function; $c>0$ is a regularization parameter; $\mathcal{D}_{\rm val}$ and $\mathcal{D}_{\rm tr}$ are validation and training sets, respectively.
Clearly, \eqref{eq:clean} is a special case of \eqref{Eq: BilevelProb_Deterministic} where the lower level problem finds the classifier $y^\ast(x)$ with the training set ${\cal D}_{\sf tr}$, and the upper level problem finds the best weights $x$ with respect to the validation set ${\cal D}_{\sf val}$. \hfill $\square$
	
\begin{table*}[t]
\centering
\renewcommand{\arraystretch}{1.25}
	\resizebox{\linewidth}{!}{\begin{tabular}{c c c c c} 
	\toprule
	Algorithm & {Sample (Upper, Lower)} &  Implementation & Batch Size & Per-Iteration Complexity     \\
	\midrule
	BSA \cite{Ghadimi_BSA_Arxiv_2018} & $ \mathcal{O}(\epsilon^{-2}),~ \mathcal{O}(\epsilon^{-3})$ &  Double loop  & $\mathcal{O}(1)$ & $\mathcal{O}(\dl^2\cdot \log T)$\\ 
	\hline
	stocBiO \cite{Ji_ProvablyFastBilevel_Arxiv_2020} & $\mathcal{O}(\epsilon^{-2}) , ~ \mathcal{O}(\epsilon^{-2})$ & Double loop   & $\mathcal{O}(\epsilon^{-1})$ & $\mathcal{O}(\dl^2\cdot \log T)$ \\
	\hline
	TTSA \cite{Hong_TTSA_Arxiv_2020} & $  \mathcal{O}(\epsilon^{-5/2}) ,~ \mathcal{O}(\epsilon^{-5/2})  $ & Single loop  & $\mathcal{O}(1)$ & $\mathcal{O}(\dl^2 \cdot \log T)$ \\
	\hline
	STABLE \cite{Chen_Arxiv_2021STABLE} & $\mathcal{O}(\epsilon^{-2}),~ \mathcal{O}(\epsilon^{-2})$ & Single loop & $\mathcal{O}(1)$ & $\mathcal{O}(\dl^3)$\\
    \hline
   {SVRB \cite{Guo_Arxiv_2021_SVRB}} & {$\mathcal{O}(\epsilon^{-3/2}),~ \mathcal{O}(\epsilon^{-3/2})$} & {Single loop}  &  {$\mathcal{O}(1)$} & {$\mathcal{O}(\dl^3  )$} \\
    \hline
	\cellcolor{black!8} \aname~(this work) & \cellcolor{black!8} $\mathcal{O}(\epsilon^{-3/2}),~ \mathcal{O}(\epsilon^{-3/2})$ & \cellcolor{black!8} Single loop  & \cellcolor{black!8} $\mathcal{O}(1)$ & \cellcolor{black!8} $\mathcal{O}(\dl^2 \cdot \log T)$\\
	\bottomrule
	\end{tabular}}
\caption{\small Comparison of  the number of upper and lower level gradient samples required to achieve an $\epsilon$-stationary point in Definition \ref{Def: StationaryPt}. For the algorithms with $\mathcal{O}(\dl^2 \cdot \log T)$ per-iteration dependence, the Hessian inverse can be computed via matrix vector products; algorithms with $\mathcal{O}(\dl^3)$ dependency requires Hessian inverses and Hessian projections, which incur heavy computational cost. } \label{Table: Comparison}
\vspace{-0.5cm}
\end{table*}

A natural approach to tackling \eqref{Eq: BilevelProb_Deterministic} is to apply alternating stochastic gradient (SG) updates. Let $\beta, \alpha > 0$ be some step sizes, one performs the recursion
\begin{align} \label{eq:roughSGD}
	& y^+ \leftarrow y - \beta \hat{\nabla}_y g( x, y ) , \quad x^+ \leftarrow x - \alpha \hat{\nabla}_x \hat{\ell}( x; y ) 
\end{align}
such that $\hat{\nabla}_y g( x, y )$, $\hat{\nabla}_x \hat{\ell}( x; y )$ are stochastic estimates of $\nabla_y g(x,y)$, $\nabla \ell(x)$, respectively. 
{Notice that \eqref{eq:roughSGD} is significantly different from the standard alternating primal-dual gradient algorithm for saddle point problems.}
Particularly, the design of $\hat{\nabla}_x \hat{\ell}( x; y )$ is crucial to the SG scheme in \eqref{eq:roughSGD}. Observe that $\nabla \ell(x)$ can be computed using the implicit function theorem, and its evaluation requires $f(\cdot, \cdot)$ and $y^\star(x)$, the minimizer of $g(x,y)$ given $x$ (cf. \eqref{Eq: True_Grad_Deterministic}). This gives rise to a unique challenge to bilevel optimization, where $y^\star(x)$ can only be {\it approximated} by $y$ obtained in the first relation of \eqref{eq:roughSGD}. 
% Such property has rendered bilevel optimization to be strictly more difficult than 

In light of the above observations, previous endeavors have considered two approaches to improve the estimate of $y^\star(x)$ while $\hat{\nabla}_x \hat{\ell}( x; y )$ is used as a {\it biased} approximation of $\nabla \ell(x)$. 
The first approach is to apply the \emph{double-loop} algorithms. For example, \cite{Ghadimi_BSA_Arxiv_2018} proposed to repeat the $y^+$ update for multiple times to obtain a better estimate of $y^\star(x)$ before performing the $x^+$ update, \cite{Ji_ProvablyFastBilevel_Arxiv_2020} proposed to take a large batch size to estimate $\nabla_y g(x,y)$. While simple to analyze, these algorithms may suffer from a poor sample complexity for the inner problem. 
The second approach is to apply \emph{single-loop} algorithms where the $y^+$-updates are performed simultaneously with the $x^+$-updates. Instead, advanced techniques are utilized that allows $y^+$ to accurately track $y^\star(x)$. For example, \cite{Hong_TTSA_Arxiv_2020} suggested to tune the step size schedule with $\beta \gg \alpha$, \cite{Chen_Arxiv_2021STABLE, Guo_Arxiv_2021_SVRB} proposed single-timescale algorithms with advanced variance reduction techniques. However, the latter two algorithms require Hessian projections onto a compact set along with Hessian matrices inversion which scales poorly with dimension (i.e., in $\mathcal{O}(\dl^3)$). We summarize and compare the complexity results of the state-of-the-art algorithms in Table~\ref{Table: Comparison}.

% works have focused on improving the estimate of $y$ while $\hat{\nabla}_x \hat{\ell}( x; y )$ is used as a {\it biased} approximation of $\nabla \ell(x)$. For example, \cite{Ghadimi_BSA_Arxiv_2018} proposed to repeat the $y$ update for multiple times to obtain a better estimate of $y^\star(x)$ before performing the $x^+$ update, resulting in a \emph{double-loop} scheme. With a similar double-loop structure, \cite{Ji_ProvablyFastBilevel_Arxiv_2020} proposed to take a large batch size to estimate $\nabla_y g(x,y)$. \cite{Hong_TTSA_Arxiv_2020} suggested to tune the step size schedule with $\beta \gg \alpha$, thereby allowing $y$ to accurately track $y^\star(x)$ via a \emph{two-timescale} stochastic approximation (TTSA) scheme; also see \cite{borkar1997stochastic}. Recently, \cite{Chen_Arxiv_2021STABLE} {\blue and \cite{Guo_Arxiv_2021_SVRB}}  proposed two related single-timescale algorithms which achieved the same rate as SG. However, these algorithms required Hessian projections onto a compact set along with Hessian matrices inversion which scales poorly with dimension (i.e., in $\mathcal{O}(\dl^3)$). 

% 	It is desired that the hessian inversions can be evaluated by only matrix vector products.  
%{\red[cite tianyi's work; discussion what do we mean by dimension dependency.]}
	
% Explicit finite sample convergence guarantees have been established in the prior works for the bilevel optimization problem \eqref{Eq: BilevelProb_Deterministic}, see the summary in Table~\ref{Table: Comparison}. 

A careful inspection on the above results reveals a gap in the iteration/sample complexity compared to \emph{single-level} stochastic optimization. For instance, an optimal stochastic gradient algorithm finds an $\epsilon$-stationary solution [cf.~Definition~\ref{Def: StationaryPt}] to $\min_{x} \mathbb{E}_\xi [ \ell(x; \xi ) ]$ in ${\cal O}(\epsilon^{-3/2})$ iterations \cite{Fang_Spider_NIPS_2018, Cutkosky_NIPS2019, Dinh_Arxiv_2019, Zhou_NIPS_2018_SNVRG}. For bilevel optimization, the fastest rate available is only ${\cal O}(\epsilon^{-2})$ to the best of the authors' knowledge. In comparison, the proposed algorithm achieves a rate of ${\cal O}(\epsilon^{-3/2})$.
% For instance, it is known that an $\epsilon$-stationary solution to the problem $\min_{x} \mathbb{E}_\xi [ \ell(x; \xi ) ]$ can be found with ${\cal O}(\epsilon^{-2})$ gradient evaluations with an SG algorithm \cite{Bottou_SIAM_2018_Review} using a batch size of ${\cal O}(1)$. Further, when assuming that per-sample gradient satisfies certain Lipschitz continuous assumption, then stochastic gradient algorithms can be modified to achieve an improved rate of ${\cal O}(\epsilon^{-3/2})$ \cite{Fang_Spider_NIPS_2018, Cutkosky_NIPS2019, Dinh_Arxiv_2019, Zhou_NIPS_2018_SNVRG}. 
% For the bilevel problem, earlier works such as TTSA and BSA achieve sample complexities that are worse, while more recent works of stocBiO and STABLE can match the SGD rate of ${\cal O}(\epsilon^{-2})$. 
% However, these works all require that the gradient samples from both the upper and lower level objective functions satisfy the Lipschitz gradient condition, therefore it is not clear if these sample complexities can still be improved. 
During the preparation of the current paper, a preprint \cite{Guo_Arxiv_2021_SVRB} has appeared which extended \cite{Chen_Arxiv_2021STABLE}, and achieves an improved rate of ${\cal O}(\epsilon^{-3/2})$. We remark that the latter work follows a different design philosophy from ours and maybe less efficient; see the detailed discussion at the end of Sec.~\ref{sec: Algo}.

% , and the algorithm is computationally expensive 

% (requires Hessian inversion per iteration, with complexity $\mathcal{O}(\dl^3)$). As a result, it is difficult, if not impossible, to implement this algorithm in practice {\red (in fact, as \cite{Chen_Arxiv_2021STABLE} stated, ``....")}. %{\red[if we include tianbao's paper, then this paragraph needs to be re-written a little bit.]}
	
\textbf{Contributions.}~~
In this paper, we depart from the prior developments which focused on finding better inner solutions $y^*(x)$ to approximate $\hat{\nabla}_x \hat{\ell}( x; y ) \approx \nabla \ell(x)$. Our idea is to exploit the gradient estimates from prior iterations to improve the quality of the current gradient estimation. This leads to   \emph{momentum}-assisted stochastic gradient estimators for {\it both} $\nabla_y g(x,y)$ and $\nabla \ell(x)$ using similar techniques  in \cite{Cutkosky_NIPS2019, Dinh_Arxiv_2019} for single-level stochastic optimization. 
The resultant algorithm only requires $O(1)$ samples at each update, and updates $x$ and $y$ using step sizes of the same order, hence the name \underline{s}ingle-timescale do\underline{u}ble-momentum \underline{st}ochastic \underline{a}pprox\underline{i}matio\underline{n}(\aname) algorithm. Additionally, it is worth noting that our algorithm has a $\mathcal{O}(\dl^2)$ per iteration complexity, compared to the $\mathcal{O}(\dl^3)$ complexity of STABLE \cite{Chen_Arxiv_2021STABLE} and SVRB \cite{Guo_Arxiv_2021_SVRB}. That is, the \aname~algorithm is both {\it sample and computation} efficient. Our specific contributions are:\vspace{-.2cm}
\begin{itemize}[leftmargin=4mm, itemsep = -.05mm]
\item We propose the \aname~algorithm for bilevel problems which matches the best complexity bounds as the optimal SGD algorithms for single-level stochastic optimization. That is, it  requires ${\cal O}(\epsilon^{-3/2})$ [resp. ${\cal O}(\epsilon^{-1})$] samples to find an $\epsilon$-stationary solution for non-convex  (resp. strongly-convex) bilevel problems; see Table~\ref{Table: Comparison}. %{\blue To our knowledge, this is the first stochastic bilevel algorithms that achieve such an optimal rate.}
Furthermore, the algorithm utilizes a single-loop update with step sizes of the same order for both upper and lower level problems. 
Such complexity bounds match the optimal sample complexity of stochastic gradient algorithms for single-level problems.
\item By developing the Lipschitz continuous property of the (biased) stochastic estimates of $\nabla \ell(x)$, we show that obtaining a good estimate of $\nabla \ell(x)$ does not require explicit (sampled) Hessian inversion. This key result ensures that our algorithm depends favorably on the problem dimension.  %Our convergence analysis relies on the construction of a single potential function which takes into account the error dynamics with momentum-assisted gradient estimator, and such an analysis provides insights into the choice of parameters. %{\red[while emphasizing the dimensionality dependency?]}
%\item We present numerical experiments on a quadratic-quadratic problem, and a data hyper-cleaning problem for cleaning data with corrupted labels. We demonstrate that the \aname~algorithm compares favorably with the state-of-the-art algorithms.
\item Comparing with prior works such as TTSA \cite{Hong_TTSA_Arxiv_2020},  BSA \cite{Ghadimi_BSA_Arxiv_2018},  STABLE \cite{Chen_Arxiv_2021STABLE} and SVRB \cite{Guo_Arxiv_2021_SVRB}, our analysis reveals that improving the gradient estimation quality for {\it both} $\nabla_y g(x,y)$ and $\nabla \ell(x)$ is the key to obtain a sample and computation efficient stochastic algorithm for bilevel optimization.
\end{itemize}
	
\textbf{Related works.}~~
The study of the bilevel problem \eqref{Eq: BilevelProb_Deterministic} can be traced to that of game theory \cite{stackelberg} and was formally introduced in \cite{Bracken73, Bracken74b,Bracken74}.  It is also related to the broader class of problems of Mathematical Programming with Equilibrium Constraints  \cite{luo_pang_ralph_1996}. 
Related algorithms include approximate descent \cite{Falk93, Vicente_QuadraticBilevel_1994}, and penalty-based methods \cite{White93}; see \cite{Colson_BilevelReview_AOR_2007} and \cite{Liu_Arxiv_2021Survey} for a comprehensive survey. 
	
In addition to the works cited in Table~\ref{Table: Comparison}, recent works on bilevel optimization have focused on algorithms with provable convergence rates. In \cite{Sabach_SAM_Siam_2017}, the authors proposed BigSAM algorithm for solving simple bilevel problems (with a single variable) with convex lower level problem. Subsequently, the works \cite{Liu_generic_Arxiv_2020, Li_ImprovedBilevel_Arxiv_2020} utilized BigSAM and developed algorithms for a general bilevel problem for the cases when the solution of the lower level problem is not a singleton. Note that all the aforementioned works  \cite{Sabach_SAM_Siam_2017,Liu_generic_Arxiv_2020, Li_ImprovedBilevel_Arxiv_2020} assumed the upper level problem to be strongly-convex with convex lower level problem. In a separate line of work, backpropagation based algorithms have been proposed to approximately solve bilevel problems \cite{Franceschi_ICML_2017, Shaban_TruncatedBackProp_2019, Grazzi_StochasticHypergradients_2020, Grazzi_ComplexityHypergrad_2020}. However, the major focus of these works was to develop efficient gradient estimators rather than on developing efficient optimization algorithms. 
%{\red why don't we add these comparison to Table 1?}{\red[I think they are not directly comparable.]}
	
\textbf{Notation.}~~ For any $x \in \mathbb{R}^d$, we denote $\| x \|$ as the standard Euclidean norm; as for $X \in \mathbb{R}^{n \times d}$, $\| X \|$ is induced by the Euclidean norm. For a multivariate function $f(x,y)$, the notation $\nabla_x f(x,y)$ [resp.~$\nabla_y f(x,y)$] refers to the partial gradient taken with respect to (w.r.t.) $x$ [resp.~$y$]. For some $\mu > 0$, a function $f(x,y)$ is said to be $\mu$-strongly-convex in $x$ if $f(x,y) - \frac{\mu}{2} \| x \|^2$ is convex in $x$. For some $L > 0$, the map ${\cal A} : \mathbb{R}^d \rightarrow \mathbb{R}^m$ is said to be $L$-Lipschitz continuous if $\| {\cal A}(x) - {\cal A}(y) \| \leq L \| x - y \|$ for any $x,y \in \mathbb{R}^d$. A function $f: \mathbb{R}^d \rightarrow \mathbb{R}$ is said to be $L$-smooth if its gradient is $L$-Lipschitz continuous. Uniform distribution over a discrete set $\{1, \ldots, T\}$ is represented by $\mathcal{U}\{1, \ldots, T \}$.

Finally, we state the following definitions for the optimality criteria of \eqref{Eq: BilevelProb_Deterministic}.  
\begin{defn}[{$\epsilon$-Stationary Point}]
\label{Def: StationaryPt}
A point $x$ is called $\epsilon$-stationary if $\| \nabla \ell(x) \|^2 \leq \epsilon$. A stochastic algorithm is said to achieve an $\epsilon$-stationary point in $t$ iterations if $\mathbb{E}[\| \nabla \ell(x_t) \|^2] \leq \epsilon$, where the expectation is over the stochasticity of the algorithm until time instant $t$.
\end{defn}
	
\begin{defn}[{$\epsilon$-Optimal Point}]
\label{Def: OptimalPt}
A point $x$ is called $\epsilon$-optimal if $\ell(x) - \ell^\ast  \leq \epsilon$, where $\ell^\ast \coloneqq \min_{x \in \mathbb{R}^{\du} } \ell (x)$. A stochastic algorithm is said to achieve an $\epsilon$-optimal point in $t$ iterations if $\mathbb{E}[ \ell(x_t)  - \ell^\ast ] \leq \epsilon$, where the expectation is over the stochasticity of the algorithm until time instant $t$.\vspace{-.1cm}
\end{defn}
	
\section{Preliminaries} \label{Sec: ProblemandAssumptions}\vspace{-.1cm}
We discuss the assumptions on \eqref{Eq: BilevelProb_Deterministic} to specify the problem class of interest. We also preface the proposed algorithm by describing a practical procedure for estimating the stochastic gradients.  
\begin{assump}[Upper Level Function] \label{Assump: OuterFunction}
$f(x,y)$ satisfies the following conditions:
\begin{enumerate}[label=(\roman*), itemsep=-.2mm, leftmargin=7mm, topsep = -2.5mm]
\item $\nabla_x f(x , y)$ and $\nabla_y f(x,y)$ are Lipschitz continuous w.r.t. $(x,y) \in \mathbb{R}^{\du} \times \mathbb{R}^{\dl}$, and with constants $L_{f_x} \geq 0$ and $L_{f_y} \geq 0$, respectively. 
% \item For any $x \in \mathbb{R}^{\du}$, $\nabla_x f(x , \cdot)$ and $\nabla_y f(x,\cdot)$ are Lipschitz continuous w.r.t. $y \in \mathbb{R}^{\du}$, and with constants $L_{f_x} > 0$ and $L_{f_y} > 0$, respectively. 
% \item For any $y \in \mathbb{R}^{\du}$, $\nabla_x f(\cdot, y)$ and $\nabla_y f(\cdot, y)$ are Lipschitz continuous w.r.t. $x \in \mathbb{R}^{\du}$ and with constants $L_{f_x} > 0$ and $L_{f_y} > 0$, respectively. 
% 			Note that here we overload the notations from Assumption \ref{Assump: OuterFunction}--(i). 
\item For any $(x,y) \in \mathbb{R}^{\du} \times \mathbb{R}^{\dl}$, we have $\|\nabla_y f(x,y)\| \leq C_{f_y}$, for some $C_{f_y} \geq 0$.
\end{enumerate}
\end{assump}
	
\begin{assump}[Lower level Function]
\label{Assump: InnerFunction} $g(x,y)$ satisfies the following conditions: 
\begin{enumerate}[label=(\roman*), itemsep=-1mm, leftmargin=7mm, topsep = -2.5mm]
\item For any $x \in \mathbb{R}^{\du}$ and $y \in \mathbb{R}^{\dl}$, $g(x,y)$ is twice continuously differentiable in $(x,y)$.
\item $\nabla_y g(x, y)$ is Lipschitz continuous w.r.t. $(x,y) \in \mathbb{R}^{\du} \times \mathbb{R}^{\dl}$, and with constant $L_g \geq 0$.
% \item For any $x \in \mathbb{R}^{\du}$, $\nabla_y g(x, \cdot)$ is Lipschitz continuous w.r.t. $y \in \mathbb{R}^{\du}$, and with constant $L_g > 0$.
% \item For any $y \in \mathbb{R}^{\du}$, $\nabla_y g(\cdot, y)$ is Lipschitz continuous w.r.t. $x \in \mathbb{R}^{\du}$, and with constant $L_g > 0$. Note that here we overload the notations from Assumption \ref{Assump: InnerFunction}--(ii).
\item For any $x \in \mathbb{R}^{\du}$, $g(x, \cdot)$ is $\mu_g$-strongly-convex in $y$ for some $\mu_g > 0$.
\item $\nabla^2_{xy} g(x, y)$ and $\nabla^2_{yy} g(x , y)$ are Lipschitz continuous w.r.t. $(x,y) \in \mathbb{R}^{\du} \times \mathbb{R}^{\dl}$, and with constants $L_{g_{xy}} \geq 0$ and $L_{g_{yy}} \geq 0$, respectively.
% \item For any $x \in \mathbb{R}^{\du}$, $\nabla^2_{xy} g(x, \cdot)$ and $\nabla^2_{yy} g(x , \cdot)$ are Lipschitz continuous w.r.t. $y \in \mathbb{R}^{\du}$, and with constants $L_{g_{xy}} > 0$ and $L_{g_{yy}} > 0$, respectively.
% \item For any $y \in \mathbb{R}^{\du}$, $\nabla^2_{xy} g(\cdot, y)$ and $\nabla^2_{yy} g(\cdot , y)$ are Lipschitz continuous w.r.t. $x \in \mathbb{R}^{\du}$, and with constants $L_{g_{xy}} > 0$ and $L_{g_{yy}} > 0$, respectively. 
% Note that here we overload the notation from Assumption \ref{Assump: InnerFunction}--(iv).  
\item For any $(x,y) \in \mathbb{R}^{\du} \times \mathbb{R}^{\dl}$, we have $\|\nabla^2_{xy} g(x,y)\|^2 \leq C_{g_{xy}}$ for some $C_{g_{xy}} > 0$.
\end{enumerate}
\end{assump}
	
\begin{assump} [Stochastic Functions] \label{Assump: StocFn}
\label{Assump: Stochastic} Assumptions \ref{Assump: OuterFunction} and \ref{Assump: InnerFunction} hold for $f(x,y;\xi)$ and $g(x,y;\zeta)$, for all $\xi\in {\rm supp}(\pi_f)$ and $\zeta\in {\rm supp}(\pi_g)$ where ${\rm supp}(\pi)$ is the support of $\pi$.
\end{assump}
These assumptions are standard in the analysis of bilevel optimization \cite{Ghadimi_BSA_Arxiv_2018}. For example,  they are satisfied by a range of applications such as the meta learning problem \eqref{Eq: Meta_Leaning}, data hypercleaning problem \eqref{eq:clean} with linear classifier. 
Notice that under these assumptions, the gradient $\nabla \ell(\cdot)$ is well-defined. By utilizing Assumption \ref{Assump: InnerFunction}--(i) and (ii) along with the implicit function theorem \cite{Rudin_Book}, it is easy to show  that for a given $\bar{x}\in\mathbb{R}^{\du}$, the following holds \cite[Lemma 2.1]{Ghadimi_BSA_Arxiv_2018}:
\begin{align}
	\label{Eq: True_Grad_Deterministic}
	\nabla \ell( \bar{x}) & =  \nabla_x f(\bar{x}, y^\ast(\bar{x}))   - \nabla^2_{xy} g(\bar{x} , y^\ast(\bar{x}))    [\nabla^2_{yy} g(\bar{x}, y^\ast(\bar{x}))]^{-1} 	\nabla_y f(\bar{x}, y^\ast(\bar{x})).
\end{align}
Obtaining $y^\ast(x)$ in closed-form is usually a challenging task, so it is natural to use the following gradient surrogate. At any $(\bar{x}, \bar{y}) \in \mathbb{R}^{\du\times \dl}$, define:
\begin{align}
	\label{Eq: Estimated_Grad_Deterministic}
	\bar{\nabla} f(\bar{x} , \bar{y}) & = \nabla_x f(\bar{x}, \bar{y})  - \nabla^2_{xy} g(\bar{x} , \bar{y})  [\nabla^2_{yy} g(\bar{x}, \bar{y})]^{-1} \nabla_y f(\bar{x}, \bar{y}).
\end{align}
	
Evaluating \eqref{Eq: Estimated_Grad_Deterministic} requires computing the exact gradients and Hessian inverse which can be non-trivial. Below, we describe a practical procedure from \cite{Ghadimi_BSA_Arxiv_2018} to generate a \emph{biased} estimate of $\bar{\nabla} f(\bar{x} , \bar{y})$.
	
\textbf{Stochastic gradient estimator for $\nabla \ell(x)$.}~~
The estimator requires a parameter $K \in \mathbb{N}$ and is based on a collection of $K+3$ independent samples $\bar{\xi} := \{ \xi , \zeta^{(0)}, ..., \zeta^{(K)}, {\sf k}(K) \}$, where $\xi \sim \mu$, $\zeta^{(i)} \sim \pi_g$, $i=0,...,K$, and ${\sf k}(K) \sim {\cal U}\{ 0,..., K-1\}$. We set {\small
\begin{align}
	\label{Eq: UpperLevel_StochasticGrad}
	\bar{\nabla} f(x, y ; \bar{\xi}) = \nabla_x f(x,y ;  \xi) - \frac{K}{L_g} \nabla^2_{xy} g(x,y ; \zeta^{(0)}) \prod_{i = 1}^{{\sf k}(K)} \bigg( I - \frac{\nabla^2_{yy} g(x, y ; \zeta^{(i)})}{L_g}  \bigg) \nabla_y f(x, y ;  \xi),  
\end{align}}%
where we have used the convention $\prod_{i = 1}^j A_i = I$ if $j = 0$. It has been shown in \cite{Ghadimi_BSA_Arxiv_2018,Hong_TTSA_Arxiv_2020} that the bias with the gradient estimator \eqref{Eq: UpperLevel_StochasticGrad} decays exponentially fast with $K$, as summarized below:
\begin{lem}\label{lm:biased}
Under Assumptions~\ref{Assump: OuterFunction}, \ref{Assump: InnerFunction}. For any $K \geq 1$, the gradient estimator in \eqref{Eq: UpperLevel_StochasticGrad} satisfies
\begin{equation} 
\| \bar{\nabla} f(x,y) - \mathbb{E}_{\bar{\xi}} [ \bar{\nabla} f(x,y ; \bar{\xi} )] \| \leq \frac{ C_{g_{xy}} C_{f_y} }{ \mu_g } \left( 1 - \frac{\mu_g}{L_g} \right)^K,~~\forall~(x,y) \in \mathbb{R}^{\du} \times \mathbb{R}^{\dl}.
\end{equation}
% Furthermore, suppose that the gradient/Hessian samples in \eqref{Eq: UpperLevel_StochasticGrad} have bounded variance, then there exists a constant $\sigma_f$ where $\mathbb{E}_{\bar{\xi}}[ \| \bar{\nabla} f(x,y ; \bar{\xi} ) - \mathbb{E}_{\bar{\xi}}[ \bar{\nabla} f(x,y ; \bar{\xi} ) ] \|^2 ] \leq \sigma_f^2$. 
\end{lem}
The detailed statement of the above lemma is included in Appendix~\ref{Sec: Appendix_PreliminaryLemmas}.
% \begin{rem}
We remark that each computation of $\bar{\nabla} f(x, y ; \bar{\xi})$ requires at most $K$ Hessian-vector products, and later we will show that setting $K=\mathcal{O}(\log(T))$ is necessary for the proposed algorithm. Since $\nabla^2_{yy}g(x,y;\zeta)$ is of size $\dl \times \dl$, the total complexity of this step is  $\mathcal{O}(\log(T) \dl^2 )$. On the contrary, STABLE \cite{Chen_Arxiv_2021STABLE} and SVRB \cite{Guo_Arxiv_2021_SVRB} require  $\mathcal{O}( \dl^3)$ to estimate the Hessian inverse, which is more computationally expensive when $\dl \gg 1$. Indeed, it has been explicitly mentioned in \cite{Chen_Arxiv_2021STABLE} that ``\emph{our algorithm (STABLE) is preferable in the regime where the sampling is
more costly than computation or the dimension $d$ is relatively small}''. %{\red[use some of the words/sentences from Liang's work?]}
% \end{rem}
	
Notice that \eqref{Eq: UpperLevel_StochasticGrad} is not the only option for estimating the gradient surrogate $\bar{\nabla} f(x,y)$. For ease of presentation, below we abstract out the conditions on the stochastic estimates of $\nabla_y g$, $\bar{\nabla} f$ required by our analysis as the following assumption:  %{\red[don't we need sample-wise Lipschitz?]}
\begin{assump}[Stochastic Gradients] For any $(x,y) \in \mathbb{R}^{\du} \times \mathbb{R}^{\dl}$, there exists constants $\sigma_f, \sigma_g \geq 0$ such that the estimates $\nabla_y g(x , y ; \zeta)$, $\bar{\nabla} f (x, y ; \bar{\xi})$ satisfy:
\label{Assump: Stochastic Grad}
\begin{enumerate}[leftmargin=6mm,label=(\roman*)]
\item The gradient estimate of the upper level objective satisfies:
\begin{align}
% 			& \hspace{-.4cm} \mathbb{E}_{\bar{\xi}}\big[\bar{\nabla} f (x, y ; \bar{\xi})  \big] = \bar{\nabla} f(x, y) + B(x,y), \\
& \mathbb{E}_{\bar{\xi}}\big[\| \bar{\nabla} f (x, y ; \bar{\xi}) - \bar{\nabla} f (x, y ) - B(x,y) \|^2 \big] \leq \sigma_f^2,
\end{align}
where $B(x,y) = \mathbb{E}_{\bar{\xi}} [\bar{\nabla} f (x, y ; \bar{\xi})] - \bar{\nabla} f (x, y )$ is the bias in estimating $\bar{\nabla} f(x,y)$.
\item The gradient estimate of the lower level objective satisfies
\begin{align}
% 			& \mathbb{E}_{\zeta}\big[\nabla_y g(x , y ; \zeta) \big]  = \nabla_y g(x , y), \\
& \mathbb{E}_{\zeta} \big[ \|\nabla_y g(x , y ; \zeta) - \nabla_y g(x , y)\|^2  \big]   \leq \sigma_g^2.
\end{align}
\end{enumerate}
\end{assump}
As observed from Lemma~\ref{lm:biased}, the gradient estimator \eqref{Eq: UpperLevel_StochasticGrad} satisfies Assumption \ref{Assump: Stochastic Grad}(i).
% 	We remark that the estimate for the lower level objective function's gradient is simply the SG, $\nabla_y g(x , y ; \zeta)$, where $\zeta \sim \pi_g$ is drawn independently. 
	
Lastly, the approximate gradient defined in \eqref{Eq: Estimated_Grad_Deterministic}, the true gradient \eqref{Eq: True_Grad_Deterministic}, as well as the optimal solution of the lower level problem are Lipschitz continuous, as proven below:
\begin{lem}{\cite[Lemma 2.2]{Ghadimi_BSA_Arxiv_2018}}\label{Lem: Lip_Ghadhimi}
Under Assumptions \ref{Assump: OuterFunction}, \ref{Assump: InnerFunction} and \ref{Assump: StocFn}, we have
\begin{equation} \label{eq:lip}
% \begin{align}
\begin{array}{c}
\| \bar{\nabla}f(x , y) - \nabla \ell(x) \| \leq L \|y^\ast(x) - y \|, 
% \label{eq:lip:bar:f}\\
\quad \| y^\ast(x_1) - y^\ast(x_2)\| \leq L_y \| x_1 - x_2\|, \\[.2cm]
\| \nabla \ell(x_1) - \nabla \ell(x_2)\| \leq L_f \|x_1 - x_2 \|, 
\end{array}
% \label{eq:lip:ell}
% \end{align}
\end{equation}
for all $x,x_1,x_2 \in \mathbb{R}^{\du}$ and $y \in \mathbb{R}^{\dl}$. The above Lipschitz constants are defined as:
\begin{align}
& L = L_{f_x}  +  \frac{L_{f_y} C_{g_{xy}}}{\mu_g} + C_{f_y}  \bigg( \frac{L_{g_{xy}}}{\mu_g}   +  \frac{L_{g_{yy}} C_{g_{xy}}}{\mu_g^2}\bigg), \quad L_f = L + \frac{ L C_{g_{xy}} }{ \mu_g }, \quad L_y  = \frac{C_g}{\mu_g}.
% \\
% &  L_f   =  L_{f_x}  +  \frac{(L_{f_y} + L) C_{g_{xy}}}{\mu_g}   +  C_{f_y} \bigg( \frac{L_{g_{xy}}}{\mu_g}  +  \frac{L_{g_{yy}} C_{g_{xy}}}{\mu_g^2}  \bigg), \quad	L_y  = \frac{C_g}{\mu_g}  . 
\end{align}
\end{lem}
The first result in \eqref{eq:lip} reveals that $\bar{\nabla} f(x,y)$ approximates $\nabla \ell(x)$ when $y \approx y^\ast(x)$. This suggests that a \emph{double-loop} algorithm which solves the strongly-convex lower level problem to sufficient accuracy can be applied to tackle \eqref{Eq: BilevelProb_Deterministic}. Such approach has been pursued in \cite{Ghadimi_BSA_Arxiv_2018,Ji_ProvablyFastBilevel_Arxiv_2020}. Next, we propose an algorithm which rely on \emph{single-loop} updates with improved sample efficiency. 

% To develop a faster algorithm, this paper focuses on a \emph{single-loop} approach which draws $\mathcal{O}(1)$ samples for upper and lower level problems.
% The latter presents significant challenges for algorithm design as the stochastic estimates \eqref{Eq: Estimated_Grad_Deterministic} and \eqref{Eq: UpperLevel_StochasticGrad} have non-vanishing variances. 
% 	\qed
% \end{rem}

% 	\paragraph{Optimality Criteria.} 
\section{The proposed \aname~algorithm}\label{sec: Algo}
Equipped with a practical stochastic gradient estimator for $\nabla \ell(x)$ [cf.~\eqref{Eq: UpperLevel_StochasticGrad}], our next endeavor is to develop a \emph{single-loop} algorithm to tackle \eqref{Eq: BilevelProb_Deterministic} through drawing $\mathcal{O}(1)$ samples for upper and lower level problems at each iteration. Our main idea is to adopt the recursive momentum techniques developed in \cite{Cutkosky_NIPS2019,Dinh_Arxiv_2019}. Notice that these works utilize \emph{unbiased} stochastic gradients evaluated at consecutive iterates to construct a variance reduced gradient estimate for single-level stochastic optimization. 

In the context of bilevel stochastic optimization \eqref{Eq: BilevelProb_Deterministic}, a few key challenges are in order:
\begin{itemize}[leftmargin=4mm, itemsep=-.5mm]
    \item Recall from Lemma~\ref{lm:biased} that obtaining an unbiased estimator for the outer gradient $\nabla \ell(x)$ requires using $K \rightarrow \infty$ samples in \eqref{Eq: UpperLevel_StochasticGrad}, this calls for the new techniques to control the bias arising from approximating $\nabla \ell(x)$. 
    % As indicated by Lemma~\ref{lm:biased}, We do not have access to an unbiased estimator of $\nabla \ell(x)$, while in the single-level case the unbiased stochastic estimator of gradients are always available. 
    \item 
    % It is reasonable to develop a good estimator for the gradient surrogate $\bar{\nabla} f(x,y)$ to approximate $\nabla \ell(x)$, because by Lemma \ref{lm:biased} we can construct a reasonably accurate (albeit still biased) estimator for it. 
    The gradient estimator \eqref{Eq: UpperLevel_StochasticGrad} has a more complicated structure than a plain gradient estimator, as it involves up to three different stochastic vectors/matrices related to $\nabla_x f(x,y)$, $\nabla_y f(x,y)$, $\nabla_{xy}g(x,y)$, and one stochastic inversion that is related to $[\nabla_{yy}g(x,y)]^{-1}$. It is not clear which are the most important objects for which variance reduction shall be applied.
\end{itemize}
% Overall, it is not clear how the gradient estimators developed for single-level problems, as well as  their associated analyses can be applied in the bilevel setting.

Our key innovation is to develop a useful estimate of $\bar{\nabla} f(x,y)$ by using a novel {\it double-momentum} technique. First, we build a recursive momentum estimator for $\nabla_y g(x,y)$, based upon  which the variable $y$ gets updated. Then, with such a "stabilized" inner iteration, we compute an estimate of $\bar{\nabla} f(x,y)$ as given in \eqref{Eq: UpperLevel_StochasticGrad},  by using the four stochastic vectors/matrices mentioned above but without performing any variance reduction.   Such a stochastic estimator will then be used to construct a recursive momentum estimator for $\bar{\nabla}f(x,y)$. The intuition is that as long as $y$ is {\it accurate enough}, then the stochastic terms in \eqref{Eq: UpperLevel_StochasticGrad} are also accurate enough, so they can be used to construct the estimator for the outer gradient. Our approach only tracks two vector estimators, while still being able to leverage the low-complexity sample-based Hessian inversion as given  in \eqref{Eq: UpperLevel_StochasticGrad}.

The \aname~algorithm is summarized in Algorithm~\ref{Algo: Accelerated_STSA}. Define $\eta^g_t \in [0,1]$, $\eta^f_t \in [0,1]$. For the lower level problem involving $y$, it utilizes the following momentum-assisted gradient estimator, $h_t^g \in \mathbb{R}^{\dl}$, defined recursively as 
% {\red[did we define $\eta^g_t$ and $\eta^f_t$ and their ranges?]}
		\begin{align} \label{Eq: HybridGradEstimate_Inner}
	h_{t}^{g}  =  \eta_t^g \nabla_y g(x_t , y_t; \zeta_t) +  (1 - \eta_{t}^g) \big(h_{t - 1}^{g} + \nabla_y g(x_t , y_t; \zeta_t) -   \nabla_y g(x_{t - 1} , y_{t - 1}; \zeta_t)  \big);
	\end{align}
	For the upper level problem involving $x$, we utilize a similar estimate, $h_t^f \in \mathbb{R}^{\du}$, defined as
	\begin{align} \label{Eq: HybridGradEstimate_Outer}
	h_{t}^{f}  =  \eta_t^f \bar{\nabla} f(x_t , y_t ; \bar{\xi}_t) +  (1 - \eta_{t}^f)  \big(h_{t - 1}^{f} + \bar{\nabla} f(x_t, y_t ; \bar{\xi}_t) -   \bar{\nabla} f(x_{t-1}, y_{t - 1} ; \bar{\xi}_t)  \big).
	\end{align}
	The gradient estimators $h_t^g$ and $h_t^f$ are computed from the current and past gradient estimates $\nabla_y g(x_t, y_t; \zeta_t)$, $\nabla_y g(x_{t-1}, y_{t-1}; \zeta_t)$ and $\bar{\nabla} f(x_t, y_t ; \bar{\xi}_t)$, $\bar{\nabla} f(x_{t-1}, y_{t-1} ; \bar{\xi}_t)$. Note that the stochastic gradients at two consecutive iterates are computed using the same sample sets $\zeta_t$ for $h_t^g$ and $\bar{\xi}_t$ for $h_t^f$. 
	
Both $x$ and $y$-update steps mark a major departure of the \aname~algorithm from existing algorithms on bilevel optimization  \cite{Ghadimi_BSA_Arxiv_2018,Hong_TTSA_Arxiv_2020,Ji_ProvablyFastBilevel_Arxiv_2020}. The latter works apply the direct gradient estimator $\bar{\nabla} f(x_t , y_{t+1} ; \bar{\xi}_t)$ [cf.~\eqref{Eq: UpperLevel_StochasticGrad}] to serve as an estimate to $\bar{\nabla} f(x,y)$ [and subsequently $\nabla \ell(x)$]. 
To guarantee convergence, these works focused on improving the \emph{tracking performance} of $y_{t+1} \approx y^\star(x_t)$ by employing double-loop updates, e.g., by repeatedly applying SG step multiple times for the inner problem; or a sophisticated two-timescale design for the step sizes, e.g., by setting $\beta_t / \alpha_t \rightarrow \infty$.

A recent preprint \cite{Guo_Arxiv_2021_SVRB} suggested the SVRB algorithm which applies a similar recursive momentum technique as \aname. However, the SVRB algorithm is different from \aname~as the momentum estimator is applied exhaustively to \emph{all} the individual random quantities involved in \eqref{Eq: UpperLevel_StochasticGrad} and requires a Hessian projection step. As a result, the SVRB algorithm entails a high complexity in storage and computation as the latter has to store matrix variables of size $\dl \times \dl$ and computes a matrix inverse for each iteration. In comparison, the \aname~algorithm only requires storing the gradient estimators $h_t^g, h_t^f$ of size $\dl, \du$, respectively, and the computation complexity is only ${\cal O}(\dl^2 K )$ for each iteration.

\begin{algorithm}[t]
\caption{The Proposed \aname~Algorithm} \label{Algo: Accelerated_STSA}
\begin{algorithmic}[1]
\State{\textbf{Input}: Parameters: $\{\beta_t\}_{t=0}^{T-1}$,  $\{\alpha_t\}_{t=0}^{T - 1}$, $\{\eta_t^f\}_{t=0}^{T-1}$, and $\{\eta_t^g\}_{t=0}^{T-1}$ with $\eta_0^f = \eta_0^g  =1$}
\State{\textbf{Initialize}: $x_0$, $y_0$; set $x_{-1} =  y_{-1} = h_{-1}^f = h_{-1}^g = 0$}
\For{$t = 0$ to $T - 1$}
	\State{\hspace{-.1cm}(\texttt{$y$-update}) Compute the gradient estimator $h^g_t$ by \eqref{Eq: HybridGradEstimate_Inner} and set  $ y_{t+1} =  y_t - \beta_{t} h_{t}^{g}$.}
	\State{\hspace{-.1cm}(\texttt{$x$-update}) Compute the gradient estimator $h^f_t$ by \eqref{Eq: HybridGradEstimate_Outer} and set $ x_{t+1} =  x_t - \alpha_{t} h_{t}^{f}$.}
\EndFor
\State{{\bf Return:} $x_{a(T)}$ where $a(T) \sim {\cal U}\{1,...,T\}$.}
\end{algorithmic}
\end{algorithm}

\subsection{Convergence analysis}
In the following, we present the convergence analysis for the \aname~algorithm when $\ell (\cdot)$ is a smooth function [cf.~consequence of Assumptions \ref{Assump: OuterFunction}, \ref{Assump: InnerFunction} and \ref{Assump: StocFn}]. 
%We outline the proof of the convergence analysis for the case where the outer objective is non-convex. 
%Before proceeding with the analysis of Algorithm~\ref{Algo: Accelerated_STSA}, 
Before proceeding to the main results, we present a lemma about the Lipschitzness of the gradient estimate $\bar{\nabla} f(x,y; \bar{\xi} )$ given in \eqref{Eq: UpperLevel_StochasticGrad}:
\begin{lem}
	\label{Lem: Lip_GradEst}
	Under Assumptions \ref{Assump: OuterFunction}, \ref{Assump: InnerFunction} and \ref{Assump: StocFn}, we have for any $(x_1,y_1), (x_2,y_2) \in \mathbb{R}^{\du} \times \mathbb{R}^{\dl}$, 
	\begin{align*}
	\mathbb{E}_{\bar{\xi}} \| \bar{\nabla} f(x_1, y_1 ; \bar{\xi}) - \bar{\nabla} f(x_2, y_2 ; \bar{\xi}) \| \leq L_K^2 \big\{ \|x_1 - x_2\| + \| y_1 - y_2 \| \big\}^2,
	\end{align*}
% 		\begin{enumerate}[label=(\roman*)]
% 			\item For a fixed $y \in \mathbb{R}^{\dl}$, and for all $x_1,x_2 \in \mathbb{R}^{\du}$, we have
% 			\begin{align*}
% 			\mathbb{E}_{\bar{\xi}}   \| \bar{\nabla} f(x_1, y ; \bar{\xi}) - \bar{\nabla} f(x_2, y ; \bar{\xi}) \|^2 \leq L_K^2 \|x_1 - x_2\|^2.
% 			\end{align*}
% 			\item For a fixed $x \in \mathbb{R}^{\du}$ and for all $y_1,y_2 \in \mathbb{R}^{\dl}$, we have
% 			\begin{align*}
% 			\mathbb{E}_{\bar{\xi}} \| \bar{\nabla} f(x, y_1 ; \bar{\xi})  - \bar{\nabla} f(x, y_2 ; \bar{\xi})   \|^2 \leq L_K^2 \|y_1 - y_2\|^2.
% 			\end{align*}
% 		\end{enumerate}
where 
\begin{align} 
L_K = \sqrt{ 2 L_{f_x}^2 + \frac{ 6 C_{g_{xy}}^2 L_{f_y}^2 K}{2 \mu_g L_g - \mu_g^2} +   \frac{ 6 C_{f_y}^2 L_{g_{xy}}^2 K}{2 \mu_g L_g - \mu_g^2}   + \frac{6 C_{g_{xy}}^2 C_{f_y}^2 L_{g_{yy}}^2 K^3 }{(L_g - \mu_g)^2 (2 \mu_g L_g - \mu_g^2)} }, \label{eq:LKDef}
\end{align}
and $K$ is the number of samples required to construct the stochastic gradient estimate given in \eqref{Eq: UpperLevel_StochasticGrad}.
\end{lem}
 The detailed proof can be found in Appendix \ref{Sec: Appendix_PreliminaryLemmas}.
We remark that the above result is crucial for analyzing the error of the gradient estimate $h_t^f$ defined in \eqref{Eq: HybridGradEstimate_Outer}. To see this,
let us first define the errors of the gradient estimates for the outer and inner functions as follows   
	\begin{align}
	e_t^f &\coloneqq h_t^f - \bar{\nabla}f(x_t , y_{t}) - B_t , \quad 
% 	\label{eq:a}\\
	    	e_t^g \coloneqq h_t^g - \bar{\nabla}_y g(x_t , y_{t}), \label{eq:atg_analyze} 
	\end{align}
	where $B_t := B( x_t, y_{t} )$ denotes the bias. Rewriting $e_t^f$ using \eqref{Eq: HybridGradEstimate_Outer} gives the following recursion:
	\begin{align*}
	%\label{eq:at_analyze} 
	e_t^f & = (1 - \eta_t^f) e_{t - 1}^f  + (1 - \eta_t^f)  \big\{ \bar{\nabla} f(x_t, y_{t} ; \bar{\xi}_t) - \bar{\nabla} f(x_{t-1}, y_{t-1} ; \bar{\xi}_t) \nonumber \\
	&  \quad -  (\bar{\nabla} f(x_t, y_{t}) + B_t -   \bar{\nabla} f(x_{t-1}, y_{t-1}) - B_{t - 1} ) \big\}   + \eta_t^f  \big( \bar{\nabla} f(x_t , y_{t} ; \bar{\xi}_t) - \bar{\nabla}f(x_t , y_{t}) - B_t \big) .
	\end{align*}
Lemma \ref{Lem: Lip_GradEst} allows us to control the variance of the second term in the above relation as ${\cal O}(\alpha_t^2 \| h_{t-1}^f \|^2 + \beta_t^2 \| h_{t-1}^g\|^2 )$. This subsequently leads to a reduced error magnitude for $\mathbb{E} [\| e_t^f\|^2]$. Similarly, we can show a reduced error magnitude for $\mathbb{E}[\|e_t^g\|^2]$ for the inner gradient estimate. 

%Lemma \ref{Lem: Lip_GradEst} enables \aname~to achieve improved (near optimal) convergence for stochastic bilevel optimization \eqref{Eq: BilevelProb_Deterministic} by improving the gradient estimate of the upper level problem. Specifically, Lemma \ref{Lem: Lip_GradEst} ensures that we can efficiently 
The above discussion suggests that we can track the gradient $\nabla \ell(x)$ using only stochastic gradient estimates \eqref{Eq: UpperLevel_StochasticGrad}, without needing to track each component stochastic vectors/matrices. This allows us to avoid costly Hessian inversions. In contrast, \cite{Chen_Arxiv_2021STABLE,Guo_Arxiv_2021_SVRB} track the individual stochastic vectors/matrices of \eqref{Eq: UpperLevel_StochasticGrad}, and then combine them together to yield an estimate of $\nabla \ell(x)$. This approach is unable to utilize the cheap stochastic estimates of Hessian and have to invert it directly.

Turning back to the convergence analysis of the \aname~algorithm, the main idea of our analysis is to demonstrate reduction of a properly constructed potential function across iterations. For smooth (possibly non-convex) objective function, this potential function consists of a linear combination of the norms of the error terms $\mathbb{E} [\| e_t^f\|^2]$ and $\mathbb{E}[\|e_t^g\|^2]$ along with the outer objective function $\ell(x_t)$ and the inner optimality gap $\|y_t-y^{*}(x_t)\|^2$. 
We obtain: %{\red[prashant, double check this paragraph. Seems correct!]}
\begin{theorem}
\label{Thm: Convergence_NC}
Under Assumptions \ref{Assump: OuterFunction}--\ref{Assump: Stochastic Grad}.
% 		\ref{Assump: InnerFunction}, \ref{Assump: StocFn},
Fix $T \geq 1$ as the maximum iteration number. Set the number of samples used for the gradient estimator in \eqref{Eq: UpperLevel_StochasticGrad} as $K = ({L_g} / {\mu_g}) \log \left( {  C_{g_{xy}} C_{f_y} } T / \mu_g \right)$ and
\begin{equation}
    \alpha_t = \frac{1}{(w + t)^{1/3}}, \quad \beta_t = c_\beta \alpha_t, \quad \eta_t^f = c_{\eta_f} \alpha_t^2, \quad \eta_t^g = c_{\eta_g} \alpha_t^2 , 
\end{equation}
where $w, c_\beta, c_{\eta_f}, c_{\eta_g}$ are defined in \eqref{eq:step_param} of appendix.
The iterates generated by Algorithm \ref{Algo: Accelerated_STSA} satisfy %{\red[explicit write log?]}
% \begin{align*}
% \bar{c}_{\eta_f} = \max \bigg\{ 36 L_K^2 , \frac{4 L_K^2  L_{\mu_g} (\mu_g + L_g) c_\beta^2}{L^2} \bigg\} \quad \text{and} \quad \bar{c}_{\eta_g}  = \max\bigg\{    36 L_g^2 , \frac{4 L_g^2 L_{\mu_g} (\mu_g + L_g) c_\beta^2}{L^2} \bigg\}.
% \end{align*}
\begin{align} \label{eq:converge_bound}
\mathbb{E}	\| \nabla \ell(x_{a(T)})\|^2 = \mathcal{O} \bigg( \frac{\ell(x_0) - \ell^\ast}{T^{2/3}} + \frac{\|y_0 - y^\ast(x_0) \|^2}{T^{2/3}} + \frac{\log(T)\sigma_f^2}{T^{2/3}} + \frac{\log(T)\sigma_g^2}{T^{2/3}} \bigg) .
\end{align}
\end{theorem}
% 	Theorem~\ref{Thm: Convergence_NC} shows that to reach an $\epsilon$-stationary point, the \aname~algorithm requires  $\widetilde{\cal O}( \epsilon^{-3/2} )$\footnote{The notation $\widetilde{\cal O}(\cdot)$ omits the logarithmic factors.} samples of stochastic gradient from both the upper and lower level problems.  
Details of the constants in the theorem and its proof can be found in Appendix~\ref{Sec: Appendix_Thm_NC}. 
The above result shows that to reach an $\epsilon$-stationary point, the \aname~algorithm requires  $\widetilde{\cal O}( \epsilon^{-3/2} )$ (omitting logarithmic factors) samples of stochastic gradients from both the upper and lower level functions. 

This sample complexity matches the best complexity bounds for single-level stochastic optimization like SPIDER \cite{Fang_Spider_NIPS_2018}, STORM \cite{Cutkosky_NIPS2019}, SNVRG \cite{Zhou_NIPS_2018_SNVRG} and Hybrid SGD \cite{Dinh_Arxiv_2019}. 
We claim that this is a \emph{near-optimal} sample complexity for bilevel stochastic optimization since for example, we have imposed additional smoothness conditions on the Hessian of the lower level problem. We will leave this as an open question to investigate the lower bound complexity for bilevel stochastic optimization.
% However, it is not clear if the sample complexity we achieved is {\it optimal} for the bilevel problem \eqref{Eq: BilevelProb_Deterministic}. This is because of the following reasons. First, the algorithm has to access the sample Hessian of the lower level problem, and we do additionally require that the  Hessian to be Lipschitz continuous. Although we do not require Hessian information for the outer function $\ell(\cdot)$, it is not clear if this additional information can serve to further improve the sample complexity bound for bilevel problems. Second, it is not clear if the $\mathcal{\tilde{O}}(\epsilon^{-3/2})$ Hessian samples  needed is optimal. In any case, it would be desirable to have a precise lower sample complexity bound for problem \eqref{Eq: BilevelProb_Deterministic}, which characterizes the smallest number of stochastic gradient and Hessian samples required to achieve $\epsilon$-stationarity. We will leave this as an open future research question.
	    
\textbf{Strongly-convex $\ell(x)$.}~~ We also discuss the case when in addition to smoothness, $\ell (\cdot)$ is $\mu_f$-strongly-convex. Here, a stronger guarantee can be obtained: 
\begin{theorem} \label{Thm: Convergence_SC_Fixed}
Under Assumptions \ref{Assump: OuterFunction}--\ref{Assump: Stochastic Grad}, and suppose $\ell(x)$ is $\mu_f$-strongly-convex. Fix any $T \geq 1$, set the number of samples for the gradient estimator \eqref{Eq: UpperLevel_StochasticGrad} as $K = ({L_g} / {2\mu_g}) \log \big(  {C_{g_{xy}}^2 C_{f_y}^2} T / \mu_g^2 \big)$ and 
% the step size parameters 
% $\alpha_t \equiv \alpha$, $\eta_t^f  \equiv (\mu_f + 1) \alpha$, $\beta_t  \equiv \hat{c}_\beta \alpha $, $\eta_t^g \equiv 1$ with $\hat{c}_\beta =    {8 L_y^2  + 8 L^2 + 2 \mu_f}/{\mu_g}$ and 
% 		$\alpha$ chosen such that  
\begin{align*}
\alpha_t \equiv \alpha \leq \bigg\{  \frac{1}{\mu_f + 1}, \frac{1}{2 \mu_g \hat{c}_\beta} ,\frac{\mu_g}{ \hat{c}_\beta L_g^2 }, \frac{1}{8L_K^2 + L_f},  \frac{L^2 + 2 L_y^2}{4 L_K^2 L_g^2 \hat{c}_\beta^2} \bigg\},~~\eta_t^f  \equiv (\mu_f + 1) \alpha,~~\beta_t  \equiv \hat{c}_\beta \alpha,\vspace{.2cm}
\end{align*}
where $\eta_t^g \equiv 1$, $\hat{c}_\beta = {8 L_y^2  + 8 L^2 + 2 \mu_f}/{\mu_g}$ and $L_K$ is defined in \eqref{eq:LKDef}. The iterates generated by Algorithm~\ref{Algo: Accelerated_STSA} satisfy for any $t \geq 1$ that:
\begin{align} \label{eq:strcvx_bd_main}
\mathbb{E}[\ell(x_t) - \ell^\ast] \leq (1 - \mu_f \alpha)^t \bar{\Delta}_0  + \frac{1}{\mu_f} \Big\{ \frac{2}{T} + \big[ (2 \hat{c}_\beta^2 +  8 \hat{c}_\beta^2 L_K^2)  \sigma_g^2 + 2 (\mu_f+1)^2  \sigma_f^2 \big] \alpha \Big\}, 
\end{align}
where $\bar{\Delta}_0 \coloneqq \ell(x_0) - \ell^\ast + \sigma_f^2 + \|y_0 - y^\ast(x_0) \|^2$.
\end{theorem}
The detailed proof can be found in Appendix \ref{Sec: Appendix_Thm_SC}. For large $T$, setting $\alpha \asymp 1/T$ shows that the bound in \eqref{eq:strcvx_bd_main} decreases at the rate of ${\cal O}(1/T)$. 
	
Theorem \ref{Thm: Convergence_SC_Fixed} shows that to reach an $\epsilon$-optimal point, the \aname~algorithm requires $\widetilde{\cal O}(\epsilon^{-1})$ stochastic gradient samples from the upper and lower level problems, also see the detailed calculations in Appendix~\ref{Sec: Appendix_Thm_SC}. This improves over TTSA \cite{Hong_TTSA_Arxiv_2020} which requires $\widetilde{\cal O}(\epsilon^{-1.5})$ samples, and BSA \cite{Ghadimi_BSA_Arxiv_2018} which requires $\widetilde{\cal O}(\epsilon^{-1})$, $\mathcal{O}(\epsilon^{-2})$ samples for the upper and lower level problems, respectively.
Again, we achieve similar sample complexity as SGD applied on strongly-convex single-level optimization. 

Interestingly, in Theorem~\ref{Thm: Convergence_SC_Fixed}, we have selected $\eta_t^g \equiv 1$ where the momentum term in the lower level gradient vanishes. In this way, the \aname~algorithm is reduced into a \emph{single-momentum} algorithm where the recursive momentum acceleration is only applied to the upper level gradient.

\section{Numerical experiments}
\label{Sec: Experiments}
In this section, we supplement the theoretical results presented in Section \ref{sec: Algo} with experiments on real datasets. We demonstrate the efficacy of \aname~for the meta learning \eqref{Eq: Meta_Leaning} and hyperparameter optimization \eqref{eq:clean} tasks. We also examine the performance of \aname~when combined with an Adam-like update rule [cf.~see Algorithm~\ref{Algo: Adam_Direction}] for the meta learning task.  

\begin{figure}[b!]
 \centering
 \includegraphics[width=.45\linewidth]{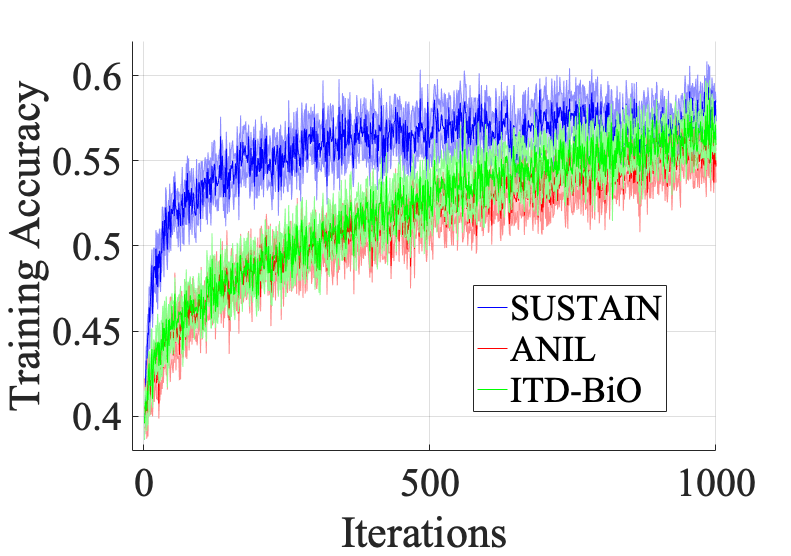} \quad
  \includegraphics[width=.45\linewidth]{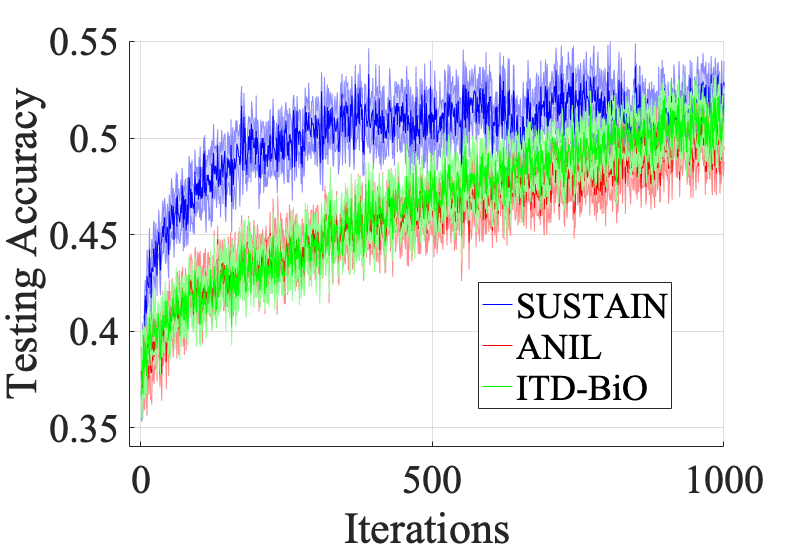}
        % \begin{minipage}[t]{0.44\textwidth}
        % \flushright
        % \includegraphics[width=1\linewidth, height = 1.6 in]{Figures/TrainAcc_Comp_50SD.png}
        % \captionsetup{labelformat=empty}
        % \end{minipage}
        % \hspace{1cm}
        % \begin{minipage}[t]{0.44\textwidth}
        % \flushleft
        % \includegraphics[width=1\linewidth,  height = 1.6 in]{Figures/TestAcc_Comp_50SD.png}
        % \captionsetup{labelformat=empty}
        % \end{minipage}
\caption{\textbf{Meta learning}: 5-way 5-shot learning task on the {\tt miniImageNet} dataset. We plot the training and testing accuracy against the number of iterations.} \label{exper: Meta: mini-image}  
\end{figure}

\textbf{Meta learning.}~~We consider meta learning problem \eqref{Eq: Meta_Leaning} with \texttt{miniImageNet} \cite{Vinyals_miniImageNet_NIPS2016, Ravi_2016_MiniImage} and \texttt{FC100} \cite{Oreshkin_2018_Tadam} datasets. Both  datasets consist of 100 classes with each class containing 600 images. For the \texttt{miniImageNet}, we apply \texttt{learn2learn} \cite{Arnold_l2l_2020} (available: \url{https://github.com/learnables/learn2learn}) to partition the $100$ classes from \texttt{miniImageNet} into subsets of $64$, $16$ and $20$ for meta training, meta validation and meta testing, respectively. For \texttt{FC100}, we follow the setting of \cite{Oreshkin_2018_Tadam, Ji_ProvablyFastBilevel_Arxiv_2020} where 100 classes are split into 60, 20 and 20 classes for meta-training, meta-validation and meta-testing, respectively. For both datasets, we consider a 5-way 5-shot learning task \cite{finn2017model, raghu2019rapid} where the algorithm aims to classify samples into 5 unseen classes using only 5 available samples. We implement the solver using a 4-layer CNN (with different width for each dataset). For both datasets, at each iteration, we sample a batch of $32$ tasks from a set of $20000$ tasks allocated for training and $600$ each for validation and testing.

We first compare the performance of \aname~to ITD-BiO \cite{Ji_ProvablyFastBilevel_Arxiv_2020} and ANIL \cite{raghu2019rapid} for meta learning task on \texttt{miniImageNet} dataset for the vanilla version of the algorithms proposed in respective works\footnote{We excluded MAML \cite{finn2017model} from this set of experiments as its performance with SG based outer update was considerably worse compared to other algorithms.}. For each algorithm, we implement $10$ inner and $1$ outer update and the performance is averaged over $10$ Monte Carlo runs. For ANIL and ITD-BiO, we use the parameter selection suggested in \cite{Arnold_l2l_2020, Ji_ProvablyFastBilevel_Arxiv_2020}. Specifically, for ANIL, we use inner-loop stepsize of $0.1$ and the outer-loop (meta) stepsize as $0.002$. For ITD-BiO, we choose the inner-loop stepsize as $0.05$ and the outer-loop stepsize to be $0.005$. For \aname, we choose the outer-loop stepsize $\alpha_t$ as $\kappa/(1 + t)^{1/3}$ and choose $\kappa \in [0.1, 1]$, we choose the momentum parameter $\eta_t$ as $\bar{c} \alpha_t^2/\kappa^2$ and tune for $\bar{c} \in \{2,5,10,15,20\}$, finally, we fix the inner stepsize as $0.05$. For the inner loop ITD-BiO and \aname~utilize the gradient descent optimizer. 
% From Figure \ref{exper: Meta: mini-image} which compares the training and testing accuracy against the iteration number, we observe that \aname~achieves a better performance compared to ANIL and ITD-BiO on the meta learning task. Also, notice that in the initial iterations \aname~converges faster but then converges probably as a consequence of diminishing stepsizes (and momentum parameter). In contrast, ANIL and ITD-BiO slowly improve in performance and catch up with \aname's performance.
% In the appendix, we show that the \aname~algorithm requires less computation time to achieve better performance compared to the ANIL and ITD-BiO.
% Through the numerical experiments, we establish the following for \aname: (1) \aname~performs well for meta learning and hyperparameter optimization tasks under a variety of settings; (2) it requires the smallest number of samples (as well as gradient computations) compared to the state-of-the-art algorithms; (3) in practice, it can be further combined with Adam to achieve improved performance. 
Figure \ref{exper: Meta: mini-image}, shows that when ITD-BiO and ANIL utilize vanilla SG direction for the outer level update, \aname~outperforms rest of the algorithms for the meta learning problem. Specifically, we compare the training and testing performance of the algorithms with the number of iterations (i.e., the outer update $t$ in Algorithm 1). In each iteration, all the algorithms access the same number of samples while \aname~requiring twice the number of gradient computations (cf. \eqref{Eq: HybridGradEstimate_Outer}). As observed from Figure \ref{exper: Meta: mini-image}, \aname~requires the smallest number of iterations (samples) and gradient computations to achieve a given training/testing accuracy on the benchmarked dataset. Next, we show that the performance of the algorithms can be substantially improved by adapting Adam \cite{kingma2014adam} as the outer optimizer.

% Further, we can also deduce that to achieve some reasonably high accuracy (such as above $50$\% in testing accuracy), \aname~ requires the smallest number of gradient computations.
\textbf{Meta learning (adam \cite{kingma2014adam} based outer update)}~~We conduct additional experiments on the meta learning task and demonstrate the following: (1) for the outer level update we can adapt Adam \cite{kingma2014adam} optimizer with the \aname~framework to achieve better performance, (2) the outer gradient estimate \eqref{Eq: HybridGradEstimate_Outer} for \aname~can be designed with only one gradient computation per iteration (instead of two) without compromising performance, and (3) \aname~outperforms MAML \cite{finn2017model}, ANIL \cite{raghu2019rapid} and ITD-BiO \cite{Ji_ProvablyFastBilevel_Arxiv_2020} when all algorithms implement Adam for the outer level update. 

\begin{figure}[t]
 \centering
 \includegraphics[width=.45\linewidth]{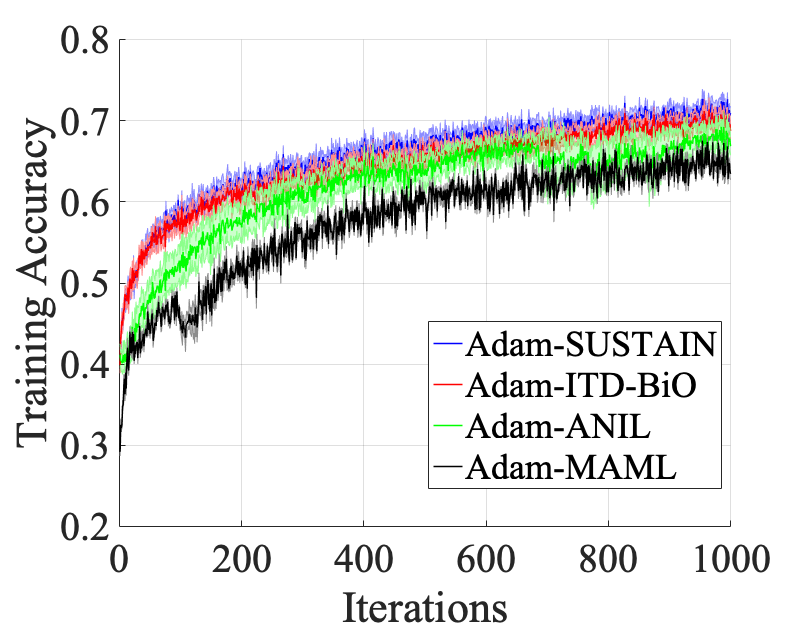} \quad
  \includegraphics[width=.45\linewidth]{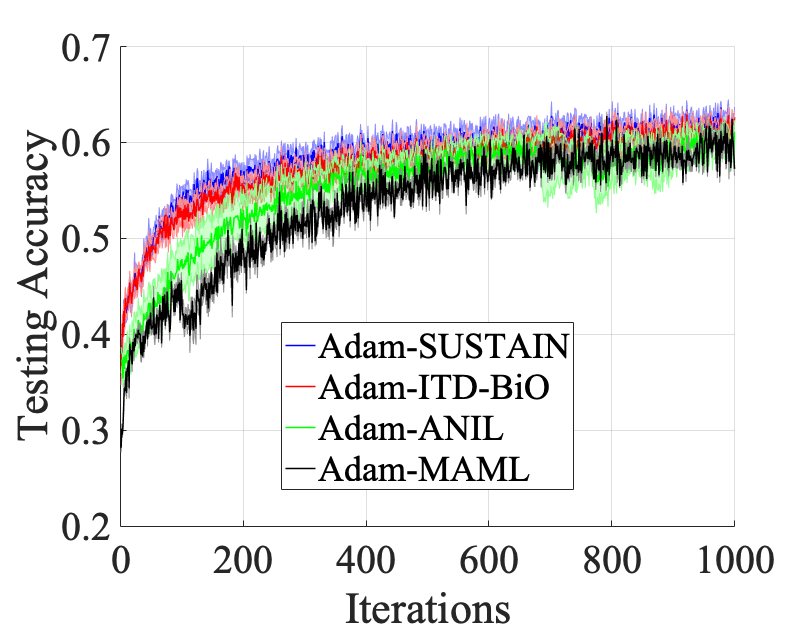}
\caption{\textbf{Meta learning}: 5-way 5-shot learning task on the {\tt miniImageNet} dataset. We plot the training and testing accuracy against the number of iterations with each iteration representing one outer level update step. All the algorithms utilize Adam \cite{kingma2014adam} optimizer for the outer loop update. }
% {\red[define iteration.]}} 
\label{exper: Meta_Adam: mini-image}  
\end{figure} 

\begin{figure}[t]
 \centering
 \includegraphics[width=.45\linewidth]{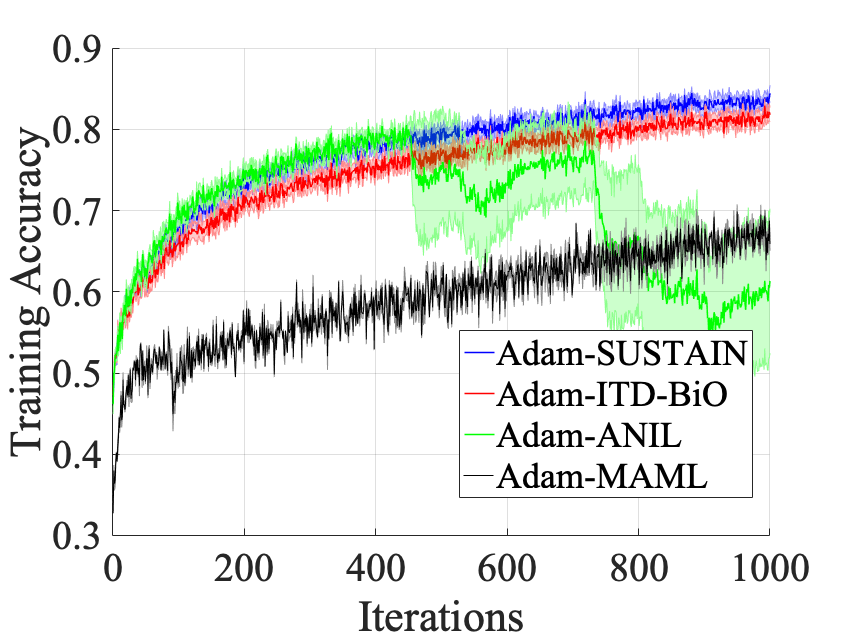} \quad
  \includegraphics[width=.45\linewidth]{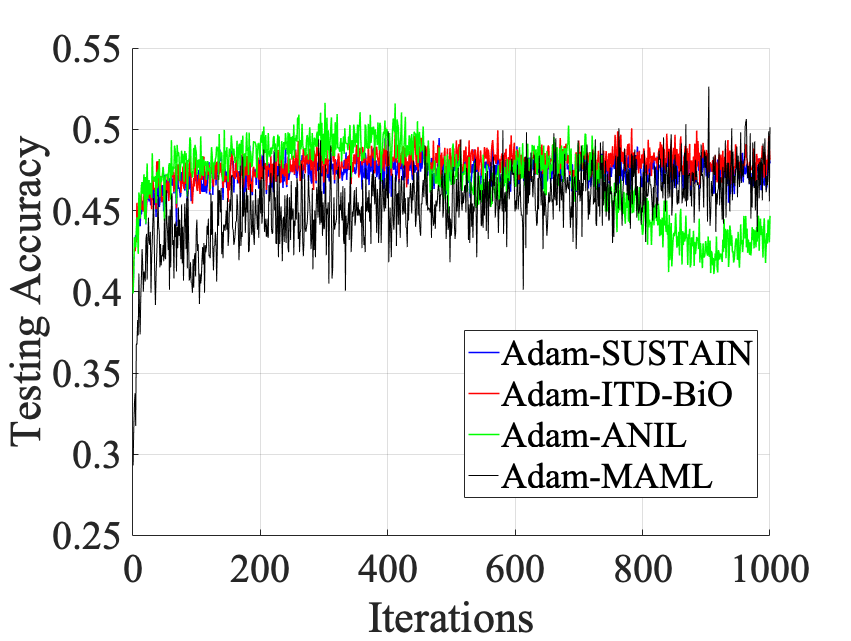}
\caption{\textbf{Meta learning}: 5-way 5-shot learning task on the {\tt FC100} dataset. We plot the training and testing accuracy against the number of iterations with each iteration representing one outer level update. All the algorithms utilize Adam \cite{kingma2014adam} optimizer for the outer loop update.} \label{exper: Meta_Adam: FC100}  
\end{figure}

For this setting, we compare heuristic versions of \aname~with MAML \cite{finn2017model}, ANIL \cite{raghu2019rapid} and recently proposed ITD-BiO \cite{Ji_ProvablyFastBilevel_Arxiv_2020}, where these algorithms all utilize the Adam \cite{kingma2014adam} solver for the outer problem's update. These heuristic algorithms are also used in \cite{Ji_ProvablyFastBilevel_Arxiv_2020} when comparing performance of the bilevel algorithms for meta-learning tasks. Note that these Adam-based bilevel algorithms for meta learning do not have any theoretical performance guarantees. Nevertheless, in the following we show that they perform well in practice \cite{Ji_ProvablyFastBilevel_Arxiv_2020}. 
	
% 	{\red[shall we mention that the rest of the three algorithms are exactly those heuristic algorithms used in Ji et al when comparing performance of the bi-level algorithms for meta-learning?]}

We first discuss the parameter setting for the meta learning task using \texttt{miniImageNet} dataset. For the Adam versions of ANIL and ITD-BiO, we choose the parameters as suggested in \cite{Arnold_l2l_2020,Ji_ProvablyFastBilevel_Arxiv_2020}. For all the algorithms, we execute 10 update steps in the inner loop followed by a single outer update step. Each update step is counted as a single iteration. The implementation of MAML and ANIL is adopted from existing implementations in \cite{Arnold_l2l_2020}. For MAML, we choose the inner loop stepsize to be $0.5$ and the outer loop stepsize to be $0.003$. For ANIL we utilize inner loop stepsize of $0.1$ and outer loop stepsize of $0.002$. Both ITD-BiO and \aname~utilize gradient descent with stepsize of $0.05$ as the inner optimizer. For the outer update ITD-BiO uses a stepsize of $0.002$ (the parameters for ITD-BiO are selected based the repository \url{https://github.com/JunjieYang97/stocBiO}).
% 	{\red[does this choice show the best performance?]}. 
For \aname~we set the outer stepsize as $\alpha_t = 0.005$ and tune for the momentum parameter $\eta_t^f = \bar{c}/\kappa^2 (1 + t)^{2/3}$ with fixed $\kappa = 0.005$ by choosing $\bar{c} \in \{0.25, 2.5, 5, 10\}$. In contrast to other algorithms, \aname~applies Adam \cite{kingma2014adam} to the hybrid stochastic gradient estimator used for the outer update \eqref{Eq: HybridGradEstimate_Outer}. For detailed steps please see Algorithm \ref{Algo: Adam_Direction}\footnote{Note that the vector division and exponent operations in the Algorithm are implemented element wise. The values of the parameters chosen for Adam are default values used by the PyTorch library.}. Moreover, it is worth noting that the direction update rule Option II given in \eqref{Eq: HybridGradEstimate_Outer_1Sample} is a modification of the original update given in \eqref{Eq: HybridGradEstimate_Outer} (or equivalently Option I in \eqref{Eq: HybridGradEstimate_Outer_1Sample}). Such a rule requires just a single (mini-batch) gradient computation per iteration (which is the same as MAML, ANIL and ITD-BiO), and in practice, its performance is very close to that of Option I. Our results below uses Option II as the update direction. 
\begin{align}
  \label{Eq: HybridGradEstimate_Outer_1Sample}
	\bar{h}_{t}^{f}   = \begin{cases}
		\bar{\nabla} f(x_t , y_t ; \bar{\xi}_t) +  (1 - \eta_{t}^f)  \big(\bar{h}_{t - 1}^{f}   -    \bar{\nabla} f(x_{t-1}, y_{t - 1} ; \bar{\xi}_{t}) \big) & \quad \text{Option I}\\
	\bar{\nabla} f(x_t , y_t ; \bar{\xi}_t) +  (1 - \eta_{t}^f)  \big(\bar{h}_{t - 1}^{f}   -  \underbrace{ \bar{\nabla} f(x_{t-1}, y_{t - 1} ; \bar{\xi}_{t-1})}_{\text{Previous SG}}  \big)   & \quad \text{Option II}
	\end{cases}
\end{align}
 %and achieves the same performance in practice as \eqref{Eq: HybridGradEstimate_Outer}. 
	
	\begin{algorithm}[t]
\caption{Update direction for Adam-\aname~(also see footnote$^1$)} \label{Algo: Adam_Direction}
\begin{algorithmic}[1]
\State{\textbf{Parameters}: $\gamma_1 = 0.9$, $\gamma_2 = 0.999$, $m_0 = 0$, $v_0 = 0$, $\epsilon = 10^{-8}$ and $\eta_t^f$}
\For{$t = 1, \cdots, T$}
\State{\textbf{Input}: $(x_t, y_t)$, $(x_{t-1}, y_{t-1})$ from Algorithm \ref{Algo: Accelerated_STSA}.}
	\State{Compute the gradient estimator $\bar{h}^f_t$ using Option I or II in \eqref{Eq: HybridGradEstimate_Outer_1Sample} }
	\State{Update first moment estimate: $m_t \leftarrow \gamma_1 \cdot m_{t-1} + (1 - \gamma_1) \bar{h}_t^f$}
		\State{Bias-correction for first moment estimate: $m_t \leftarrow  m_t/(1 - (\gamma_1)^t)$}
	\State{Update second moment estimate: $v_t \leftarrow \gamma_2 \cdot v_{t-1} + (1 - \gamma_2) (\bar{h}_t^f)^2$}
	\State{Bias-correction for second moment estimate: $v_t \leftarrow  v_t/(1 - (\gamma_2)^t)$}
\State{Use the update direction: $h^f_t \leftarrow  m_t/ (\sqrt{v_t} + \epsilon)$}
\EndFor
\State{{\bf Return:} $h^f_t$}
\end{algorithmic}
\end{algorithm}
	
In Figure \ref{exper: Meta_Adam: mini-image}, we plot the training and testing performance against the number of iterations for the Adam version \aname~with other algorithms for 5-way 5-shot learning task on \texttt{miniImageNet} dataset. Note from the discussion above, we know that in each iteration all the algorithms access the same number of sample, and spend the same amount of (mini-batch) gradient computation efforts. Consequently, Figure \ref{exper: Meta_Adam: mini-image} implies that \aname~outperforms ITD-BiO, ANIL and MAML as it requires fewest iterations (thus samples and gradient computation) to achieve the improved performance. Importantly, these Adam-based algorithms  significantly outperform their vanilla version (cf. Figure \ref{exper: Meta: mini-image} for performance with SGD), in terms of both accuracy and speed. 

% 	{\red[also mention that in our proposed algorithm, one stochastic sample is needed?]}

 Next, we compare the performance of \aname~with other algorithms for the meta learning task using \texttt{FC100} dataset. In contrast to the previous dataset, for this task we execute 20 update steps in the inner loop followed by a single outer update step. Similar to \texttt{miniImageNet} dataset, we adopt existing implementations of MAML and ANIL from \cite{Arnold_l2l_2020} and ITD-BiO from \cite{Ji_ProvablyFastBilevel_Arxiv_2020}. For MAML, we choose inner loop stepsize of 0.5 and the outer loop stepsize of 0.001. For ANIL we utilize inner loop stepsize of 0.1 and outer loop stepsize of 0.001. In the inner loop, both ITD-BiO and \aname~utilize gradient descent with a stepsize of 0.1. For the outer update ITD-BiO uses a stepsize of 0.001 (the parameters for ITD-BiO are selected based the repository \url{https://github.com/JunjieYang97/stocBiO}).
%  {\red[does this choice give the best performance for itd-bio?]}. 
 For the outer update \aname~utilizes the same setting as required for \texttt{miniImageNet} dataset and the Adam based outer update direction as computed in Algorithm \ref{Algo: Adam_Direction}. In Figure \ref{exper: Meta_Adam: FC100}, we plot the training and testing performance with the number of iterations for \aname~and other algorithms for 5-way 5-shot learning task on \texttt{FC100} dataset. Note that \aname~outperforms rest of the algorithms on the training task and performs on par with other algorithms with respect to the testing performance. Moreover, note that initially ANIL performs better but since the number of inner steps are relatively large (20 in this case), ANIL's performance degrades after a certain number of iterations. Similar behavior was noted for ANIL in the results of \cite{Ji_ProvablyFastBilevel_Arxiv_2020}. 
 	
 The above set of experiments showed that the Adam \cite{kingma2014adam} optimizer can be incorporated with \aname~and other algorithms to achieve improved performance compared to vanilla SG based algorithms. We also showed that the gradient estimator for \aname~can be modified to require only single (batch) gradient evaluation per iteration (cf. \eqref{Eq: HybridGradEstimate_Outer_1Sample}) without comprising performance of the algorithm. The experiments demonstrate that under most settings \aname~outperforms other state-of-the-art algorithms.  
 
 Next, we evaluate the performance of \aname~on a Hyperparameter optimization task. 

\begin{figure}[b] \centering
\includegraphics[width=.425\linewidth]{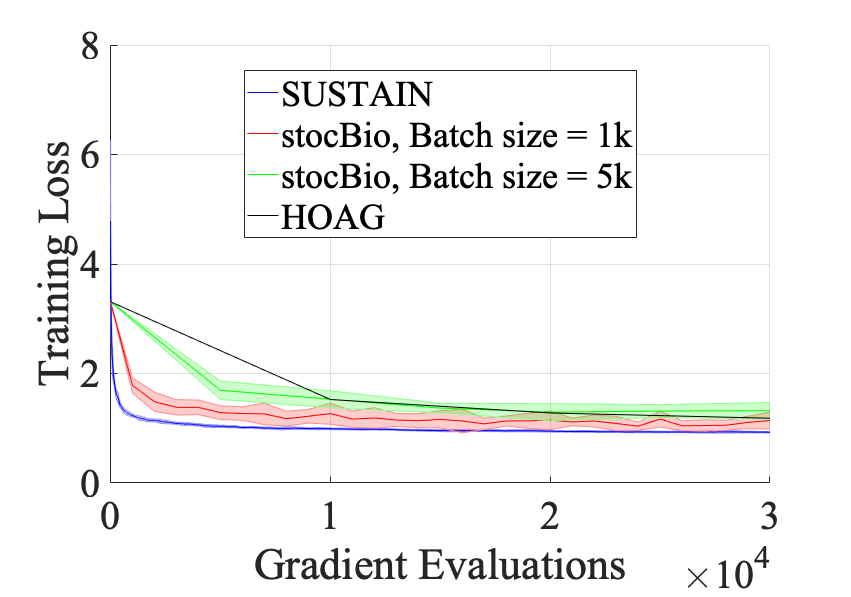} \quad 
\includegraphics[width=.425\linewidth]{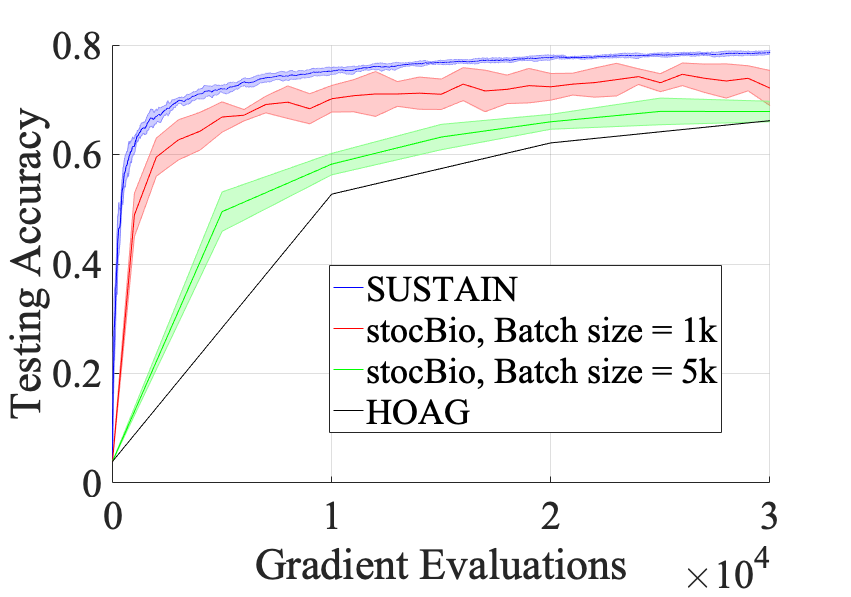}
% \vspace{-0.5cm}
% \begin{minipage}[t]{0.45\textwidth}
% \flushright
% \includegraphics[width=1\linewidth,height = 1.6 in]{Figures/P03_FMNIST_TrainingLoss_Comp.png}
% \captionsetup{labelformat=empty}
% \end{minipage}
% \hspace{1cm}
% \begin{minipage}[t]{0.45\textwidth}
% \flushleft
% \includegraphics[width=1\linewidth,height = 1.6 in]{Figures/P03_FMNIST_TestAcc_Comp.png}
% \captionsetup{labelformat=empty}
% \end{minipage}
\caption{\textbf{Hyperparameter optimization}: Data hyper-cleaning task on the {\tt FashionMNIST} dataset. We plot the training loss and testing accuracy against the number of gradients evaluated with corruption rate $p = 0.3$.}
\label{exper: data-cleaning:1}
\end{figure}

\textbf{Hyperparameter optimization.}~~
We consider the data hyper-cleaning task \eqref{eq:clean}, and compare \aname~with several algorithms such as stocBiO \cite{Ji_ProvablyFastBilevel_Arxiv_2020} for different batch size choices, and the HOAG algorithm in \cite{Pedregosa_ICML_2016}. Note that in \cite{Ji_ProvablyFastBilevel_Arxiv_2020}, the authors have shown that stocBio exhibits better practical performance compared with other bilevel optimization algorithms. Importantly, in this section we demonstrate that \aname~performs well under different levels of data corruption.

We consider hyper-cleaning task \eqref{eq:clean} on {\tt Fashion-MNIST} dataset \cite{FashionMNIST_Xiao_2017} with $L(\cdot)$ being the cross-entropy loss (i.e., a data cleaning problem for logistic regression); $ \sigma(x) := \frac{1}{1 + \exp(-x)} $ and $c = 0.001$; see \cite{Shaban_TruncatedBackProp_2019}. The problem is trained on the {\tt FashionMNIST} dataset \cite{FashionMNIST_Xiao_2017} with $50$k, $10$k, and $10$k image samples allocated for training, validation and testing purposes, respectively. We consider two levels of corruption, namely 30$\%$ and 40$\%$ corruption rate. Note that HOAG is a deterministic algorithm and requires full gradient computation at each iteration. In contrast, stocBiO is a stochastic algorithm but it relies on large batch gradient computations. We conduct experiments for two settings where stocBiO uses a batch size of 5000 and 1000 (for both inner and outer updates). Our algorithm \aname~is purely a stochastic algorithm and does not rely on large batch gradient computations. Specifically, \aname~computes two gradients (on a single sample) in each iteration for both inner and outer updates (cf. \eqref{Eq: HybridGradEstimate_Inner} and \eqref{Eq: HybridGradEstimate_Outer}). Since at each outer iteration, the  sample sizes (and gradient computations) accessed by each algorithm are very different, so it is no longer fair to compare the per-iteration performance for different algorithms (this is different compared with the meta learning example in the previous section). Therefore, in this section we compare the training and testing performance of the competing algorithms using the number of total outer gradient computations (which is same as the inner gradient computations) across iterations. Note that for HOAG and stocBiO, the number of samples accessed is same as the number of gradient evaluations, whereas for \aname~we compute two gradients for each sample accessed (cf. \eqref{Eq: HybridGradEstimate_Outer})\footnote{Note that this requirement can be easily relaxed without compromising performance via using the gradient construction \eqref{Eq: HybridGradEstimate_Outer_1Sample}.}. The step sizes for different algorithms are chosen according to their theoretically suggested values. Let the outer iteration be indexed by $t$, for \aname~we choose $\alpha_t = \beta_t = 0.1/(1 + t)^{1/3}$ and tune for $c_{\eta_f}$ and $c_{\eta_g}$ (see Theorem \ref{Thm: Convergence_NC}), for stocBiO and HOAG we select $\alpha_t= d_\alpha$, $\beta_t= d_\beta$ and tune for parameters $d_\alpha$ and $d_\alpha$ in the range $[0, 1]$. 

	\begin{figure}[t]
\centering
\includegraphics[width=0.45\linewidth]{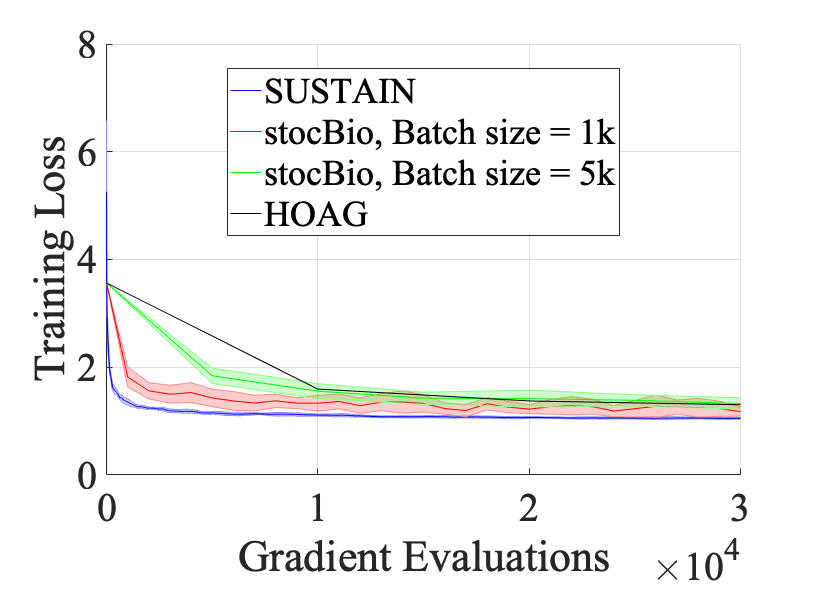}
\quad
\includegraphics[width=0.45\linewidth]{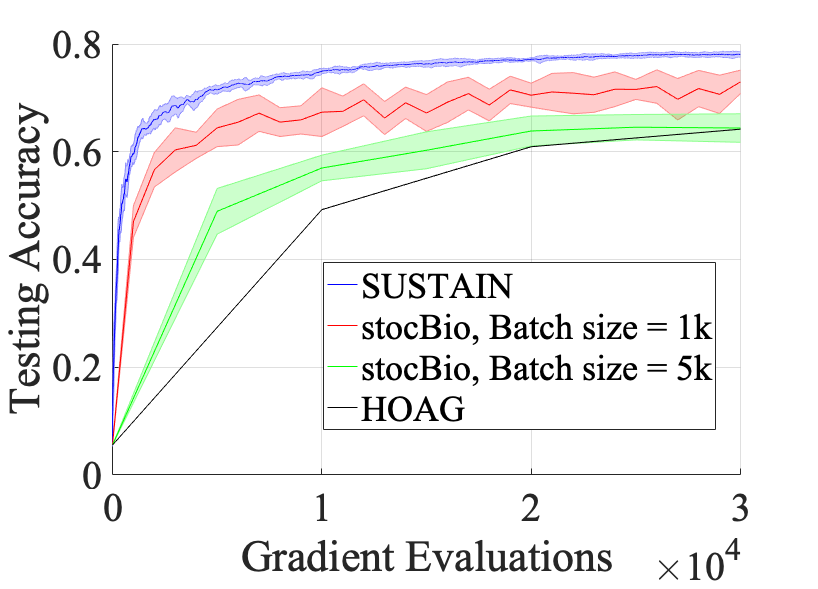}
\caption{Data hyperparameter optimization: Training loss and testing accuracy against the number of gradients evaluated with corruption rate $p = 0.4$. }
\label{exper: data-cleaning:2}  
\end{figure} 

In Figures \ref{exper: data-cleaning:1} and \ref{exper: data-cleaning:2}, we compare the performance of different algorithms when the dataset has a corruption probability of $0.3$ and $0.4$, respectively. The experiments establish that \aname~outperforms HOAG and stocBiO, in terms of the total number of gradient evaluations as well as the number of samples. We remark that relatively large batch sizes used by HOAG and stocBio  result in relatively slow convergence. Moreover, this fast convergence of \aname~results form the single timescale update with reduced variance resulting from the double-momentum variance reduced updates.

	\section{Conclusion and future work} 
	We have developed the \aname~algorithm for unconstrained bilevel optimization with strongly convex lower level subproblems. The proposed algorithm executes on a single-timescale, without the need to use either two-timescale updates, large batch gradients, or double-loop algorithm. 
We showed that \aname~ is both {\it sample} and {\it computation} efficient, because it matches the best-known sample complexity guarantees on single-level problems with non-convex and strongly convex objective functions, while matching the best-known per-iteration computational complexity for the same class of bi-level problems. In the future, we plan to rigorously show the sample complexity lower bound for the considered class of bilevel problems. Further, we plan to develop sample and communication efficient algorithms for a more general class of bilevel problems, such as those with constraints in the lower level problem. 
	
\newpage
\bibliographystyle{IEEEtran}
\bibliography{abrv,References,ref-bi}

	\newpage
	\appendix
	
\section*{Appendix}

	Now we present the proofs of the theoretical results.
	\section{Useful lemmas}
	\label{Sec: Appendix_UsefulLemmas}
	\begin{lem}
		\label{Lem: Lip_Product}
		Consider a collection of functions $\Phi_i : \mathbb{R}^n \to \mathcal{Z}$ with $i = \{1,2,\ldots, k\}$ and $\mathcal{Z} \subseteq \mathbb{R}^{n\times n}$, which satisfy the following assumptions:
		\begin{enumerate}[label=(\roman*)]
			\item There exist $L_i > 0, \; i\in [k]$, such that  
			\begin{align*}
			\|\Phi_i(x)  - \Phi_i(y)\| \leq L_i \|x - y\|, \; \forall~i\in[k], ~x, y \in \mathbb{R}^n.
			\end{align*}
			\item For each $i \in [N]$ and $k \in \mathbb{N}$ we have $\| \Phi_i(x) \| \leq M_i$ for all $x \in \mathbb{R}^n$.
		\end{enumerate}
		Then the following holds for all $x, y \in \mathbb{R}^{n}$:
		\begin{align}
		\label{Eq: Lip_Prod_Fixed_k}
		\bigg\| \prod_{i=1}^k \Phi_i(x) -  \prod_{i=1}^k \Phi_i(y) \bigg\|^2      \leq  k \sum_{i=1}^k \Big(\prod_{j=1, j \neq i}^k M_j \Big)^2  L_i^2 \|x - y \|^2.
		\end{align}
		Moreover, if $k$ is generated uniformly at random from $\{0,1, \ldots, K-1\}$, then the following holds for all $x, y \in \mathbb{R}^{n}$:
		\begin{align}
		\label{Eq: Lip_Prod_Random_k}
		\mathbb{E}_k   \bigg\| \prod_{i=1}^k \Phi_i(x) -  \prod_{i=1}^k \Phi_i(y) \bigg\|^2    \leq K \sum_{i=1}^K \mathbb{E}_k \bigg[ \Big(\prod_{j=1, j \neq i}^k M_j \Big)^2 \bigg]  L_i^2 \|x - y \|^2.
		\end{align}
		Here we use the convention that $\prod_{i=1}^k \Phi_i(x) = I$ if $k = 0$.
	\end{lem}
	\begin{proof}
		We first prove \eqref{Eq: Lip_Prod_Fixed_k}. To do so we will first show that the following holds for all $x , y \in \mathbb{R}^n$ and $k \in \mathbb{N}$:
		\begin{align}
		\label{Eq: Lip_Prod_Claim}
		\bigg\| \prod_{i=1}^k \Phi_i(x) -  \prod_{i=1}^k \Phi_i(y) \bigg\|      \leq  \sum_{i=1}^k \Big(\prod_{j=1, j \neq i}^k M_j \Big)  L_i \|x - y \|,
		\end{align}
		Then by combining the above result with the identity that  
		\begin{align}
		\label{Eq: Sum_Norms}
		\|z_1 + z_2 + \ldots + z_k\|^2 \leq k \|z_1 \|^2 + k\|z_2\|^2 + \ldots + k\|z_k\|^2, \; \text{for all}~z, \; k\in \mathbb{N},
		\end{align}
		we can conclude the first statement.

		To show \eqref{Eq: Lip_Prod_Claim}, we use an induction argument. The base case for $k = 1$ holds because of the Lipschitz assumption $(i)$ given in the statement of the lemma. Then assuming claim \eqref{Eq: Lip_Prod_Claim} holds for arbitrary $k$, we have for $k+1$  \begin{align*}
	&	\bigg\| \prod_{i=1}^{k+1} \Phi_i(x) -  \prod_{i=1}^{k+1} \Phi_i(y) \bigg\|   =    \bigg\| \prod_{i=1}^{k+1} \Phi_i(x) - \prod_{i=1}^{k} \Phi_i(x) \Phi_{k+1}(y) + \prod_{i=1}^{k} \Phi_i(x) \Phi_{k+1}(y)  -  \prod_{i=1}^{k+1} \Phi_i(y) \bigg\|\\
		& \qquad \qquad \overset{(a)}{\leq} \bigg\| \prod_{i=1}^{k} \Phi_i(x) \bigg\|~ \big\| \Phi_{k+1}(x) - \Phi_{k+1}(y) \big\|  +  \big\|\Phi_{k+1}(y) \big\| ~ \bigg\| \prod_{i=1}^{k} \Phi_i(x)   -  \prod_{i=1}^k \Phi_i(y) \bigg\| \\
		& \qquad \qquad \overset{(b)}{\leq} \bigg( \prod_{j = 1}^k M_j \bigg) L_{k+1} \big\|x - y \big\| +    \sum_{i=1}^k \Big(\prod_{j=1, j \neq i}^{k+1} M_j \Big)  L_i \big\|x - y \big\|\\
		& \qquad \qquad \overset{(c)}{\leq} \sum_{i=1}^{k+1} \Big(\prod_{j=1, j \neq i}^{k+1} M_j \Big)  L_i \|x - y \|.
		\end{align*}
		where $(a)$ follows from the application of the triangle inequality and the Cauchy-Schwartz inequality; the first expression in $(b)$ results from the application of Cauchy-Schwartz inequality and Assumption (i) and (ii) of the statement of the lemma; the second expression in $(b)$ follows from the assumption that claim \eqref{Eq: Lip_Prod_Claim} holds for $k$; $(c)$ follows from combining the two expressions. We conclude that \eqref{Eq: Lip_Prod_Claim} holds for all $k \in \mathbb{N}$. 
		
		Now consider the case when $k$ is chosen uniformly at random from $k \in \{0,1, \ldots, K - 1\}$. First, note from the definition that for $k = 0$ we have $\prod_{i=1}^k \Phi_i(x) = I$. This implies that \eqref{Eq: Lip_Prod_Fixed_k} is also satisfied if we have $k = 0$.  We then have
		\begin{align*}
		\mathbb{E}_k  \bigg\| \prod_{i=1}^k \Phi_i(x) -  \prod_{i=1}^k \Phi_i(y) \bigg\|^2 & \overset{(a)}{\leq} \mathbb{E}_k \bigg[ k \sum_{i=1}^k \Big(\prod_{j=1, j \neq i}^k M_j \Big)^2  \|\Phi_i(x) - \Phi_i(y) \|^2 \bigg] \\
		& \overset{(b)}{\leq} K \sum_{i=1}^K \mathbb{E}_k \bigg[ \Big(\prod_{j=1, j \neq i}^k M_j \Big)^2 \bigg]  \|\Phi_i(x) - \Phi_i(y) \|^2 \\
		& \overset{(c)}{\leq} K \sum_{i=1}^K \mathbb{E}_k \bigg[ \Big(\prod_{j=1, j \neq i}^k M_j \Big)^2 \bigg]  L_i^2 \|x - y\|^2. 
		\end{align*}
		where $(a)$ uses the fact that \eqref{Eq: Lip_Prod_Fixed_k} holds for all $k \in \{0,1, \ldots, K-1\}$ almost surely; $(b)$ follows from the fact that $k \leq K$ almost surely; $(c)$ results from Assumption $(i)$ of the lemma. 
		%The lemma is proved. 
	\end{proof}
	
	\section{Proofs of preliminary lemmas}
	\label{Sec: Appendix_PreliminaryLemmas}
	\subsection{Estimation of the stochastic gradient}
	\label{Subsec: EstimationGrad}
	We construct the stochastic gradient $\bar{\nabla}f(x,y; \bar{\xi})$ as \cite{Ghadimi_BSA_Arxiv_2018, Hong_TTSA_Arxiv_2020}:
	\begin{enumerate}[leftmargin = 0.8 cm]
		\item For $K \in \mathbb{N}$, choose $k \in \{0,1, \ldots, K-1\}$ uniformly at random.
		\item Compute unbiased Hessian approximations $\nabla^2_{xy} g(x,y ; \zeta^{(0)})$ and $\nabla^2_{yy} g(x,y; \zeta^{(i)})$ for $i \in \{1, \ldots, k\}$, where $\{\zeta^{(i)}\}_{i = 0}^k$ are chosen independently.
		\item Compute unbiased gradient approximations $\nabla_x f(x,y ;  \xi)$ and $\nabla_y f(x,y ;  \xi)$ where $ \xi$ is chosen independently of $\{\zeta^{(i)}\}_{i = 0}^k$.
		\item Construct the stochastic gradient estimate $\bar{\nabla} f(x, y ; \bar{\xi})$ with $\bar{\xi}$ denoted as $\bar{\xi} = \{ \xi, \{\zeta^{(i)}\}_{i = 0}^k\}$:
		\begin{align}
		\label{Eq: UpperLevel_StochasticGrad_Appendix}
		& \bar{\nabla} f(x, y ; \bar{\xi}) \nonumber\\
		& = \nabla_x f(x,y ;  \xi) - \nabla^2_{xy} g(x,y ; \zeta^{(0)}) \bigg[ \frac{K}{L_g} \prod_{i = 1}^k \bigg( I - \frac{1}{L_g} \nabla^2_{yy} g(x, y ; \zeta^{(i)}) \bigg) \bigg] \nabla_y f(x, y ;  \xi),
		\end{align}
		with $\prod_{i = 1}^k \bigg( I - \frac{1}{L_g} \nabla^2_{yy} g(x, y ; \zeta^{(i)}) \bigg) = I$ if $k = 0$.
	\end{enumerate}

	Next, we state the result showing that the bias of the stochastic gradient estimate of the upper level objective defined in \eqref{Eq: UpperLevel_StochasticGrad} decays linearly with the number of samples $K$ chosen to approximate the Hessian inverse. 
	\begin{lem}{\cite[Lemma 11]{Hong_TTSA_Arxiv_2020}} 
		\label{Lem: Bias}
		Under Assumptions \ref{Assump: OuterFunction}, \ref{Assump: InnerFunction} and \ref{Assump: StocFn} the stochastic gradient estimate of the upper level objective defined in \eqref{Eq: UpperLevel_StochasticGrad_Appendix}, satisfies 
		\begin{align*}
		\| B(x,y)\| & = \|\bar{\nabla}f(x , y) - \mathbb{E}[\bar{\nabla}f(x , y ; \bar{\xi})] \| \leq \frac{C_{g_{xy}} C_{f_y}}{\mu_g} \bigg( 1 - \frac{\mu_g}{L_g} \bigg)^K, 
		\end{align*}
		where $B(x,y)$ is the bias of the stochastic gradient estimate and $K$ is the number of samples chosen to approximate the Hessian inverse in \eqref{Eq: UpperLevel_StochasticGrad_Appendix}.  Moreover, if we assume, 
		\begin{align*}
		& \qquad \qquad \mathbb{E}[\|\nabla_y f(x, y ;\xi^{(1)}) \|^2] \leq C_y, ~~~~\mathbb{E}[\|\nabla_{xy}^2 g(x, y ;\xi^{(2)}) \|^2] \leq C_g, \\
		& \qquad \qquad \quad \qquad \mathbb{E}\big[ \| \nabla_x f(x,y) - \nabla_x f(x, y ; \xi^{(1)}) \|^2 \big] \leq \sigma_{f_x}^2, \\
		&	  \mathbb{E}\| \nabla_y f(x,y) - \nabla_y f(x, y; \xi^{(1)}) \|^2 \leq \sigma_{f_y}^2, ~~
		\mathbb{E}\| \nabla_{xy}^2 g(x,y) - \nabla_{xy}^2 g(x, y; \xi^{(2)}) \|^2 \leq \sigma_{g_{xy}}^2 
		\end{align*}
		Then we have
		\begin{align*}
		\mathbb{E}_{\bar{\xi}}\big[ \big\| \bar{\nabla} f(x,y) - \mathbb{E}_{\bar{\xi}} [\bar{\nabla} f(x, y; \bar{\xi})] \big\|^2 \big] \leq \sigma_{f_x}^2 + \frac{3}{\mu_g^2} \Big[ (\sigma_{f_y}^2 + C_y^2) \big(\sigma^2_{g_{xy}} + 2 C_{g_{xy}}^2 \big) + \sigma_{f_y}^2 C_{g_{xy}}^2 \Big].
		\end{align*} 
	\end{lem}
Lemma \ref{Lem: Bias} implies that the bias $B(x,y)$ can be made to satisfy $\|B(x,y)\| \leq \epsilon$ with only 
	$$K = (L_g/\mu_g)\log (C_{g_{xy}} C_{f_y} / \mu_g \epsilon)$$ stochastic Hessian samples of $\nabla^2_{yy} g(x, y)$.
	
	\subsection{Lipschitz continuity of gradient estimate}
	\begin{lem}[Lipschitzness of Stochastic Gradient Estimate]
		\label{Lem: Lip_GradEst_Appendix}
		If the stochastic functions $f(x, y;\xi)$ and $g(x,y; \zeta)$ satisfy Assumptions \ref{Assump: OuterFunction}, \ref{Assump: InnerFunction} and \ref{Assump: StocFn}, then we have 
		\begin{enumerate}[label=(\roman*)]
			\item For a fixed $y \in \mathbb{R}^{\du}$ 
			\begin{align*}
			\mathbb{E}_{\bar{\xi}}   \| \bar{\nabla} f(x_1, y ; \bar{\xi}) - \bar{\nabla} f(x_2, y ; \bar{\xi}) \|^2 \leq L_K^2 \|x_1 - x_2\|^2,~\forall~x_1,x_2 \in \mathbb{R}^{\du}.
			\end{align*}
			\item For a fixed $x \in \mathbb{R}^{\du}$ 
			\begin{align*}
			\mathbb{E}_{\bar{\xi}} \| \bar{\nabla} f(x, y_1 ; \bar{\xi})  - \bar{\nabla} f(x, y_2 ; \bar{\xi})   \|^2 \leq L_K^2 \|y_1 - y_2\|^2,~\forall~y_1,y_2 \in \mathbb{R}^{\du}.
			\end{align*}
		\end{enumerate}
		In the above expressions, $L_K > 0$ is defined as:
		\begin{align*}
		L_K^2 = 2 L_{f_x}^2 + 6 C_{g_{xy}}^2 L_{f_y}^2 \bigg( \frac{K}{2 \mu_g L_g - \mu_g^2}\bigg)& + 6 C_{f_y}^2 L_{g_{xy}}^2 \bigg( \frac{K}{2 \mu_g L_g - \mu_g^2} \bigg) \\
		& \quad + 6 C_{g_{xy}}^2 C_{f_y}^2 \frac{K^3 L_g^2}{(L_g - \mu_g)^2 (2 \mu_g L_g - \mu_g^2)},
		\end{align*}
		and where $K$ is the number of samples required to construct the stochastic approximation of $\bar{\nabla}f$ (see \eqref{Eq: UpperLevel_StochasticGrad_Appendix} above).
	\end{lem}
	\begin{proof}
		We prove only statement $(i)$ of the lemma, the proof of $(ii)$ follows from a similar argument. From the definition of $\bar{\nabla} f(x_1, y ; \bar{\xi})$ we have for $x_1, x_2 \in \mathbb{R}^{\du}$ and $y \in \mathbb{R}^{\du}$
		\begin{align}
		& \| \bar{\nabla} f(x_1, y ; \bar{\xi}) - \bar{\nabla} f(x_2, y ; \bar{\xi}) \|^2 \nonumber \\
		& \overset{(a)}{\leq} 2   \big\| \nabla_x f(x_1,y ;  \xi) - \nabla_x f(x_2,y ;  \xi) \big\|^2 \nonumber\\
		& +  2   \bigg\| \nabla^2_{xy} g(x_1,y ; \zeta^{(0)}) \bigg[ \frac{K}{L_g} \prod_{i = 1}^k \bigg( I - \frac{1}{L_g} \nabla^2_{yy} g(x_1, y ; \zeta^{(i)}) \bigg) \bigg] \nabla_y f(x_1, y ;  \xi) \nonumber\\
		& \qquad\qquad  - \nabla^2_{xy} g(x_2,y ; \zeta^{(0)}) \bigg[ \frac{K}{L_g} \prod_{i = 1}^k \bigg( I - \frac{1}{L_g} \nabla^2_{yy} g(x_2, y ; \zeta^{(i)}) \bigg) \bigg] \nabla_y f(x_2, y ;  \xi) \bigg\|^2 \nonumber \\
		& \overset{(b)}{\leq} 2 L_{f_x}^2  \big\| x_1 - x_2\big\|^2 \nonumber\\
		& +  2   \bigg\| \nabla^2_{xy} g(x_1,y ; \zeta^{(0)}) \bigg[ \frac{K}{L_g} \prod_{i = 1}^k \bigg( I - \frac{1}{L_g} \nabla^2_{yy} g(x_1, y ; \zeta^{(i)}) \bigg) \bigg] \nabla_y f(x_1, y ;  \xi) \nonumber\\
		& \qquad\qquad  - \nabla^2_{xy} g(x_2,y ; \zeta^{(0)}) \bigg[ \frac{K}{L_g} \prod_{i = 1}^k \bigg( I - \frac{1}{L_g} \nabla^2_{yy} g(x_2, y ; \zeta^{(i)}) \bigg) \bigg] \nabla_y f(x_2, y ;  \xi) \bigg\|^2,
		\label{Eq: Lip_GradEst_1st}
		\end{align}
		where inequality $(a)$ follows from the definition of $\bar{\nabla} f(x_1, y ; \bar{\xi})$ and \eqref{Eq: Sum_Norms}; inequality $(b)$ follows from the Lipschitz-ness
		Assumption \ref{Assump: OuterFunction}--(ii)
		made for stochastic upper level objective. The variable $k \in \{0, \ldots, K-1 \}$ above is a random variable define in Section \ref{Subsec: EstimationGrad} above. Let us consider the second term of \eqref{Eq: Lip_GradEst_1st} above, we have
		\begingroup
		\allowdisplaybreaks
		\begin{align*}
		&  \bigg\| \nabla^2_{xy} g(x_1,y ; \zeta^{(0)}) \bigg[ \frac{K}{L_g} \prod_{i = 1}^k \bigg( I - \frac{1}{L_g} \nabla^2_{yy} g(x_1, y ; \zeta^{(i)}) \bigg) \bigg] \nabla_y f(x_1, y ;  \xi) \\
		& \qquad\qquad \qquad \qquad - \nabla^2_{xy} g(x_2,y ; \zeta^{(0)}) \bigg[ \frac{K}{L_g} \prod_{i = 1}^k \bigg( I - \frac{1}{L_g} \nabla^2_{yy} g(x_2, y ; \zeta^{(i)}) \bigg) \bigg] \nabla_y f(x_2, y ;  \xi) \bigg\|^2  \\
		& \overset{(a)}{\leq} 3 C_{g_{xy}}^2 \frac{K^2}{L_g^2}\bigg(1 - \frac{\mu_g}{L_g} \bigg)^{2k} \| \nabla_y f(x_1, y;  \xi) - \nabla_y f(x_2, y;  \xi) \|^2 \\
		& \qquad   + 3  C_{f_y}^2 \frac{K^2}{L_g^2}\bigg(1 - \frac{\mu_g}{L_g} \bigg)^{2k}  \| \nabla_{xy}^2 g(x_1, y ; \zeta^{(0)}) - \nabla_{xy}^2 g(x_2, y ; \zeta^{(0)}) \|^2  \\
		& \qquad   + 3 C_{g_{xy}}^2 C_{f_y}^2   \bigg\|\frac{K}{L_g} \prod_{i = 1}^k \bigg( I - \frac{1}{L_g} \nabla^2_{yy} g(x_1, y ; \zeta^{(i)}) \bigg) - \frac{K}{L_g} \prod_{i = 1}^k \bigg( I - \frac{1}{L_g} \nabla^2_{yy} g(x_2, y ; \zeta^{(i)}) \bigg) \bigg\|^2 \\
		& \overset{(b)}{\leq} 3 C_{g_{xy}}^2 \frac{K^2}{L_g^2}\bigg(1 - \frac{\mu_g}{L_g} \bigg)^{2k}  L_{f_y}^2 \|x_1 - x_2\|^2 + 3 C_{f_y}^2 \frac{K^2}{L_g^2}\bigg(1 - \frac{\mu_g}{L_g} \bigg)^{2k}  L_{g_{xy}}^2 \| x_1 - x_2\|^2\\
		& \qquad \quad   + 3 C_{g_{xy}}^2 C_{f_y}^2 \frac{K^2}{L_g^2}    \bigg\|  \prod_{i = 1}^k \bigg( I - \frac{1}{L_g} \nabla^2_{yy} g(x_1, y ; \zeta^{(i)}) \bigg) -   \prod_{i = 1}^k \bigg( I - \frac{1}{L_g} \nabla^2_{yy} g(x_2, y ; \zeta^{(i)}) \bigg) \bigg\|^2,
		\end{align*}
		\endgroup
		where inequality $(a)$ follows from \eqref{Eq: Lip_Prod_Fixed_k} in Lemma \ref{Lem: Lip_Product}, Assumption \ref{Assump: OuterFunction}--(iii) and Assumption \ref{Assump: InnerFunction}--(ii)(iii)(vi); inequality $(b)$ follows from the Lipschitz continuity Assumption \ref{Assump: OuterFunction}--(ii) and Assumption \ref{Assump: InnerFunction}--(v) made for the stochastic upper and lower level objectives. On both sides taking expectation w.r.t $k$, we get:
		\begin{align}
		&  \mathbb{E}_k \bigg\| \nabla^2_{xy} g(x_1,y ; \zeta^{(0)}) \bigg[ \frac{K}{L_g} \prod_{i = 1}^k \bigg( I - \frac{1}{L_g} \nabla^2_{yy} g(x_1, y ; \zeta^{(i)}) \bigg) \bigg] \nabla_y f(x_1, y ;  \xi) \nonumber \\
		& \qquad\qquad \qquad  - \nabla^2_{xy} g(x_2,y ; \zeta^{(0)}) \bigg[ \frac{K}{L_g} \prod_{i = 1}^k \bigg( I - \frac{1}{L_g} \nabla^2_{yy} g(x_2, y ; \zeta^{(i)}) \bigg) \bigg] \nabla_y f(x_2, y ;  \xi) \bigg\|^2 \nonumber \\
		& \leq 3 C_{g_{xy}}^2 \frac{K^2}{L_g^2} \mathbb{E}_k \bigg[ \bigg(1 - \frac{\mu_g}{L_g} \bigg)^{2k} \bigg] L_{f_y}^2 \|x_1 - x_2\|^2 + 3 C_{f_y}^2 \frac{K^2}{L_g^2} \mathbb{E}_k \bigg[ \bigg(1 - \frac{\mu_g}{L_g} \bigg)^{2k} \bigg]  L_{g_{xy}}^2 \| x_1 - x_2\|^2 \nonumber \\
		& \qquad   + 3 C_{g_{xy}}^2 C_{f_y}^2 \frac{K^2}{L_g^2}   \mathbb{E}_k   \bigg\|  \prod_{i = 1}^k \bigg( I - \frac{1}{L_g} \nabla^2_{yy} g(x_1, y ; \zeta^{(i)}) \bigg) -   \prod_{i = 1}^k \bigg( I - \frac{1}{L_g} \nabla^2_{yy} g(x_2, y ; \zeta^{(i)}) \bigg) \bigg\|^2     \nonumber \\
		& \overset{(a)}{\leq} 3 C_{g_{xy}}^2 L_{f_y}^2 \bigg(\frac{K}{2 \mu_g L_g - \mu_g^2} \bigg)    \|x_1 - x_2\|^2 + 3 C_{f_y}^2 L_{g_{xy}}^2 \bigg(\frac{K}{ 2 \mu_g L_g - \mu_g^2}\bigg)     \| x_1 - x_2\|^2 \nonumber\\
		& \qquad    + 3 C_{g_{xy}}^2 C_{f_y}^2 \frac{K^2}{L_g^2}   \mathbb{E}_k   \bigg\|  \prod_{i = 1}^k \bigg( I - \frac{1}{L_g} \nabla^2_{yy} g(x_1, y ; \zeta^{(i)}) \bigg) -   \prod_{i = 1}^k \bigg( I - \frac{1}{L_g} \nabla^2_{yy} g(x_2, y ; \zeta^{(i)}) \bigg) \bigg\|^2 ,
		\label{Eq: Lip_GradEst_2nd}
		\end{align}
		where $(a)$ follows from the fact that we have:
		\begin{align*}
		\mathbb{E}_k \bigg[ \bigg( 1 - \frac{\mu_g}{L_g} \bigg)^{2k} \bigg] = \frac{1}{K} \sum_{k = 0}^{K - 1} \bigg( 1 - \frac{\mu_g}{L_g} \bigg)^{2k} \leq \frac{1}{K} \bigg( \frac{L_g^2}{2 \mu_g L_g - \mu_g^2} \bigg),
		\end{align*}
		where the first equality above follows from the fact that $k \in \{0, 1, \ldots, K-1 \}$ is chosen uniformly at random and the second equality results from the sum of a geometric progression. 
		
		Finally, considering the last term of \eqref{Eq: Lip_GradEst_2nd}, we have
		\begin{align}
		&   \mathbb{E}_{k} \bigg\|  \prod_{i = 1}^k \bigg( I - \frac{1}{L_g} \nabla^2_{yy} g(x_1, y ; \zeta^{(i)}) \bigg) -   \prod_{i = 1}^k \bigg( I - \frac{1}{L_g} \nabla^2_{yy} g(x_2, y ; \zeta^{(i)}) \bigg) \bigg\|^2 \nonumber\\
		& \qquad \qquad  \overset{(a)}{\leq} K \sum_{i = 1}^K \mathbb{E}_k \bigg[ \bigg( 1 - \frac{\mu_g}{L_g} \bigg)^{2(k - 1)} \bigg] \frac{1}{L_g^2} \big\| \nabla_{yy}^2 g(x_1, y ; \zeta^{(i)}) - \nabla_{yy}^2 g(x_2, y ; \zeta^{(i)}) \big\|^2 \nonumber\\
		& \qquad \qquad  \overset{(b)}{\leq}  \bigg( \frac{L_g^2}{(L_g - \mu_g)^2} \bigg) \bigg( \frac{1}{2 \mu_g L_g - \mu_g^2} \bigg) \sum_{i = 1}^K \big\| \nabla_{yy}^2 g(x_1, y ; \zeta^{(i)}) - \nabla_{yy}^2 g(x_2, y ; \zeta^{(i)}) \big\|^2 \nonumber\\
		& \qquad \qquad  \overset{(c)}{\leq}
		\frac{K L_g^2 L_{g_{yy}}^2}{(L_g - \mu_g)^2 (2 \mu_g L_g - \mu_g^2)}    \|x_1 - x_2\|^2,
		\label{Eq: Lip_GradEst_3rd}
		\end{align}
		where $(a)$ follows from the application of \eqref{Eq: Lip_Prod_Random_k} in Lemma \ref{Lem: Lip_Product} along with Assumption \ref{Assump: InnerFunction}--(ii)(iii); inequality $(b)$ utilizes
		\begin{align*}
		\mathbb{E}_k \bigg[ \bigg( 1 - \frac{\mu_g}{L_g} \bigg)^{2(k - 1)} \bigg] = \frac{1}{K} \sum_{k = 0}^{K - 1} \bigg( 1 - \frac{\mu_g}{L_g} \bigg)^{2(k - 1)} \leq \frac{1}{K} \bigg( \frac{L_g^2}{(L_g - \mu_g)^2}\bigg) \bigg( \frac{L_g^2}{2 \mu_g L_g - \mu_g^2} \bigg), 
		\end{align*}
		where the first equality above again utilizes the fact that $k \in \{0, 1, \ldots, K-1 \}$ is chosen uniformly at random and the second equality results from the sum of a geometric progression; inequality $(c)$ utilizes Assumption \ref{Assump: InnerFunction}--(v) made for stochastic lower level objective.
		
		Finally, taking expectation in \eqref{Eq: Lip_GradEst_1st} and substituting the expressions obtained in \eqref{Eq: Lip_GradEst_2nd} and \eqref{Eq: Lip_GradEst_3rd} in \eqref{Eq: Lip_GradEst_1st}, we obtain
		\begin{align*}
		\mathbb{E}   \| \bar{\nabla} f(x_1, y ; \bar{\xi}) - \bar{\nabla} f(x_2, y ; \bar{\xi}) \|^2 \leq L_K^2 \|x_1 - x_2\|^2,
		\end{align*}
		where $L_K^2$ defined as:
		\begin{align*}
		L_K^2 \coloneqq 2 L_{f_x}^2 + 6 C_{g_{xy}}^2 L_{f_y}^2 \bigg( \frac{K}{2 \mu_g L_g - \mu_g^2}\bigg) & + 6 C_{f_y}^2 L_{g_{xy}}^2 \bigg( \frac{K}{2 \mu_g L_g - \mu_g^2} \bigg) \\
		& \qquad  + 6 C_{g_{xy}}^2 C_{f_y}^2 \frac{K^3 L_{g_{yy}}^2}{(L_g - \mu_g)^2 (2 \mu_g L_g - \mu_g^2)}.
		\end{align*}
		Statement $(i)$ of the Lemma is proved.
		
		The proof of the statement $(ii)$ follows the same procedure, so it is omitted. 
	\end{proof}
	
	\section{Proof of Theorem \ref{Thm: Convergence_NC}: smooth (possibly non-convex) outer objective} \label{Sec: Appendix_Thm_NC}
	First, we consider the descent achieved by the outer objective in consecutive iterates generated by the Algorithm \ref{Algo: Accelerated_STSA} when the outer problem is smooth and is possibly non-convex. We define the following constants for the stepsize parameters:
	\begin{equation} \label{eq:step_param}
	\begin{split}
& w = \max \Big\{ 2 , ~27 L_f^3, ~8 L_{\mu_g}^3   c_\beta^3, ~(\mu_g + L_g)^3   c_\beta^3,~   c_{\eta_f}^{3/2},~  c_{\eta_g}^{3/2}   \Big\}, \quad c_\beta = \frac{6 \sqrt{2}  L_y L}{L_{\mu_g}}, \\
& c_{\eta_f} = \frac{1}{3 L_f} + \max \bigg\{ 36 L_K^2 , \frac{4 L_K^2  L_{\mu_g} (\mu_g + L_g) c_\beta^2}{L^2} \bigg\}, \\
& c_{\eta_g} = \frac{1}{3  L_f} + 8 L_g^2 c_\beta^2 + \bigg[ \frac{8 L^2}{L_{\mu_g}^2} + \frac{2L^2}{L_{\mu_g} (\mu_g + L_g)} \bigg] \max\bigg\{36 L_g^2 , \frac{4 L_g^2 L_{\mu_g} (\mu_g + L_g) c_\beta^2}{L^2} \bigg\},
\end{split}
\end{equation}
where we have defined $L_{\mu_g} = \frac{\mu_g L_g}{\mu_g + L_g}$. 
	
	\subsection{Descent in the function value}  
	\begin{lem}
		\label{Lem: Smoothness_l_Main}
		For non-convex and smooth $\ell(\cdot)$, with $e_t^f$ defined as: $e_t^f \coloneqq h_t^f - \bar{\nabla} f(x_t,y_t) - B_t$,
		the consecutive iterates of Algorithm \ref{Algo: Accelerated_STSA} satisfy:
		\begin{align*}
		\mathbb{E} [ \ell(x_{t+1}) ] & \leq \mathbb{E} \Big[ \ell(x_t) - \frac{\alpha_{t}}{2} \|\nabla \ell(x_t)\|^2 - \frac{\alpha_{t}}{2} (1 - \alpha_t L_f) \|h_t^f\|^2   + \alpha_{t} \|e_t^f\|^2 \\
		& \qquad \qquad \qquad  \qquad \qquad \qquad \qquad \qquad  + 2 \alpha_{t} L^2 \|y_{t} -   y^\ast(x_t)\|^2  + 2 \alpha_{t} \|B_t \|^2 \Big]. 
		\end{align*}
		for all $t \in \{0,1, \ldots, T-1\}$, where the expectation is w.r.t. the stochasticity of the algorithm. 
	\end{lem}
	
	\begin{proof}
		Using the Lipschitz smoothness of the objective function from Lemma \ref{Lem: Lip_Ghadhimi} we have:
		\begin{align}
		\ell(x_{t+1}) & \leq \ell(x_t) + \big\langle \nabla \ell(x_t), x_{t+1} - x_t \big\rangle + \frac{L_f}{2} \| x_{t+1} - x_t\|^2 \nonumber\\
		& \overset{(a)}{=} \ell(x_t) - \alpha_t \big\langle \nabla \ell(x_t), h_t^f \big\rangle + \frac{\alpha_t^2 L_f}{2} \| h_t^f \|^2 \nonumber\\
		& \overset{(b)}{=} \ell(x_t) - \frac{\alpha_t}{2} \|\nabla \ell(x_t)\|^2 - \frac{\alpha_t}{2} (1 - \alpha_t L_f) \|h_t^f\|^2 + \frac{\alpha_t}{2} \|h_t^f - \nabla \ell(x_t)\|^2.
		\label{Eq: Smoothness_l_1st}
		\end{align}
		where $(a)$ results from Step 7 of Algorithm \ref{Algo: Accelerated_STSA} and $(b)$ uses $\langle a, b \rangle =  \frac{1}{2} \| a\|^2 + \frac{1}{2} \| b \|^2  - \frac{1}{2} \| a - b \|^2$. Next, we bound the term $\|h_t^f - \nabla \ell(x_t)\|^2$ as follows
		\begin{align*}
		\|h_t^f - \nabla \ell(x_t)\|^2 &=  \|h_t^f - \bar{\nabla} f(x_t,y_t) - B_t + \bar{\nabla} f(x_t,y_t) + B_t -   \nabla \ell(x_t)\|^2\\
		& \overset{(c)}{\leq} 2 \|h_t^f - \bar{\nabla} f(x_t,y_t) - B_t \|^2 + 4\|\bar{\nabla} f(x_t,y_t) -   \nabla \ell(x_t)\|^2 + 4\|B_t \|^2\\
		& \overset{(d)}{\leq} 2 \|e_t^f\|^2 + 4 L^2 \|y_t -   y^\ast(x_t)\|^2 + 4\|B_t \|^2,
		\end{align*}
		where inequality $(c)$ uses \eqref{Eq: Sum_Norms} and $(d)$ results from the definition of $e_t^f \coloneqq h_t^f - \bar{\nabla} f(x_t,y_t) - B_t $ and \eqref{eq:lip} in Lemma \ref{Lem: Lip_Ghadhimi}. Substituting the above in \eqref{Eq: Smoothness_l_1st} and taking expectation w.r.t. the stochasticity of the algorithm we get the statement of the lemma. 
	\end{proof}

	\subsection{Descent in the iterates of the lower level problem} 
\begin{lem}
	\label{Lem: InnerIterates_Descent}
		Define $e_t^g \coloneqq h_t^g -  \nabla_y g(x_t , y_t)$. then the iterates of the inner problem generated according to Algorithm \ref{Algo: Accelerated_STSA}, satisfy
	\begin{align*}
&	\mathbb{E}\|y_{t+1} - y^\ast(x_{t + 1})\|^2 \\
& \leq (1 + \gamma_t) (1 + \delta_t) \bigg(1 -  2 \beta_t \frac{\mu_g L_g}{\mu_g + L_g } \bigg)  \mathbb{E} \|y_t - y^\ast(x_{t})\|^2 + \bigg( 1 + \frac{1}{\gamma_t} \bigg) L_y^2 \alpha_{t}^2 \mathbb{E}   \| h_{t}^f \|^2  \\
	& \quad  - 
	(1 + \gamma_t) (1 + \delta_t) \bigg( \frac{2 \beta_t}{\mu_g + L_g} - \beta_t^2  \bigg) \mathbb{E} \|\nabla_y g(x_t , y_t)\| +
(1 + \gamma_t) \bigg( 1 + \frac{1}{\delta_t} \bigg)	\beta_t^2 \mathbb{E}  \| e_t^g \|^2.
	\end{align*}
	for all $t \in \{0,\ldots, T - 1\}$ with some $\gamma_t, \delta_t > 0$., where the expectation is w.r.t. the stochasticity of the algorithm. 
\end{lem}
\begin{proof}
	Consider the term $\mathbb{E}\|y_{t+1} - y^\ast(x_{t + 1})\|^2$, we have
	\begin{align}
	\mathbb{E}\|y_{t+1} - y^\ast(x_{t + 1})\|^2 & \overset{(a)}{ \leq} 
 (1 + \gamma_t)	\mathbb{E}\|y_{t+1} - y^\ast(x_t) \|^2 + \bigg(1 + \frac{1}{\gamma_t} \bigg) 	\mathbb{E}\| y^\ast(x_t) - y^\ast(x_{t + 1})\|^2 
	 \nonumber\\
& \overset{(b)}{=} (1 + \gamma_t)	\mathbb{E}\|y_t - \beta_t h_t^g - y^\ast(x_t) \|^2 + \bigg(1 + \frac{1}{\gamma_t} \bigg) L_y^2	\mathbb{E}\|x_{t+1} -  x_t \|^2 \nonumber \\
& \overset{(c)}{\leq}  (1 + \gamma_t) (1 + \delta_t)	\mathbb{E}\|y_t - \beta_t \nabla_y g(x_t, y_t) - y^\ast(x_t) \|^2 \nonumber\\
&  +  (1 + \gamma_t) \bigg( 1 + \frac{1}{\delta_t} \bigg) \beta_t^2 \| h_t^g - \nabla_y g(x_t , y_t)\|^2  + \bigg(1 + \frac{1}{\gamma_t} \bigg) L_y^2 \alpha_t^2	\mathbb{E}\| h_t^f \|^2 .
	\label{Eq: Descent_InnerIterates_1st}
	\end{align}
where $(a)$ results from the Young's inequality; $(b)$ uses Step 5 of Algorithm \ref{Algo: Accelerated_STSA} and Lipschitzness of $y^\ast(\cdot)$ in Lemma \ref{Lem: Lip_Ghadhimi} and $(c)$ again utilizes Young's inequality and Step 7 of Algorithm \ref{Algo: Accelerated_STSA}. Next, we consider the first term of the above equation we have
\begin{align*}
 & \|y_t - \beta_t \nabla_y g(x_t, y_t) - y^\ast(x_t) \|^2  \\
 & \qquad \qquad =    \|y_t  - y^\ast(x_t) \|^2  + \beta_t^2  \|\nabla_y g(x_t, y_t) \|^2  - 2 \beta_t \langle \nabla_y g(x_t, y_t), y_t - y^\ast(x_t) \rangle \\
 	& \qquad \qquad \overset{(d)}{\leq} \bigg( 1 - 2 \beta_t \frac{\mu_g L_g}{\mu_g + L_g} \bigg) \| y_t - y^\ast(x_t)\|^2 - \bigg(\frac{ 2 \beta_t }{\mu_g + L_g} - \beta_t^2 \bigg) \|\nabla_y g(x_t, y_t)\|^2,
\end{align*}
where inequality $(d)$ above results from the strong convexity of $g$, which implies
\begin{align*}
\langle \nabla_y g(x_t, y_t), y_t - y^\ast(x_t) \rangle \geq \frac{\mu_g L_g}{\mu_g + L_g} \| y_t - y^\ast(x_t) \|^2 + \frac{1}{\mu_g + L_g} \| \nabla_y g(x_t, y_t)\|^2.
\end{align*}
Substituting in \eqref{Eq: Descent_InnerIterates_1st} and using the definition $e_t^g \coloneqq h_t^g -  \nabla_y g(x_t , y_t)$ we get the statement of the lemma. 
\end{proof}

	\subsection{Descent in the gradient estimation error of the outer function} 
	\label{Subsec: Descent_GradEstError_Appendix}
	Before presenting the descent in the gradient estimation error of the outer function we define $\mathcal{F}_t = \sigma \{y_0, x_0,  \ldots, y_t, x_t\}$ as the sigma algebra generated by the sequence of iterates up to the $t$th iteration of \aname.
	
	\begin{lem}
		\label{Lem: Descent_grad_error}
		Define $e_t^f \coloneqq h_t^f - \bar{\nabla}f(x_t , y_t) - B_t$. Then the consecutive iterates of Algorithm \ref{Algo: Accelerated_STSA} satisfy:
		\begin{align*}
		\mathbb{E} \|e_{t+1}^f\|^2 	&  \leq (1 - \eta_{t+1}^f)^2 \mathbb{E} \| e_t^f \|^2 +  2 (\eta_{t+1}^f )^2 \sigma_f^2  +  4 (1 - \eta_{t+1}^f)^2 L_K^2 \alpha_t^2 \mathbb{E}\|h_t^f\|^2 \\
		&  \qquad \qquad \qquad   + 8 (1 - \eta_{t+1}^f)^2 L_K^2 \beta_t^2 \mathbb{E}\|e_t^g\|^2 + 8 (1 - \eta_{t+1}^f)^2 L_K^2 \beta_t^2 \mathbb{E}\|\nabla_y g(x_t, y_t)\|^2,
		\end{align*}
		for all $t \in \{0,\ldots, T - 1\}$, with $L_K$ defined in the statement of Lemma \ref{Lem: Lip_GradEst_Appendix}. Here the expectation is taken w.r.t the stochasticity of the algorithm. 	
	\end{lem}
	\begin{proof}
		From the definition of $e_t^f$ we have
		\begin{align}
	&	\mathbb{E} \|e_{t + 1}^f\|^2 \\
		& = \mathbb{E} \| h_{t+1}^f - \bar{\nabla} f(x_{t+1} , y_{t+1}) - B_{t+1} \|^2 \nonumber\\
		&  \overset{(a)}{=} \mathbb{E} \big\|  \eta_{t+1}^f \bar{\nabla} f(x_{t+1} , y_{t+1} ; \xi_{t+1}) +  (1 - \eta_{t + 1})      \big(  h_{t}^{f}   + \bar{\nabla} f(x_{t+1}, y_{t+1} ; \xi_{t+1}) -   \bar{\nabla} f(x_{t}, y_{t} ; \xi_{t+1})  \big) \nonumber\\ 
	& \qquad \qquad \qquad \qquad \qquad    \qquad \qquad \qquad  \qquad \qquad \qquad \qquad \qquad 	- \bar{\nabla} f(x_{t+1} , y_{t+1}) - B_{t+1} \big\|^2 \nonumber\\
		& \overset{(b)}{=} \mathbb{E} \big\| (1 - \eta_{t+1}^f) e_t^f + \eta_{t+1}^f (\bar{\nabla} f(x_{t+1}, y_{t+1} ; \xi_{t + 1})  - \bar{\nabla} f(x_{t+1} , y_{t+1}) - B_{t+1} ) \nonumber\\
		&    \qquad \qquad \qquad   +  (1 - \eta_{t + 1}^f) \big((\bar{\nabla} f(x_{t+1} , y_{t+1} ; \xi_{t+1})  - \bar{\nabla} f(x_{t+1} , y_{t+1}) - B_{t + 1} ) \nonumber \\
	&\qquad \qquad \qquad \qquad \qquad \qquad \qquad	- (\bar{\nabla} f(x_{t} , y_{t} ; \xi_{t+1})  - \bar{\nabla} f(x_{t} , y_t) - B_{t} ) \big) \big\|^2 \nonumber\\
			& \overset{(c)}{=}  (1 - \eta_{t+1}^f)^2 \mathbb{E} \big\| e_t^f  \big\|^2 +  \mathbb{E} \big\| \eta_{t+1}^f (\bar{\nabla} f(x_{t+1}, y_{t+1} ; \xi_{t + 1})  - \bar{\nabla} f(x_{t+1} , y_{t+1}) - B_{t+1} ) \nonumber\\
		& \qquad \qquad\qquad \qquad\qquad    +  (1 - \eta_{t + 1}^f) \big((\bar{\nabla} f(x_{t+1} , y_{t+1} ; \xi_{t+1})  - \bar{\nabla} f(x_{t+1} , y_{t+1}) - B_{t + 1} ) \nonumber \\
		& \qquad \qquad\qquad \qquad\qquad \qquad\qquad \qquad \qquad \qquad - (\bar{\nabla} f(x_{t} , y_{t} ; \xi_{t+1})  - \bar{\nabla} f(x_{t} , y_t) - B_{t} ) \big) \big\|^2 \nonumber\\
		&  \overset{(d)}{\leq} (1 - \eta_{t + 1}^f)^2 \mathbb{E} \| e_{t}^f \|^2 +  2 (\eta_{t+1}^f)^2 \mathbb{E} \big\| \bar{\nabla} f(x_{t+1} , y_{t+1} ; \xi_{t + 1})  - \bar{\nabla} f(x_{t+1} , y_{t+1}) - B_{t+1}   \big\|^2 \nonumber\\
		& \qquad \qquad \qquad \qquad +  2 (1 - \eta_{t}^f)^2\mathbb{E}\big\| \big(\bar{\nabla} f(x_{t+1} , y_{t+1} ; \xi_{t+1})  - \bar{\nabla} f(x_{t+1} , y_{t+1}) - B_{t+1} \big) \nonumber\\
		& \qquad \qquad \qquad \qquad \qquad\qquad \qquad\qquad \qquad \qquad - \big( \bar{\nabla} f(x_{t} , y_{t} ; \xi_{t+1})  - \bar{\nabla} f(x_{t} , y_t) - B_{t}  \big) \big\|^2  \nonumber\\
			&  \overset{(e)}{\leq} (1 - \eta_{t + 1}^f)^2 \mathbb{E} \| e_{t}^f \|^2 +  2 (\eta_{t+1}^f)^2 \sigma_f^2 \nonumber\\
		& \qquad \qquad \qquad\qquad   +  2 (1 - \eta_{t}^f)^2\mathbb{E}\big\| \big(\bar{\nabla} f(x_{t+1} , y_{t+1} ; \xi_{t+1})  - \bar{\nabla} f(x_{t+1} , y_{t+1}) - B_{t+1} \big) \nonumber \\
		& \qquad \qquad\qquad \qquad\qquad \qquad\qquad \qquad \qquad \qquad - \big( \bar{\nabla} f(x_{t} , y_{t} ; \xi_{t+1})  - \bar{\nabla} f(x_{t} , y_t) - B_{t}  \big) \big\|^2  
			\label{Eq: Descent_grad_error_1st}
		\end{align}
			where equality $(a)$ uses the definition of the recursive gradient estimator \eqref{Eq: HybridGradEstimate_Outer}; $(b)$ results from the definition $e_t^f \coloneqq h_t^f - \bar{\nabla}f(x_t , y_t) - B_t$; $(c)$ follows from the fact that conditioned on $\mathcal{F}_{t+1} = \sigma \{y_0, x_0, \ldots, y_t,x_t, y_{t + 1},x_{t + 1} \}$
	\begin{align*}
	&	\mathbb{E} \Big\langle e_t^f,   (\bar{\nabla} f(x_{t+1}, y_{t+1} ; \xi_{t + 1})  - \bar{\nabla} f(x_{t+1} , y_{t+1}) - B_{t+1} )  \\
	& \qquad \qquad \qquad \qquad  -  (1 - \eta_{t + 1}^f) \big( (\bar{\nabla} f(x_{t} , y_{t} ; \xi_{t+1})  - \bar{\nabla} f(x_{t} , y_t) - B_{t} ) \big)   \Big\rangle \\
		& \mathbb{E} \Big\langle e_t^f,  \mathbb{E} \big[ (\bar{\nabla} f(x_{t+1}, y_{t+1} ; \xi_{t + 1})  - \bar{\nabla} f(x_{t+1} , y_{t+1}) - B_{t+1} )  \\
		& \quad \qquad \underbrace{ \qquad \qquad -  (1 - \eta_{t + 1}^f) \big( (\bar{\nabla} f(x_{t} , y_{t} ; \xi_{t+1})  - \bar{\nabla} f(x_{t} , y_t) - B_{t} ) \big) | \mathcal{F}_{t+1} \big]}_{=0}  \Big\rangle = 0,
	\end{align*}
	which follows from the fact that the second term in the inner product above is zero mean as a consequence of Assumption \ref{Assump: Stochastic Grad}-(i) and inequality $(d)$ utilizes \eqref{Eq: Sum_Norms}; and $(e)$ results from Assumption \ref{Assump: Stochastic Grad}-(i). 
	
	Next, we bound the last term of \eqref{Eq: Descent_grad_error_1st} above
		\begin{align}
	 &  2 (1 - \eta_{t+1}^f)^2\mathbb{E}\big\| \big(\bar{\nabla} f(x_{t+1} , y_{t+1} ; \xi_{t+1})  - \bar{\nabla} f(x_{t} , y_{t} ; \xi_{t+1}) \big)    \nonumber\\
	 & \qquad \qquad \qquad \qquad \qquad \qquad - \Big(\big(\bar{\nabla} f(x_{t+1} , y_{t+1}) + B_{t+1} \big) - \big(   \bar{\nabla} f(x_{t} , y_t) + B_{t}  \big) \Big)\big\|^2 \nonumber\\
		&   \overset{(a)}{\leq}    2 (1 - \eta_{t + 1}^f)^2\mathbb{E}\big\|\bar{\nabla} f(x_{t+1} , y_{t+1} ; \xi_{t+1})  - \bar{\nabla} f(x_{t} , y_{t} ; \xi_{t+1}) \|^2 \nonumber\\
		& \overset{(b)} {\leq}    4 (1 - \eta_{t+1}^f)^2 \mathbb{E}\big\|\bar{\nabla} f(x_{t+1} , y_{t+1} ; \xi_{t+1})  - \bar{\nabla} f(x_t , y_{t  + 1} ; \xi_{t + 1}) \big\|^2 \nonumber \\
		&  \qquad \qquad \qquad \qquad   \qquad \qquad   + 4 (1 - \eta_{t + 1}^f)^2 \mathbb{E}\big\|\bar{\nabla} f(x_t , y_{t + 1} ; \xi_{t  +1})  - \bar{\nabla} f(x_{t} , y_t ; \xi_{t + 1}) \big\|^2 \nonumber\\
		& \overset{(c)}{\leq}   4 (1 - \eta_{t + 1}^f)^2 L_K^2 \mathbb{E}\| x_{t+1}  -  x_{t} \|^2   + 4 (1 - \eta_{t + 1}^f)^2 L_K^2 \mathbb{E}\|y_{t+1} - y_{t}  \|^2 \nonumber\\
		&  \overset{(d)}{\leq}    4 (1 - \eta_{t + 1}^f)^2 L_K^2 \alpha_t^2 \mathbb{E}\|  h_t^f \|^2   + 4 (1 - \eta_{t+1}^f)^2 L_K^2 \beta_{t}^2 \mathbb{E}\|h_{t}^g \|^2,\nonumber \\
		& \overset{(e)}{\leq}   4 (1 - \eta_{t + 1}^f)^2 L_K^2 \alpha_t^2 \mathbb{E}\|  h_t^f \|^2    + 8 (1 - \eta_{t+1}^f)^2 L_K^2 \beta_{t}^2 \mathbb{E}\|e_{t}^g \|^2 \nonumber\\
		& \qquad \qquad \qquad \qquad \qquad \qquad \qquad \qquad + 8 (1 - \eta_{t+1}^f)^2 L_K^2 \beta_{t}^2 \mathbb{E}\|\nabla_y g(x_t, y_t) \|^2,
			\label{Eq: Descent_grad_error_2nd}
		\end{align}
  where $(a)$ follows from the mean variance inequality: For a random variable $Z$ we have $\mathbb{E}\|Z - \mathbb{E}[Z]\|^2 \leq \mathbb{E}\| Z\|^2$ with $Z$ defined as $Z \coloneqq \bar{\nabla} f(x_{t+1}, y_{t+1} ; \xi_{t + 1}) - \bar{\nabla} f(x_t, y_t ; \xi_{t + 1})$; $(b)$ again uses \eqref{Eq: Sum_Norms}; $(c)$ follows from Lemma \ref{Lem: Lip_GradEst_Appendix}; inequality $(d)$ uses Steps 5 and 7 of Algorithm \ref{Algo: Accelerated_STSA}; finally, $(e)$ utilizes \eqref{Eq: Sum_Norms} and the definition of $e_t^g$.
		
Finally, substituting \eqref{Eq: Descent_grad_error_2nd} in \eqref{Eq: Descent_grad_error_1st}, we get the statement of the lemma. 

	 Therefore, the lemma is proved.  
	\end{proof}
	
	\subsection{Descent in the gradient estimation error of the inner function} 
	We consider the descent on the gradient estimation error of the inner function. 
	\begin{lem}
		\label{lem: Descent_Grad_Error_Inner}
	Define $e_t^g \coloneqq h_t^g - \nabla_y g(x_t , y_t)$. Then the iterates generated from Algorithm \ref{Algo: Accelerated_STSA} satisfy
	\begin{align*}
	\mathbb{E} \|e_{t+1}^g\|^2	& \leq \Big(  (1 - \eta_{t+1}^g)^2 + 8 (1 - \eta_{t+1}^g)^2 L_g^2 \beta_t^2
		\Big) \mathbb{E}\| e_t^g \|^2  + 2 (\eta_{t+1}^g)^2 \sigma_g^2 \\
		&  \qquad \qquad \qquad + 4 (1 - \eta_{t+1}^g)^2 L_g^2 \alpha_t^2 \mathbb{E} \|  h_t^f \|^2    +  8 (1 - \eta_{t+1}^g)^2 L_g^2 \beta_t^2 \mathbb{E} \|   \nabla_y g(x_t, y_t)   \|^2
	\end{align*}
	for all $t \in \{ 0, 1, \cdots, T - 1 \}$, where the expectation is taken w.r.t. the stochasticity of the algorithm. 
	\end{lem}
	
	\begin{proof}
		 From the definition of $e_t^g$ we have
		 \begin{align*}
		 & \mathbb{E}\|e_{t+1}^g\|^2  = \mathbb{E}\|h_{t+1}^g - \nabla_y g(x_{t+1}, y_{t+1})\|^2 \\
		  &  \overset{(a)}{=} \mathbb{E} \| \nabla_y g(x_{t+1}, y_{t+1},  \zeta_{t+1}) + (1 - \eta_{t+1}^g) \big(h_{t}^g - \nabla_y g(x_t , y_t ; \zeta_{t+1}) \big) - \nabla_y g(x_{t+1}, y_{t+1}) \|^2\\
		&  \overset{(b)}{=}   \mathbb{E}\big\| (1 - \eta_{t+1}^g) e_t^g  + \big( \nabla_y g(x_{t+1}, y_{t+1},  \zeta_{t+1})  - \nabla_y g(x_{t+1}, y_{t+1}) \big)
		  \\
		  & \qquad \qquad \qquad \qquad   \qquad \qquad \qquad \qquad   -  (1 - \eta_{t+1}^g) \big( \nabla_y g(x_t , y_t ; \zeta_{t+1}) - \nabla_y g(x_t, y_t) \big) \big\|^2 \\
		  & \overset{(c)}{=}  (1 - \eta_{t+1}^g)^2 \mathbb{E}\| e_t^g \|^2 \\
		  & +  \mathbb{E} \| \nabla_y g(x_{t+1}, y_{t+1},  \zeta_{t+1})  - \nabla_y g(x_{t+1}, y_{t+1})
		 -  (1 - \eta_{t+1}^g) \big( \nabla_y g(x_t , y_t ; \zeta_{t+1}) - \nabla_y g(x_t, y_t) \big) \|^2 \\
		 	  & \overset{(d)}{\leq}  (1 - \eta_{t+1}^g)^2 \mathbb{E}\| e_t^g \|^2  + 2 (\eta_{t+1}^g)^2 \sigma_g^2 + 2 (1 - \eta_{t+1}^g)^2 \mathbb{E} \| \nabla_y g(x_{t+1}, y_{t+1},  \zeta_{t+1})  -     \nabla_y g(x_t , y_t ; \zeta_{t+1})     \|^2\\
		 	    & \overset{(e)}{\leq}  (1 - \eta_{t+1}^g)^2 \mathbb{E}\| e_t^g \|^2  + 2 (\eta_{t+1}^g)^2 \sigma_g^2 \\
		 	    &  \qquad \qquad  \qquad   + 4 (1 - \eta_{t+1}^g)^2 \mathbb{E} \| \nabla_y g(x_{t+1}, y_{t+1},  \zeta_{t+1})  -     \nabla_y g(x_t , y_{t+1} ; \zeta_{t+1})     \|^2 \\
		 	    &  \qquad \qquad \qquad   \qquad   \qquad \qquad  + 4 (1 - \eta_{t+1}^g)^2 \mathbb{E} \| \nabla_y g(x_{t}, y_{t+1},  \zeta_{t+1})  -     \nabla_y g(x_{t} , y_t ; \zeta_{t+1})     \|^2\\
		 	     	    & \overset{(f)}{\leq}  (1 - \eta_{t+1}^g)^2 \mathbb{E}\| e_t^g \|^2  + 2 (\eta_{t+1}^g)^2 \sigma_g^2   + 4 (1 - \eta_{t+1}^g)^2 L_g^2 \mathbb{E} \|  x_{t+1}  -      x_t \|^2    + 4 (1 - \eta_{t+1}^g)^2 L_g^2 \mathbb{E} \|   y_{t+1}  -      y_t   \|^2\\
		 	     	     & \overset{(g)}{\leq}  (1 - \eta_{t+1}^g)^2 \mathbb{E}\| e_t^g \|^2  + 2 (\eta_{t+1}^g)^2 \sigma_g^2   + 4 (1 - \eta_{t+1}^g)^2 L_g^2 \alpha_t^2 \mathbb{E} \|  h_t^f \|^2    + 4 (1 - \eta_{t+1}^g)^2 L_g^2 \beta_t^2 \mathbb{E} \|   h_t^g   \|^2\\
		 	     	       & \overset{(h)}{\leq}  (1 - \eta_{t+1}^g)^2 \mathbb{E}\| e_t^g \|^2  + 2 (\eta_{t+1}^g)^2 \sigma_g^2   + 4 (1 - \eta_{t+1}^g)^2 L_g^2 \alpha_t^2 \mathbb{E} \|  h_t^f \|^2  \\
		 	     	       & \qquad \qquad \qquad \qquad\qquad \qquad\qquad   + 8 (1 - \eta_{t+1}^g)^2 L_g^2 \beta_t^2 \mathbb{E} \|   e_t^g   \|^2 +  8 (1 - \eta_{t+1}^g)^2 L_g^2 \beta_t^2 \mathbb{E} \|   \nabla_y g(x_t, y_t)   \|^2\\
		 	     	          & ~{\leq} \Big(  (1 - \eta_{t+1}^g)^2 + 8 (1 - \eta_{t+1}^g)^2 L_g^2 \beta_t^2
		 	     	          \Big) \mathbb{E}\| e_t^g \|^2  + 2 (\eta_{t+1}^g)^2 \sigma_g^2   + 4 (1 - \eta_{t+1}^g)^2 L_g^2 \alpha_t^2 \mathbb{E} \|  h_t^f \|^2  \\
		 	     	       & \qquad \qquad \qquad \qquad\qquad \qquad \qquad\qquad \qquad \qquad\quad \qquad\qquad   +  8 (1 - \eta_{t+1}^g)^2 L_g^2 \beta_t^2 \mathbb{E} \|   \nabla_y g(x_t, y_t)   \|^2,
		 \end{align*}
		 where equality $(a)$ uses the definition of hybrid gradient estimator \eqref{Eq: HybridGradEstimate_Inner}; $(b)$ uses the definition of $e_t^g$; $(c)$ uses the fact that conditioned on $\mathcal{F}_{t + 1} = \sigma \{ y_0, x_0, \ldots, y_t, x_t, y_{t + 1}, x_{t+1} \}$ 
		 \begin{align*}
		& \mathbb{E}  \big\langle  e_t^g  , \big( \nabla_y g(x_{t+1}, y_{t+1},  \zeta_{t+1})  - \nabla_y g(x_{t+1}, y_{t+1}) \big)  -  (1 - \eta_{t+1}^g) \big( \nabla_y g(x_t , y_t ; \zeta_{t+1}) - \nabla_y g(x_t, y_t) \big) \big\rangle \\
		 &   = 		 \mathbb{E}  \big\langle  e_t^g  , \underbrace{\mathbb{E} \big[ \big( \nabla_y g(x_{t+1}, y_{t+1},  \zeta_{t+1})  - \nabla_y g(x_{t+1}, y_{t+1}) \big)  -  (1 - \eta_{t+1}^g) \big( \nabla_y g(x_t , y_t ; \zeta_{t+1}) - \nabla_y g(x_t, y_t) \big) | \mathcal{F}_{t+1}  \big]}_{ = 0} \big\rangle \\
		 & \quad = 0.
		 \end{align*}
		 Inequality $(d)$ results from the application of \eqref{Eq: Sum_Norms} and Assumption \ref{Assump: Stochastic Grad}-(ii); $(e)$ again uses \eqref{Eq: Sum_Norms}; $(f)$ utilizes Assumption \ref{Assump: InnerFunction}; $(g)$ follows from Steps 5 and 7 of Algorithm \ref{Algo: Accelerated_STSA} and finally, $(h)$ follows from the application of \eqref{Eq: Sum_Norms} and the definition of $e_t^g$.
		 
		 Therefore, the lemma is proved. 
	\end{proof}

	\subsection{Descent in the potential function}  
	Let us define the potential function as:
	\begin{align}
	\label{Eq: PotentialFunction}
	V_{t} \coloneqq \ell(x_t) + \frac{2L}{3 \sqrt{2} L_y} \|y_t - y^\ast(x_t)\|^2  +  \frac{1}{\bar{c}_{\eta_f}}  \frac{\|e_t^f\|^2}{\alpha_{t-1}} + \frac{1}{\bar{c}_{\eta_g}} \frac{\|e_t^g\|^2}{\alpha_{t-1}}
	\end{align}
	where we define 
	\begin{align}
	\bar{c}_{\eta_f} \coloneqq \max \bigg\{ 36 L_K^2 , \frac{4 L_K^2  L_{\mu_g} (\mu_g + L_g) c_\beta^2}{L^2} \bigg\} \quad \text{and} \quad \bar{c}_{\eta_g} \coloneqq \max \bigg\{  36 L_g^2 , \frac{4 L_g^2  L_{\mu_g} (\mu_g + L_g) c_\beta^2}{L^2} \bigg\}.
	\end{align}
	with $L_{\mu_g}$ defined as $L_{\mu_g} \coloneqq \frac{\mu_g L_g}{\mu_g + L_g}$.
	
Next, we quantify the expected descent in the potential function $\mathbb{E}[V_{t+1} - V_t]$. 
	
	\begin{lem}
		\label{Lem: Decent_PotentialFn}
		Consider $V_t$ defined in \eqref{Eq: PotentialFunction}. Suppose the parameters of Algorithm \ref{Algo: Accelerated_STSA} are chosen as
		\begin{align*}
		\alpha_t \coloneqq \frac{1}{(w + t)^{1/3}} , ~\beta_{t} \coloneqq c_\beta \alpha_{t}, ~ \eta_{t + 1}^f \coloneqq  c_{\eta_f} \alpha_{t}^2 ,   ~\text{and} ~ \eta_{t + 1}^g \coloneqq  c_{\eta_g} \alpha_t^2 ~ \text{for all}~~t \in \{0,1, \ldots, T-1\}.
		\end{align*}
		with 
		\begin{align*}
		c_\beta \coloneqq \frac{6 \sqrt{2}  L_y L}{L_{\mu_g}}, ~ c_{\eta_f}  \coloneqq \frac{1}{3  L_f} + \bar{c}_{\eta_f}~    \text{and} ~ c_{\eta_g} \coloneqq  \frac{1}{3   L_f} + 8 L_g^2 c_\beta^2 + \bigg[ \frac{8 L^2}{L_{\mu_g}^2} + \frac{2L^2}{L_{\mu_g} (\mu_g + L_g)} \bigg] \bar{c}_{\eta_g} ,
		\end{align*}
	where $L_{\mu_g} \coloneqq \frac{\mu_g L_g}{\mu_g + L_g}$ and
		\begin{align*}
		\bar{c}_{\eta_f} = \max \bigg\{ 36 L_K^2 , \frac{4 L_K^2  L_{\mu_g} (\mu_g + L_g) c_\beta^2}{L^2} \bigg\} \quad \text{and} \quad \bar{c}_{\eta_g}  = \max\bigg\{    36 L_g^2 , \frac{4 L_g^2 L_{\mu_g} (\mu_g + L_g) c_\beta^2}{L^2} \bigg\}.
		\end{align*}
 and the parameters 
 \begin{align*}
 \gamma_t  \coloneqq  \frac{\beta_t L_{\mu_g} / 2}{1 - \beta_t L_{\mu_g}} \qquad  \text{and} \qquad \delta_t  \coloneqq \frac{\beta_t L_{\mu_g}}{1 - 2\beta_t L_{\mu_g}}.
 \end{align*}
 Then the iterates generated by Algorithm \ref{Algo: Accelerated_STSA} when the outer problem is non-convex satisfy:
		\begin{align*}
		\mathbb{E} [ V_{t+1} - V_t ]  \leq - \frac{\alpha_{t}}{2} \mathbb{E} \| \nabla \ell(x_t)\|^2   +  2 \alpha_{t} \|B_t\|^2  +  \frac{2   (\eta_{t+1}^f)^2}{\bar{c}_{\eta_f} \alpha_t} \sigma_f^2 +  \frac{2   (\eta_{t+1}^g)^2}{\bar{c}_{\eta_g} \alpha_t} \sigma_g^2.
		\end{align*}
		for all $t \in \{0, 1, \ldots, T-1\}$
	\end{lem}
	\begin{proof}
		We have from Lemma \ref{Lem: InnerIterates_Descent}
	 \begin{align}
	 \label{Eq: InnerIterates_PotFn}
	 		\mathbb{E}\|y_{t+1} - y^\ast(x_{t + 1})\|^2 - 	\mathbb{E}\|y_{t} - y^\ast(x_{t})\|^2 & \leq  \bigg[ (1 + \gamma_t) (1 + \delta_t) \bigg(1 -  2 \beta_t \frac{\mu_g L_g}{\mu_g + L_g } \bigg)  - 1 \bigg] \mathbb{E} \|y_t - y^\ast(x_{t})\|^2 \nonumber\\
	 		&   \qquad     - 
	 	(1 + \gamma_t) (1 + \delta_t) \bigg( \frac{2 \beta_t}{\mu_g + L_g} - \beta_t^2  \bigg) \mathbb{E} \|\nabla_y g(x_t , y_t)\| \nonumber\\
	 	& \qquad    + (1 + \gamma_t) \bigg( 1 + \frac{1}{\delta_t} \bigg)	\beta_t^2 \mathbb{E}  \| e_t^g \|^2 + \bigg( 1 + \frac{1}{\gamma_t} \bigg) L_y^2 \alpha_{t}^2 \mathbb{E}   \| h_{t}^f \|^2 .
	 \end{align}
	 Let us consider coefficient of the first term of \eqref{Eq: InnerIterates_PotFn} above, choosing $\gamma_t$ and $\delta_t$ such that we have 
	 \begin{align}
	 \label{Eq: gamma_delta}
	   (1 + \gamma_t) (1 + \delta_t)  (1 -  2 \beta_t L_{\mu_g} )   = 1 -  \frac{\beta_t L_{\mu_g}}{2}  
	 \end{align}
	 where we define $L_{\mu_g} \coloneqq \frac{\mu_g L_g}{\mu_g + L_g}$. First we choose $\gamma_t$ such that we have
	 	 \begin{align*}
 (1 + \delta_t)  (1 -  2 \beta_t L_{\mu_g} )   = 1 -  \beta_t L_{\mu_g} \quad \Rightarrow \quad	 1 + \delta_t    =  \frac{1 -  \beta_t L_{\mu_g}}{  1 -  2 \beta_t L_{\mu_g} } \quad \Rightarrow \quad  \delta_t = \frac{\beta_t L_{\mu_g}}{1 - 2\beta_t L_{\mu_g}}
	 \end{align*}
Moreover, this implies that we have:
\begin{align*}
1 + \frac{1}{\delta_t} = 1 + \frac{1 - 2\beta_t L_{\mu_g} }{\beta_t L_{\mu_g}} \leq \frac{1}{\beta_t L_{\mu_g}}.
\end{align*}	
Using the definition of $\delta_t$ in \eqref{Eq: gamma_delta} we
	 \begin{align*}
	  (1 + \gamma_t)  (1 -  \beta_t L_{\mu_g} )   = 1 -  \frac{\beta_t L_{\mu_g}}{2} \quad \Rightarrow \quad	 1 + \gamma_t    =  \frac{1 -  \frac{\beta_t L_{\mu_g}}{2}}{  1 -   \beta_t L_{\mu_g} } \quad \Rightarrow \quad  \gamma_t = \frac{\beta_t L_{\mu_g} / 2}{1 - \beta_t L_{\mu_g}}
	 \end{align*}
	 Moreover, this implies that we have:
	 \begin{align*}
	 1 + \frac{1}{\gamma_t} = 1 + \frac{1 - \beta_t L_{\mu_g} }{\beta_t L_{\mu_g}/2} \leq \frac{2}{\beta_t L_{\mu_g}}.
	 \end{align*}	
 Substituting the above bounds in \eqref{Eq: InnerIterates_PotFn}, we get
  \begin{align*}
 \mathbb{E}\|y_{t+1} - y^\ast(x_{t + 1})\|^2 - 	\mathbb{E}\|y_{t} - y^\ast(x_{t})\|^2 & \leq  - \frac{\beta_t L_{\mu_g}}{2}  \mathbb{E} \|y_t - y^\ast(x_{t})\|^2     - \bigg( \frac{2 \beta_t}{\mu_g + L_g} - \beta_t^2  \bigg) \mathbb{E} \|\nabla_y g(x_t , y_t)\| \nonumber\\
 & \qquad \qquad \qquad \qquad \quad   +    \frac{2}{\beta_t L_{\mu_g}} 	\beta_t^2 \mathbb{E}  \| e_t^g \|^2 + \frac{2}{\beta_t L_{\mu_g}} L_y^2 \alpha_{t}^2 \mathbb{E}   \| h_{t}^f \|^2.
 \end{align*}
Choosing $\beta_t \leq \frac{1}{\mu_g + L_g}$ we get 
   \begin{align*}
 \mathbb{E}\|y_{t+1} - y^\ast(x_{t + 1})\|^2 - 	\mathbb{E}\|y_{t} - y^\ast(x_{t})\|^2 & \leq  - \frac{\beta_t L_{\mu_g}}{2}  \mathbb{E} \|y_t - y^\ast(x_{t})\|^2     -   \frac{\beta_t}{\mu_g + L_g}   \mathbb{E} \|\nabla_y g(x_t , y_t)\| \nonumber\\
 & \qquad \qquad \qquad \qquad \quad   +    \frac{2}{\beta_t L_{\mu_g}} 	\beta_t^2 \mathbb{E}  \| e_t^g \|^2 + \frac{2}{\beta_t L_{\mu_g}} L_y^2 \alpha_{t}^2 \mathbb{E}   \| h_{t}^f \|^2.
 \end{align*}
 Using the definition of $\beta_t = c_\beta \alpha_t$ and multiplying both sides by $\frac{4L^2}{c_\beta L_{\mu_g}}$ we get
   \begin{align*}
\frac{4L^2}{c_\beta L_{\mu_g}}  \mathbb{E} \big[ \|y_{t+1} - y^\ast(x_{t + 1})\|^2 - 	 \|y_{t} - y^\ast(x_{t})\|^2 \big] & \leq  - 2 \alpha_t L^2 \mathbb{E} \|y_t - y^\ast(x_{t})\|^2     -   \frac{4L^2 \alpha_t }{L_{\mu_g } (\mu_g + L_g)}   \mathbb{E} \|\nabla_y g(x_t , y_t)\| \nonumber\\
 & \qquad \qquad \qquad \qquad \quad   +    \frac{8L^2 \alpha_t}{L_{\mu_g}^2}   \mathbb{E}  \| e_t^g \|^2 + \frac{8 L_y^2 L^2 \alpha_t}{c_\beta^2 L_{\mu_g}^2}   \mathbb{E}   \| h_{t}^f \|^2.
 \end{align*}
Finally, choosing $c_\beta = \frac{6 \sqrt{2}  L_y L}{L_{\mu_g}}$ such that $\frac{8 L_y^2 L^2}{c_\beta^2 L_{\mu_g}^2} = \frac{1}{9}$ 
 \begin{align}
 \label{Eq: InnerIterates_PotFn_Main}
 \frac{ 2L}{3 \sqrt{2}   L_y}  \mathbb{E} \big[ \|y_{t+1} - y^\ast(x_{t + 1})\|^2 - 	 \|y_{t} - y^\ast(x_{t})\|^2 \big] & \leq  - 2 \alpha_t L^2 \mathbb{E} \|y_t - y^\ast(x_{t})\|^2     -   \frac{4L^2 \alpha_t }{L_{\mu_g } (\mu_g + L_g)}   \mathbb{E} \|\nabla_y g(x_t , y_t)\| \nonumber\\
 & \qquad \qquad \qquad \qquad \qquad \qquad    +    \frac{8L^2 \alpha_t}{L_{\mu_g}^2}   \mathbb{E}  \| e_t^g \|^2 + \frac{\alpha_t}{9}   \mathbb{E}   \| h_{t}^f \|^2.
 \end{align}

 Next, we have from Lemma \ref{Lem: Descent_grad_error}
 	\begin{align}
 	\label{Eq: Grad_error_OuterFn_PotFn}
\frac{ \mathbb{E} \|e_{t+1}^f\|^2}{\alpha_t} -   \frac{\mathbb{E} \|e_{t+1}^f\|^2}{\alpha_{t-1}} &   \leq \bigg[  \frac{(1 - \eta_{t+1}^f)^2}{\alpha_t}  - \frac{1}{\alpha_{t-1}} \bigg] \mathbb{E} \| e_t^f \|^2 +  \frac{2 (\eta_{t+1}^f )^2}{\alpha_t} \sigma_f^2  +  4  L_K^2 \alpha_t \mathbb{E}\|h_t^f\|^2 \nonumber\\
 &  \qquad \qquad \qquad  \qquad \qquad \qquad   + \frac{8 L_K^2 \beta_t^2}{\alpha_t} \mathbb{E}\|e_t^g\|^2 + \frac{8   L_K^2 \beta_t^2 }{\alpha_t} \mathbb{E}\|\nabla_y g(x_t, y_t)\|^2,
 \end{align}
 where we have utilized the fact that $0 < 1 - \eta_t < 1$ for all $t \in \{0, 1, \ldots, T-1 \}$. Now we consider the coefficient of the first term on the right hand side of \eqref{Eq: Grad_error_OuterFn_PotFn}, we have
 \begin{align}
 \label{Eq: 1Coeff_GradError_Descent}
  \frac{(1 - \eta_{t+1}^f)^2}{\alpha_t}  - \frac{1}{\alpha_{t-1}}  \leq \frac{1}{\alpha_t} -   \frac{\eta_{t+1}^f}{\alpha_t}  - \frac{1}{\alpha_{t-1}}.
 \end{align}
 Using the definition of $\alpha_t$ we have 
 \begin{align*}
 	\frac{1}{\alpha_t} - \frac{1}{\alpha_{t-1}} & =  {(w + t)^{1/3}}  -  {(w + t - 1)^{1/3}} ] \overset{(a)}{\leq} \frac{1}{3   (w + t - 1)^{2/3}} \overset{(b)}{\leq}  \frac{1}{3   (w/2 + t )^{2/3}} \\
 	& =  \frac{2^{2/3}}{3   (w + 2t )^{2/3}} \leq  \frac{2^{2/3}}{3   (w + t )^{2/3}} \overset{(c)}{\leq}  \frac{2^{2/3}}{3  } \alpha_t^2	\overset{(d)}{\leq}  \frac{\alpha_t}{3   L_f},
 \end{align*}
 where $(a)$ follows from $(x + y)^{1/3} - x^{1/3} \leq y/(3 x^{2/3})$; $(b)$ results from the fact that we choose $w \geq 2$ hence $1 \leq w/2$; $(c)$ results from the definition of $\alpha_t$ and $(d)$ uses the fact that we choose $\alpha_t \leq 1/3L_f$.  Substituting in \eqref{Eq: 1Coeff_GradError_Descent} and using $\eta_{t+1}^f = c_{\eta_f} \alpha_t^2$, we get
 \begin{align*}
  \frac{(1 - \eta_{t+1}^f)^2}{\alpha_t}  - \frac{1}{\alpha_{t-1}}  & \leq    \frac{\alpha_t}{3   L_f}  - c_{\eta_f} \alpha_t   \leq -\bar{c}_{\eta_f} \alpha_t,
 \end{align*}
 which follows from the choice  
 $$ c_{\eta_f}  = \frac{1}{3   L_f} + \bar{c}_{\eta_f}\quad  \text{with} \quad \bar{c}_{\eta_f} = \max \bigg\{ 36 L_K^2 , \frac{4 L_K^2  L_{\mu_g} (\mu_g + L_g) c_\beta^2}{L^2} \bigg\} .$$
 
 Substiuting in \eqref{Eq: Grad_error_OuterFn_PotFn}
 \begin{align}
 \label{Eq: Grad_error_OuterFn_PotFn_Main}
 \frac{1}{\bar{c}_{\eta_f}}  \mathbb{E} \bigg[ \frac{ \|e_{t+1}^f\|^2}{\alpha_t} -   \frac{ \|e_{t+1}^f\|^2}{\alpha_{t-1}} \bigg] &   \leq  - \alpha_t \mathbb{E} \| e_t^f \|^2 +  \frac{2 (\eta_{t+1}^f )^2}{\bar{c}_{\eta_f}  \alpha_t} \sigma_f^2  +  \frac{ \alpha_t}{9} \mathbb{E}\|h_t^f\|^2    + \frac{2 L^2}{L_{\mu_g} (\mu_g + L_g)} \alpha_t \mathbb{E}\|e_t^g\|^2 \nonumber\\
 & \qquad \qquad \qquad \qquad \qquad \qquad \qquad  \qquad +\frac{2 L^2}{L_{\mu_g} (\mu_g + L_g)} \alpha_t \mathbb{E}\|\nabla_y g(x_t, y_t)\|^2,
 \end{align}
 Next, from Lemma \ref{lem: Descent_Grad_Error_Inner}, we have
 \begin{align}
 \label{Eq: Grad_error_InnerFn_PotFn}
 \frac{	\mathbb{E} \|e_{t+1}^g\|^2}{\alpha_t} -\frac{	\mathbb{E} \|e_{t}^g\|^2}{\alpha_{t-1}}	& \leq \bigg[  \frac{(1 - \eta_{t+1}^g)^2 + 8 (1 - \eta_{t+1}^g)^2 L_g^2 \beta_t^2}{\alpha_t}
 - \frac{1}{\alpha_{t-1}}\bigg] \mathbb{E}\| e_t^g \|^2  + \frac{2 (\eta_{t+1}^g)^2}{\alpha_t} \sigma_g^2  \nonumber \\
 &  \qquad \qquad \qquad\qquad \qquad \qquad \quad   + 4   L_g^2 \alpha_t \mathbb{E} \|  h_t^f \|^2     +  \frac{8  L_g^2 \beta_t^2}{\alpha_t} \mathbb{E} \|   \nabla_y g(x_t, y_t)   \|^2
 \end{align} 
 where we have utilized the fact that $0 < 1 - \eta_t^g \leq 1$ for all $t \in \{0,1, \ldots, T-1\}$. Let us consider the coefficient of the first term on the right hand side of \eqref{Eq: Grad_error_InnerFn_PotFn} we have
 \begin{align*}
 \frac{(1 - \eta_{t+1}^g)^2 + 8 (1 - \eta_{t+1}^g)^2 L_g^2 \beta_t^2}{\alpha_t}
 - \frac{1}{\alpha_{t-1}}  &\leq   \frac{(1 - \eta_{t+1}^g) }{\alpha_t}   \big(1+ 8   L_g^2 \beta_t^2 \big) - \frac{1}{\alpha_{t-1}} \\
 & = \frac{1}{\alpha_t} -   \frac{1}{\alpha_{t-1}} + \frac{8 L_g^2 \beta_t^2}{\alpha_t} - c_{\eta_g} \alpha_t (1 + 8 L_g^2 \beta_t^2),
 \end{align*}
 using the fact that from earlier we have $\frac{1}{\alpha_t} - \frac{1}{\alpha_{t-1}} \leq \frac{\alpha_t}{3 L_f}$ and the definition of $\beta_t = c_\beta \alpha_t$, we have
 \begin{align*}
 	\frac{(1 - \eta_{t+1}^g)^2 + 8 (1 - \eta_{t+1}^g)^2 L_g^2 \beta_t^2}{\alpha_t}
 	- \frac{1}{\alpha_{t-1}} \leq  \frac{\alpha_t}{3  L_f} + 8 L_g^2 c_\beta^2 \alpha_t - c_{\eta_g} \alpha_t,
 \end{align*}
 Next choosing $c_{\eta_g}$ as
 \begin{align*}
 c_{\eta_g} = \frac{1}{3  L_f} + 8 L_g^2 c_\beta^2 + \bigg[ \frac{8 L^2}{L_{\mu_g}^2} + \frac{2L^2}{L_{\mu_g} (\mu_g + L_g)} \bigg] \bar{c}_{\eta_g} \quad \text{with} \quad \bar{c}_{\eta_g}  = \max\bigg\{    36 L_g^2 , \frac{4 L_g^2 L_{\mu_g} (\mu_g + L_g) c_\beta^2}{L^2} \bigg\}.
 \end{align*}
 Therefore, we get
  \begin{align*}
 \frac{(1 - \eta_{t+1}^g)^2 + 8 (1 - \eta_{t+1}^g)^2 L_g^2 \beta_t^2}{\alpha_t}
 - \frac{1}{\alpha_{t-1}} \leq  - \bigg[  \frac{8 L^2}{L_{\mu_g}^2} + \frac{2L^2}{L_{\mu_g} (\mu_g + L_g)}  \bigg]  \bar{c}_{\eta_g}  \alpha_t,
 \end{align*}
 Finally, replacing in \eqref{Eq: Grad_error_InnerFn_PotFn} we get
 \begin{align}
 \label{Eq: Grad_error_InnerFn_PotFn_Main}
 \frac{1}{\bar{c}_{\eta_g}} 	\mathbb{E}\bigg[ \frac{ \|e_{t+1}^g\|^2}{\alpha_t} -\frac{  \|e_{t}^g\|^2}{\alpha_{t-1}} \bigg]	& \leq  - \bigg[  \frac{8 L^2}{L_{\mu_g}^2} + \frac{2L^2}{L_{\mu_g} (\mu_g + L_g)}  \bigg] \alpha_t \mathbb{E}\| e_t^g \|^2  + \frac{2 (\eta_{t+1}^g)^2}{\bar{c}_{\eta_g} \alpha_t} \sigma_g^2  \nonumber \\
  &  \qquad \qquad \qquad\qquad \qquad \quad \quad   + \frac{\alpha_t}{9} \mathbb{E} \|  h_t^f \|^2     +  \frac{2L^2}{ L_{\mu_g} (\mu_g + L_g)} \alpha_{t} \mathbb{E} \|   \nabla_y g(x_t, y_t)   \|^2
 \end{align}
 Finally, adding  \eqref{Eq: InnerIterates_PotFn_Main}, \eqref{Eq: Grad_error_OuterFn_PotFn_Main}, \eqref{Eq: Grad_error_InnerFn_PotFn_Main}  and the result of Lemma \ref{Lem: Smoothness_l_Main} with $\alpha_t \leq 1/3L_f$, we get
 	\begin{align*}
 \mathbb{E} [ V_{t+1} - V_t ]  \leq - \frac{\alpha_{t}}{2} \mathbb{E} \| \nabla \ell(x_t)\|^2   +  2 \alpha_{t} \|B_t\|^2  +  \frac{2   (\eta_{t+1}^f)^2}{\bar{c}_{\eta_f} \alpha_t} \sigma_f^2 +  \frac{2   (\eta_{t+1}^g)^2}{\bar{c}_{\eta_g} \alpha_t} \sigma_g^2.
 \end{align*}
 Therefore, we have the statement of the Lemma. 
 
 \subsection{Proof of Theorem \ref{Thm: Convergence_NC}}
 Summing the result of Lemma \ref{Lem: Decent_PotentialFn} for $t = 0$ to $T-1$, dividing by $T$ on both sides and using the definition $\eta_{t+1}^f \coloneqq c_{\eta_f} \alpha_t^2$ and $\eta_{t+1}^g \coloneqq c_{\eta_g} \alpha_t^2$  we get
 \begin{align}
 \label{Eq: PotFn_Summation_Alpha}
	\frac{\mathbb{E} [ V_{T} - V_0 ] }{T} \leq - \frac{1}{T}  \sum_{t = 0}^{T - 1} \frac{\alpha_{t}}{2} \mathbb{E} \| \nabla \ell(x_t)\|^2   +  \frac{2}{T}  \sum_{t = 0}^T \alpha_{t} \|B_t\|^2  +  \frac{2   c_{\eta_f}^2 \sigma_f^2}{\bar{c}_{\eta_f} } \sum_{t = 0}^{T - 1} \alpha_t^3+  \frac{2   c_{\eta_g}^2 \sigma_g^2}{\bar{c}_{\eta_g} }  \sum_{t = 0}^{T - 1}  \alpha_t^3.
 \end{align}
 Next considering $\sum_{t = 0}^{T - 1} \alpha_t$ in the last two terms on the right hand side of \eqref{Eq: PotFn_Summation_Alpha}, we have from the definition of $\alpha_t$ that
 \begin{align*}
\sum_{t = 0}^{T - 1} \alpha_t^3 &= \sum_{t = 0}^{T - 1}  \frac{1}{w + t}  \overset{(a)}{\leq}  \sum_{t = 0}^{T - 1}  \frac{1}{1 + t} \leq   \log (T + 1)
 \end{align*}
 where inequality $(a)$ results from the fact that we choose $w \geq 1$. Substituting the above in \eqref{Eq: PotFn_Summation_Alpha} we get
 \begin{align*}
 \frac{\mathbb{E} [ V_{T} - V_0 ] }{T} \leq - \frac{1}{T}  \sum_{t = 0}^{T - 1} \frac{\alpha_{t}}{2} \mathbb{E} \| \nabla \ell(x_t)\|^2   +  \frac{2}{T}  \sum_{t = 0}^T \alpha_{t} \|B_t\|^2  +  \frac{2   c_{\eta_f}^2  }{\bar{c}_{\eta_f} } \frac{\log (T + 1)}{T}  \sigma_f^2   +  \frac{2   c_{\eta_g}^2 }{\bar{c}_{\eta_g} }   \frac{\log (T + 1)}{T} \sigma_g^2
 \end{align*}
 Rearranging the terms we get
  \begin{align*}
  \frac{1}{T}  \sum_{t = 0}^{T - 1} \frac{\alpha_{t}}{2} \mathbb{E} \| \nabla \ell(x_t)\|^2\leq   \frac{\mathbb{E} [ V_0 - \ell^\ast ] }{T}   +  \frac{2}{T}  \sum_{t = 0}^T \alpha_{t} \|B_t\|^2  +  \frac{2   c_{\eta_f}^2 }{\bar{c}_{\eta_f} }  \frac{\log (T + 1)}{T}  \sigma_f^2   +  \frac{2   c_{\eta_g}^2   }{\bar{c}_{\eta_g} }  \frac{\log (T + 1)}{T}  \sigma_g^2
 \end{align*}
 Using the fact that $\alpha_t$ is decreasing in $t$ we have $\alpha_T \leq \alpha_t$ for all $t \in \{0,1, \ldots, T - 1\}$ and multiplying by $2/\alpha_T$ on both sides we get
 \begin{align*}
 \frac{1}{T}  \sum_{t = 0}^{T - 1}   \mathbb{E} \| \nabla \ell(x_t)\|^2\leq   \frac{2 \mathbb{E} [  V_0 - \ell^\ast ] }{\alpha_T T}   +  \frac{4}{\alpha_T T}  \sum_{t = 0}^T \alpha_{t} \|B_t\|^2  +  \frac{4   c_{\eta_f}^2   }{\bar{c}_{\eta_f} }  \frac{\log (T + 1)}{\alpha_T T}  \sigma_f^2   +  \frac{4 c_{\eta_g}^2  }{\bar{c}_{\eta_g} }  \frac{\log (T + 1)}{ \alpha_T T}  \sigma_g^2
 \end{align*}
 Finally, we have from the definition of the Potential function
 \begin{align*}
 \mathbb{E}[V_0] & \coloneqq \mathbb{E}\bigg[ \ell(x_0) + \frac{2L}{3\sqrt{2}L_y } \| y_0 - y^\ast(x_0)\|^2  + \frac{1}{\bar{c}_{\eta_f}} \frac{\| e_0^f\|^2}{\alpha_{-1}}  + \frac{1}{\bar{c}_{\eta_g}} \frac{\| e_0^g\|^2}{\alpha_{-1}} \bigg] \\
 & \leq \ell(x_0) + \frac{2L}{3\sqrt{2}L_y } \| y_0 - y^\ast(x_0)\|^2  +   \frac{\sigma_f^2}{\bar{c}_{\eta_f} \alpha_{-1}}  +  \frac{\sigma_g^2}{\bar{c}_{\eta_g} \alpha_{-1}} ,
 \end{align*}
 which follows from the assumption and the definition of $h_t^f$ and $h_t^g$. Therefore, we have
  \begin{align*}
 \frac{1}{T}  \sum_{t = 0}^{T - 1}  \| \nabla \ell(x_t)\|^2 & \leq   \frac{2   (\ell(x_0) - \ell^\ast )}{\alpha_T T}   + \frac{4L}{3\sqrt{2}L_y }   \frac{\| y_0 - y^\ast(x_0)\|^2}{\alpha_T T}  +   \frac{2}{\bar{c}_{\eta_f} \alpha_{-1}}   \frac{\sigma_f^2}{\alpha_T T} +  \frac{2}{\bar{c}_{\eta_g} \alpha_{-1}} \frac{\sigma_g^2}{\alpha_T T} \\ 
 & + \frac{4}{\alpha_T T}  \sum_{t = 0}^T \alpha_{t} \|B_t\|^2  +  \frac{4   c_{\eta_f}^2  }{\bar{c}_{\eta_f} }  \frac{\log (T + 1)}{\alpha_T T}  \sigma_f^2   +  \frac{4 c_{\eta_g}^2 }{\bar{c}_{\eta_g} }  \frac{\log (T + 1)}{ \alpha_T T}  \sigma_g^2
 \end{align*}
 Finally, we have from the definition of $\alpha_T \coloneqq 1/(w + T)^{1/3}$ and $\alpha_1 = \alpha_0$, moreover using the fact that for the choice of $K = (L_g/\mu_g)\log (C_{g_{xy}} C_{f_y} T / \mu_g)$ stochastic Hessian samples of $\nabla^2_{yy} g(x, y)$ we have $\|B_t\| = 1/T$, we get
 {
 \begin{align*}
  	\mathbb{E} \| \nabla \ell(x_a(T))\|^2  \leq \mathcal{O} \bigg( \frac{\ell(x_0) - \ell^\ast}{T^{2/3}} \bigg) + {\mathcal{O}} \bigg( \frac{\|y_0 - y^\ast(x_0) \|^2}{T^{2/3}} \bigg) + \tilde{\mathcal{O}} \bigg( \frac{\sigma_f^2}{T^{2/3}} \bigg) + \tilde{\mathcal{O}} \bigg( \frac{\sigma_g^2}{T^{2/3}} \bigg) .
 \end{align*}}
 Hence, the theorem is proved. 
	\end{proof}

		\section{Proof of Theorem \ref{Thm: Convergence_SC_Fixed}: strongly-convex outer objective}
	\label{Sec: Appendix_Thm_SC}

	To prove Theorem \ref{Thm: Convergence_SC_Fixed}, we utilize the descent results obtained for the proof of Theorem \ref{Thm: Convergence_NC} in Appendix \ref{Sec: Appendix_Thm_NC}. The proof follows similar structure as the proof of non-convex case. We first consider the descent achieved by the consecutive iterates generated by Algorithm \ref{Algo: Accelerated_STSA} when the outer function is strongly-convex and smooth.
	
	\subsection{Descent in the function value} 
	\begin{lem}
		\label{Lem: Smoothness_l_main_SC}
		For strongly-convex and smooth $\ell(\cdot)$, with $e_t^f$ defined as: $e_t^f \coloneqq h_t^f - \bar{\nabla} f(x_t,y_{t+1}) - B_t$, the consecutive iterates of Algorithm \ref{Algo: Accelerated_STSA} satisfy:
		\begin{align*}
		\mathbb{E} [ \ell(x_{t+1}) - \ell^\ast] & \leq \mathbb{E} \Big[ (1 - \alpha_t \mu_f) \big(\ell(x_t) - \ell^\ast \big) - \frac{\alpha_{t}}{2} (1 - \alpha_{t} L_f) \|h_t^f\|^2 + \alpha_{t} \|e_t^f\|^2 \\
		&  \qquad \qquad \qquad \qquad\qquad \qquad\qquad \qquad + 2 \alpha_{t} L^2 \|y_{t} -   y^\ast(x_t)\|^2 + 2 \alpha_{t} \|B_t \|^2 \Big],
		\end{align*}
		for all $t \in \{0,1,\ldots, T-1\}$, where the expectation is w.r.t. the stochasticity of the algorithm. 
	\end{lem}
	\begin{proof}
		Note that from Lemma \ref{Lem: Smoothness_l_Main} derived in Appendix \ref{Sec: Appendix_Thm_NC}, we have
		\begin{align}
		\label{Eq: LemmaC1}
			\mathbb{E} [ \ell(x_{t+1}) ] & \leq \mathbb{E} \Big[ \ell(x_t) - \frac{\alpha_{t}}{2} \|\nabla \ell(x_t)\|^2 - \frac{\alpha_{t}}{2} (1 - \alpha_t L_f) \|h_t^f\|^2   + \alpha_{t} \|e_t^f\|^2 \\
			& \quad \qquad \qquad\qquad \qquad \qquad \qquad + 2 \alpha_{t} L^2 \|y_{t} -   y^\ast(x_t)\|^2 + 2 \alpha_{t} \|B_t \|^2 \Big]. \nonumber
		\end{align}
		Now using the fact that for a strongly convex function we have:
		\begin{align*}
		\|\nabla \ell(x)\|^2 \geq 2 \mu_f (\ell(x) - \ell^\ast)\quad \text{for all} \quad x \in \mathbb{R}^{\du}, 
		\end{align*}
		substituting in \eqref{Eq: LemmaC1}, subtracting $\ell^\ast$ from both sides and rearranging the terms yields the statement of the Lemma. 
	\end{proof}

		\subsection{Descent in the iterates of the lower level problem} 
{
\begin{lem}
	\label{Lem: InnerIterates_Descent_SC}
		The iterates of the inner problem generated according to Algorithm \ref{Algo: Accelerated_STSA}, satisfy
	\begin{align*}
	\mathbb{E}\|y_{t+1} - y^\ast(x_{t + 1})\|^2 & \leq (1 + \gamma_t)   \big(1 -  2 \beta_t \mu_g + \beta_t^2 L_g^2 \big)  \mathbb{E} \|y_t - y^\ast(x_{t})\|^2 \\
	& \qquad \qquad \qquad \qquad \qquad \quad + \bigg( 1 + \frac{1}{\gamma_t} \bigg) L_y^2 \alpha_{t}^2 \mathbb{E}   \| h_{t}^f \|^2  +
(1 + \gamma_t)  \beta_t^2   \sigma_g^2.
	\end{align*}
	for all $t \in \{0,\ldots, T - 1\}$ with some $\gamma_t > 0$, where the expectation is w.r.t. the stochasticity of the algorithm. 
\end{lem}
\begin{proof}
    Consider the term $\mathbb{E} \| y_{t + 1} - y^\ast(x_{t+1})\|^2$, we have
    \begin{align}
     \mathbb{E} \| y_{t + 1} - y^\ast(x_{t+1})\|^2 & \overset{(a)}{\leq} (1 + \gamma_t) \mathbb{E} \| y_{t + 1} - y^\ast(x_{t})\|^2  + \bigg( 1 + \frac{1}{\gamma_t} \bigg) \mathbb{E} \|y^\ast(x_{t+1}) - y^\ast(x_{t}) \|^2 \nonumber\\
     & \overset{(b)}{\leq} (1 + \gamma_t) \mathbb{E} \| y_{t} - \beta_t h_t^g - y^\ast(x_{t})\|^2  + \bigg( 1 + \frac{1}{\gamma_t} \bigg)  L_y^2 \mathbb{E}\|x_{t+1} - x_{t} \|^2 \nonumber\\
     & \overset{(c)}{\leq} (1 + \gamma_t) \mathbb{E} \| y_{t} - \beta_t h_t^g - y^\ast(x_{t})\|^2  + \bigg( 1 + \frac{1}{\gamma_t} \bigg)  L_y^2 \alpha_t^2 \mathbb{E} \|h_t^f \|^2 
     \label{Eq: SC_InnerIterate_1st}
    \end{align}
    where $(a)$ results from Young's inequality; $(b)$ uses Step 5 of Algorithm \ref{Algo: Accelerated_STSA} and Lipschitzness of $y^\ast(\cdot)$ given in Lemma \ref{Lem: Lip_Ghadhimi}; and $(c)$ uses Step 7 of Algorithm \ref{Algo: Accelerated_STSA}.
    
    Next, we consider the first term of \eqref{Eq: SC_InnerIterate_1st} above:
    \begin{align}
    \mathbb{E} \| y_{t} - \beta_t h_t^g - y^\ast(x_{t})\|^2 & =  \mathbb{E} \| y_{t}   - y^\ast(x_{t})\|^2 + \beta_t^2 \mathbb{E} \|   h_t^g \|^2 - \beta_t \mathbb{E} \langle y_t - y^\ast (x_t) ,  h_t^g \rangle \nonumber \\
  & \overset{(a)}{\leq}  \mathbb{E} \| y_{t}   - y^\ast(x_{t})\|^2 + \beta_t^2 \mathbb{E} \|  \nabla_y g(x_t, y_t) \|^2 + \beta_t^2 \mathbb{E} \|   h_t^g - \nabla_y g(x_t, y_t)\|^2  \nonumber\\
  & \qquad \qquad \qquad \qquad \qquad \qquad \qquad \qquad   - \beta_t \mathbb{E} \langle y_t - y^\ast (x_t) , \nabla_y g(x_t, y_t)\rangle  \nonumber\\
    & \overset{(b)}{\leq} (1- 2 \mu_g \beta_t + \beta_t^2 L_g^2 ) \mathbb{E} \| y_{t}   - y^\ast(x_{t})\|^2 +   + \beta_t^2 \sigma_g^2  
          \label{Eq: SC_InnerIterate_2nd}
    \end{align}
    where $(a)$ utilizes the fact that for $\eta_t^g = 1$ we have $\mathbb{E}[h_t^g | \mathcal{F}_t ] = \nabla_y g(x_t,y_t)$ and $(b)$ uses the fact that (1) $\nabla_y g(x , y^\ast(x)) = 0$ and the Lipschitzness of $\nabla_y g(x, \cdot)$ in Assumption \ref{Assump: InnerFunction}-(ii); (2) Assumption \ref{Assump: Stochastic Grad}-(ii); and (3) $g(x,y)$ is $\mu_g$-strongly convex w.r.t. $y$, we therefore have
 	$$\big\langle \nabla g_y(x , y_1) - \nabla g_y(x, y_2), y_1 - y_2 \big\rangle \geq \mu_g \|y_1 - y_2 \|^2,$$
    using $y_1 = y_t$ and $y_2 = y^\ast(x_t)$ yields inequality $(b)$. Finally, substituting \eqref{Eq: SC_InnerIterate_2nd} in \eqref{Eq: SC_InnerIterate_1st} yields the statement of the lemma. 
\end{proof}	

\subsection{Descent in the gradient estimation error}
\begin{lem}
		\label{Lem: SC_Descent_grad_error}
		Define $e_t^f \coloneqq h_t^f - \bar{\nabla}f(x_t , y_t) - B_t$. Then the consecutive iterates of Algorithm \ref{Algo: Accelerated_STSA} satisfy:
		\begin{align*}
		\mathbb{E} \|e_{t+1}^f\|^2 &  \leq (1 - \eta_{t+1}^f)^2 \mathbb{E} \| e_t^f \|^2 +  2 (\eta_{t+1}^f )^2 \sigma_f^2  +  4 (1 - \eta_{t+1}^f)^2 L_K^2 \alpha_t^2 \mathbb{E}\|h_t^f\|^2 \\
		&  \qquad \qquad \qquad   + 8 (1 - \eta_{t+1}^f)^2 L_K^2 \beta_t^2 \sigma_g^2 + 8 (1 - \eta_{t+1}^f)^2 L_K^2 L_g^2 \beta_t^2 \mathbb{E}\| y_t - y^\ast(x_t)\|^2,
		\end{align*}
		for all $t \in \{0,\ldots, T - 1\}$, with $L_K$ defined in the statement of Lemma \ref{Lem: Lip_GradEst_Appendix}. Here the expectation is taken w.r.t the stochasticity of the algorithm. 	
	\end{lem}
\begin{proof}
From the statement of Lemma \ref{Lem: Descent_grad_error}, we have
 \begin{align*}
		\mathbb{E} \|e_{t+1}^f\|^2 &  \leq (1 - \eta_{t+1}^f)^2 \mathbb{E} \| e_t^f \|^2 +  2 (\eta_{t+1}^f )^2 \sigma_f^2  +  4 (1 - \eta_{t+1}^f)^2 L_K^2 \alpha_t^2 \mathbb{E}\|h_t^f\|^2 \\
		&  \qquad \qquad \qquad   + 8 (1 - \eta_{t+1}^f)^2 L_K^2 \beta_t^2 \mathbb{E}\|e_t^g\|^2 + 8 (1 - \eta_{t+1}^f)^2 L_K^2 \beta_t^2 \mathbb{E}\|\nabla_y g(x_t, y_t)\|^2,
		\end{align*}
		The proof follows by noticing the fact that for the gradient estimate $h_t^g$ with $\eta_t^g = 1$, we have $\mathbb{E}\|e_t^g\|^2 \leq \sigma_g^2$ from Assumption \ref{Assump: Stochastic Grad}-(ii) and the Lipschitzness of $\nabla_y g(x, \cdot)$ combined with the fact that $\nabla_y g(x, y^\ast(x)) = 0$.
\end{proof}

	\subsection{Descent in potential function}
	In this section, we define the potential function as:
	\begin{align}
	\label{Eq: PotentialFunction_SC}
	\widehat{V}_{t} \coloneqq (\ell(x_t) - \ell^\ast ) +  \|e_t^f\|^2 +  \| y_{t} - y^\ast(x_t)\|^2,
	\end{align}
	which is different from that of \eqref{Eq: PotentialFunction}. 
	We next show that the potential function decreases with appropriate choice of parameters.  
	\begin{lem}
		\label{Lem: SC-SC_PotentialFunction_Descent}
		With the potential function, $\widehat{V}_t$, defined in \eqref{Eq: PotentialFunction_SC}, with the choice of parameters
		\begin{align*}
		\eta_{t + 1}^f = (\mu_f + 1) \alpha_{t}, ~ \beta_{t} = \hat{c}_\beta \alpha_{t} ~\text{with} ~	\hat{c}_\beta =   \frac{8 L_y^2  + 8 L^2 + 2 \mu_f}{\mu_g}  ~\text{and} ~ \gamma_{t} = \frac{ \mu_g \beta_{t}}{2 (1 -  \mu_g \beta_{t})}~ \text{for all}~t \in \{0,1, \ldots, T-1\},
		\end{align*}
		with $\alpha_{-1} = \alpha_{0}$, moreover, we choose
		\begin{align}\label{eq:alpha_cvx_app}
&		\alpha_t \leq \bigg\{  \frac{1}{\mu_f + 1}, \frac{1}{2 \mu_g \hat{c}_\beta} ,\frac{\mu_g}{ \hat{c}_\beta L_g^2 }, \frac{1}{8L_K^2 + L_f},  \frac{L^2 + 2 L_y^2}{4 L_K^2 L_g^2 \hat{c}_\beta^2} \bigg\}.
% &	{\red	\alpha_t < \min \bigg\{   \frac{1}{\mu_f + 1}, \frac{1}{\bar{c}_\beta}, \frac{1}{\mu_f + 4 L^2}, \frac{2L^2}{\bar{L}_g^2 L_K^2 \bar{c}_\beta^2}, \frac{1}{32 L_K^2 + 2 L_f} \bigg\} } \nonumber
		\end{align}
		Further, we choose 
		$$K = \frac{L_g}{2\mu_g} \log \bigg( \bigg(\frac{ C_{g_{xy}} C_{f_y} }{ \mu_g} \bigg)^2 T \bigg)$$ 
		such that we have $\|B_t\|^2 \leq 1/T$, then we have
		\begin{align*}
	\mathbb{E}[\widehat{V}_{t + 1}] \leq (1 - \mu_f \alpha_{t+1}) \mathbb{E} [\widehat{V}_t] + \frac{2 \alpha_{t}}{T} + \big[ (2 \hat{c}_\beta^2 +  8 \hat{c}_\beta^2 L_K^2)  \sigma_g^2 + 2 (\mu_f+1)^2  \sigma_f^2 \big]\alpha_t^2,
		\end{align*}
		for all $t \in \{0,1, \ldots, T-1\}$.
	\end{lem}
	\begin{proof}
		From Lemma \ref{Lem: SC_Descent_grad_error}, we have
		\begin{align}
			\mathbb{E} \|e_{t+1}^f\|^2 &  \leq (1 - \eta_{t+1}^f) \mathbb{E} \| e_t^f \|^2 +  2 (\eta_{t+1}^f )^2 \sigma_f^2  +  4   L_K^2 \alpha_t^2 \mathbb{E}\|h_t^f\|^2   + 8   L_K^2 \beta_t^2 \sigma_g^2 + 8   L_K^2 L_g^2 \beta_t^2 \mathbb{E}\| y_t - y^\ast(x_t)\|^2,
		\label{Eq: Descent_grad_error_1st_SC}
		\end{align}
		which follows from $1 - \eta_{t+1}^f \leq 1$. 
With the choice of $\eta_{t} = (\mu_f + 1) \alpha_t$ and $\beta_t = \hat{c}_\beta \alpha_t$ we get from \eqref{Eq: Descent_grad_error_1st_SC}:
		\begin{align}
		\mathbb{E} \|e_{t+1}^f \|^2 &  \leq (1 - (\mu_f + 1) \alpha_{t}) \mathbb{E} \| e_t^f \|^2 +   2 (\mu_f + 1 )^2 \alpha_t^2 \sigma_f^2  +  4   L_K^2 \alpha_t^2 \mathbb{E}\|h_t^f\|^2 \nonumber \\
		&  \qquad \qquad \qquad \qquad\qquad \qquad \qquad + 8   L_K^2 \hat{c}_\beta^2 \alpha_t^2 \sigma_g^2 + 8   L_K^2 L_g^2 \hat{c}_\beta^2 \alpha_t^2 \mathbb{E}\| y_t - y^\ast(x_t)\|^2,
		\label{Eq: Descent_grad_error_SC}
		\end{align}
Next, we consider the descent in the iterates of inner problem. Again using Lemma \ref{Lem: InnerIterates_Descent_SC} we have
\begin{align}
\mathbb{E}\|y_{t+1} - y^\ast(x_{t + 1})\|^2 & \leq (1 + \gamma_t)   \big(1 -  2 \beta_t \mu_g + \beta_t^2 L_g^2 \big)  \mathbb{E} \|y_t - y^\ast(x_{t})\|^2 \label{Eq: InnerIterates_Descent_1st_SC} \\
& \quad + \bigg( 1 + \frac{1}{\gamma_t} \bigg) L_y^2 \alpha_{t}^2 \mathbb{E}   \| h_{t}^f \|^2  +
(1 + \gamma_t)  \beta_t^2   \sigma_g^2. \nonumber
		\end{align}
	Using the fact that $\beta_t \leq \frac{\mu_g}{L_g^2}$, $\beta_t \leq \frac{1}{2 \mu_g}$ and from the choice of $\gamma_t$ we have
		$1 + \frac{1}{\gamma_{t}} \leq \frac{2}{\mu_g \beta_t}$
		Substituting the $\gamma_{t}$, $\beta_{t}$ and the upper bound on $1 + \frac{1}{\gamma_{t}}$ in \eqref{Eq: InnerIterates_Descent_1st_SC} above we get:
		\begin{align}
		\mathbb{E}\|y_{t+1} - y^\ast(x_{t+1})\|^2 \leq  \bigg(1 -  \frac{\hat{c}_\beta \mu_g \alpha_t}{2} \bigg)  \mathbb{E} \|y_{t} - y^\ast(x_t)\|^2 + \frac{2 L_y^2 \alpha_{t}}{ \mu_g \hat{c}_\beta} \mathbb{E}   \| h_t^f \|^2 + 2 \hat{c}_\beta^2 \alpha_{t}^2 \sigma_g^2.
		\label{Eq: InnerIterates_Descent_SC}
		\end{align}
		Next, replacing the choice of $\hat{c}_\beta$ in \eqref{Eq: InnerIterates_Descent_SC}, we get:
		\begin{align}
		    	\mathbb{E}\|y_{t+1} - y^\ast(x_{t+1})\|^2 \leq  \big(1 -  [ 4L_y^2 + 4 L^2 + \mu_f] \alpha_t \big)  \mathbb{E} \|y_{t} - y^\ast(x_t)\|^2 + \frac{ \alpha_{t}}{4} \mathbb{E}   \| h_t^f \|^2 + 2 \hat{c}_\beta^2 \alpha_{t}^2 \sigma_g^2.
		    \label{Eq: InnerIterates_Descent_SC_Main}
		\end{align}
Finally, to construct the potential function defined in \eqref{Eq: PotentialFunction_SC} we add \eqref{Eq: Descent_grad_error_SC} and \eqref{Eq: InnerIterates_Descent_SC_Main} to the expression of Lemma \ref{Lem: Smoothness_l_main_SC}, we get
		\begin{align*}
		\mathbb{E} [\widehat{V}_{t + 1}] & \leq (1 - \mu_f \alpha_t) \mathbb{E} [\widehat{V}_{t + 1}]  - \bigg( \frac{\alpha}{2} (1 - \alpha_t L_f)  - \frac{\alpha_t}{4} -  4 L_K^2 \alpha_t^2 \bigg)  \mathbb{E}\|h_t^f \|^2 + 2 \alpha_t \| B_t\|^2\\
	& \qquad -  \big( 4L^2 \alpha_t + 4L_y^2 \alpha_t - 2 L^2 \alpha_t - 8 L_K^2 L_g^2 \hat{c}_\beta^2 \alpha_t^2 \big) \mathbb{E}\|y_t - y^\ast(x_t) \|^2 \\
	& \qquad + (2 \hat{c}_\beta^2 + 8 \hat{c}_\beta^2 L_K^2) \alpha_t^2 \sigma_g^2 + 2 (\mu_f+1)^2 \alpha_t^2 \sigma_f^2.
		\end{align*}
		Noting the fact that $\alpha_t \leq \frac{1}{8L_K^2 + L_f}$ and $\alpha_t \leq \frac{L^2 + 2 L_y^2}{4 L_K^2 L_g^2 \hat{c}_\beta^2}$ and choosing $B_t$ such that we have $\|B_t\|^2 \leq \frac{1}{T}$, we get
		\begin{align*}
		\mathbb{E}[\widehat{V}_{t + 1}] \leq (1 - \mu_f \alpha_{t}) \mathbb{E} [\widehat{V}_t] + \frac{2 \alpha_{t}}{T} + \big[ (2 \hat{c}_\beta^2 +  8 \hat{c}_\beta^2 L_K^2)  \sigma_g^2 + 2 (\mu_f+1)^2  \sigma_f^2 \big]\alpha_t^2.
		\end{align*}
		This concludes the proof of the lemma.
	\end{proof}
	
	\subsection{Proof of Theorem \ref{Thm: Convergence_SC_Fixed}} 
	Next, we conclude the proof for the case of strongly-convex outer objective function case based on fixed step sizes and momentum parameters.
	\begin{proof}
		With fixed step sizes, i.e. $\alpha_{t} = \alpha$ for all $t \in \{0,1, \ldots, T-1\}$, we have from the Lemma \ref{Lem: SC-SC_PotentialFunction_Descent}
		\begin{align*}
		\mathbb{E}[\widehat{V}_{t + 1}] \leq (1 - \mu_f \alpha) \mathbb{E} [\widehat{V}_t] + \frac{2 \alpha}{T} + \big[ (2 \hat{c}_\beta^2 +  8 \hat{c}_\beta^2 L_K^2)  \sigma_g^2 + 2 (\mu_f+1)^2  \sigma_f^2 \big]\alpha^2.
		\end{align*}
		applying the above inequality recursively we get
		\begin{align}
		\mathbb{E}[\widehat{V}_{t}] & \leq (1 - \mu_f \alpha)^t \mathbb{E} [\widehat{V}_0] + \frac{2 \alpha}{T} \sum_{k = 0}^{t-1}(1 - \mu_f \alpha)^k  + \big[ (2 \hat{c}_\beta^2 +  8 \hat{c}_\beta^2 L_K^2)  \sigma_g^2 + 2 (\mu_f+1)^2  \sigma_f^2 \big] \alpha^2 \sum_{k = 0}^{t-1}(1 - \mu_f \alpha)^k  \nonumber\\
		& \overset{(a)}{\leq} (1 - \mu_f \alpha)^t \big((\ell(x_0) - \ell^\ast ) +  \mathbb{E} \|e_0^f\|^2 +  \mathbb{E} \| y_{0} - y^\ast(x_0)\|^2\big) + \frac{2 \alpha}{T} \sum_{k = 0}^{t-1}(1 - \mu_f \alpha)^k \nonumber\\
		& \qquad \qquad \qquad \qquad \qquad \qquad \qquad \qquad \qquad + \big[ (2 \hat{c}_\beta^2 +  8 \hat{c}_\beta^2 L_K^2)  \sigma_g^2 + 2 (\mu_f+1)^2  \sigma_f^2 \big] \alpha^2 \sum_{k = 0}^{t-1}(1 - \mu_f \alpha)^k \nonumber\\
		& \overset{(b)}{\leq} (1 - \mu_f \alpha)^t \big\{ (\ell(x_0) - \ell^\ast) + \sigma_f^2 + \|y_0 - y^\ast(x_0) \|^2   \big\} +  \frac{2}{\mu_f T} + \frac{(2 \hat{c}_\beta^2 +  8 \hat{c}_\beta^2 L_K^2)  \sigma_g^2 + 2 (\mu_f+1)^2  \sigma_f^2 }{\mu_f} \, \alpha,
		\label{Eq: SC_DescentPot_FixedStep}
		\end{align}
		where $(a)$ follows from the definition of $\widehat{V}_t$ given in \eqref{Eq: PotentialFunction_SC} and $(b)$ utilizes the summation of a geometric progression.
		
		This concludes the proof of the theorem.
	\end{proof}
	
	\paragraph{Sample complexity of \aname~in the strongly convex setting} 
	Let us estimate the total number of iterations, $T$, needed to reach an $\epsilon$-optimal solution. 
	First, we select a constant step size such that 
	\begin{equation}
	\label{Eq: alpha_epsilon}
	\alpha \leq \frac{\mu_f}{4  \big[ (2 \hat{c}_\beta^2 +  8 \hat{c}_\beta^2 L_K^2)  \sigma_g^2 + 2 (\mu_f+1)^2  \sigma_f^2 \big] } \epsilon \quad 
	\Longrightarrow \quad 
	\frac{\big[ (2 \hat{c}_\beta^2 +  8 \hat{c}_\beta^2 L_K^2)  \sigma_g^2 + 2 (\mu_f+1)^2  \sigma_f^2 \big] }{\mu_f} \alpha \leq \frac{\epsilon}{4} ,
	\end{equation}
	which controls the last term in \eqref{Eq: SC_DescentPot_FixedStep}.
	Secondly, to control the second term in \eqref{Eq: SC_DescentPot_FixedStep}, we observe that $T \geq \frac{8}{\mu_f \epsilon}$ implies $\frac{2}{\mu_f T} \leq \frac{\epsilon}{4}$. Finally, controlling the first term in \eqref{Eq: SC_DescentPot_FixedStep} requires   
	\begin{align}
	\label{Eq: Complexity_SC}
	\frac{\epsilon}{2} & \ge(1 - \mu_f \alpha)^T  \big((\ell(x_0) - \ell^\ast) + \sigma_f^2 + \|y_0 - y^\ast(x_0) \|^2 \big) 
	\end{align}
	which means we require:
	\begin{align}
	(1 - \mu_f \alpha)^T & \leq \frac{\epsilon}{2 \big((\ell(x_0) - \ell^\ast) + \sigma_f^2 + \|y_0 - y^\ast(x_0) \|^2 \big)} \nonumber\\
	\Longleftrightarrow T \log (1 - \mu_f \alpha) & \leq \log \bigg( \frac{\epsilon}{2 \big((\ell(x_0) - \ell^\ast) + \sigma_f^2 + \|y_0 - y^\ast(x_0) \|^2 \big)} \bigg) \nonumber\\
	\Longleftrightarrow T & \geq  \frac{ \log\bigg( \frac{2 \big((\ell(x_0) - \ell^\ast) + \sigma_f^2 + \|y_0 - y^\ast(x_0) \|^2  \big)}{\epsilon} \bigg)}{-\log(1 - \mu_f \alpha)} \nonumber\\
	\Longleftarrow T & \overset{(a)}{\geq}  \log\bigg( \frac{2 \big((\ell(x_0) - \ell^\ast) + \sigma_f^2 + \|y_0 - y^\ast(x_0) \|^2 \big)}{\epsilon} \bigg) \frac{1}{\mu_f \alpha},
	\end{align}
	where $(a)$ is due to $\log x \leq x - 1$ for all $x > 0$. This along with \eqref{Eq: alpha_epsilon} imply that we require at most $T = \tilde{\cal O}( \epsilon^{-1} )$ iterations to reach an $\epsilon$-optimal solution, i.e., $\mathbb{E} [ \ell(x_t) - \ell^\ast ] \leq \epsilon$. 
	Finally, as each iteration takes a batch of $K = {\cal O}( \log (T) )$ samples, the total sample complexity required to reach an $\epsilon$-optimal solution is bounded as $T = \tilde{\cal O}(\epsilon^{-1})$. 
	\qed
	}

\end{document}